\documentclass[11pt, reqno]{amsart}
\usepackage[letterpaper, margin=1in]{geometry}
\usepackage{amssymb}
\usepackage{mathrsfs}
\usepackage{bbm}
\usepackage{enumitem}
\usepackage{enumerate}
\usepackage{graphicx}
\usepackage{color}
\usepackage{xcolor} 
\usepackage{amsfonts,amssymb,amsmath}
\usepackage{slashed}
\usepackage[titletoc,title]{appendix}
\usepackage{wrapfig}
\usepackage{cancel}
\usepackage{cite}
\usepackage{nth}

\usepackage{marginnote}
\usepackage{verbatim}
\usepackage{dsfont}

\definecolor{deepgreen}{cmyk}{1,0,1,0.5}

\numberwithin{equation}{section}

\newtheorem{theorem}{Theorem}[section]
\newtheorem{corollary}[theorem]{Corollary}
\newtheorem{lemma}[theorem]{Lemma}
\newtheorem{proposition}[theorem]{Proposition}
\newtheorem{remark}[theorem]{Remark}
\newtheorem{definition}[theorem]{Definition}

\newcommand{\jap}[1]{\langle #1\rangle}

\renewcommand{\Re}{\mathrm{Re}}
\renewcommand{\Im}{\mathrm{Im}}

\newcommand{\barv}{{\overline v}}

\renewcommand{\hbar}{{\underline h}}

\newcommand{\bbR}{\mathbb R}

\newcommand{\bbZ}{\mathbb Z}

\newcommand{\calC}{\mathcal C}
\newcommand{\calD}{\mathcal D}
\newcommand{\calE}{\mathcal E}
\newcommand{\calF}{\mathcal F}

\newcommand{\calH}{\mathcal H}
\newcommand{\calI}{\mathcal I}
\newcommand{\calJ}{\mathcal J}

\newcommand{\calL}{\mathcal L}

\newcommand{\calN}{\mathcal N}
\newcommand{\calO}{\mathcal O}

\newcommand{\calQ}{\mathcal Q}
\newcommand{\calR}{\mathcal R}
\newcommand{\calS}{\mathcal S}
\newcommand{\calT}{\mathcal T}

\newcommand{\hatf}{{\hat{f}}}

\newcommand{\hatg}{{\hat{g}}}

\newcommand{\hath}{{\hat{h}}}

\newcommand{\hatr}{{\hat{r}}}

\newcommand{\hatw}{{\hat{w}}}
\newcommand{\whatw}{{\widehat{w}}}

\newcommand{\px}{\partial_x}
\newcommand{\py}{\partial_y}
\newcommand{\pxi}{\partial_\xi}

\newcommand{\pt}{\partial_t}
\newcommand{\ps}{\partial_s}

\newcommand{\jt}{\jap{t}}
\newcommand{\js}{\jap{s}}
\newcommand{\jx}{\jap{x}}
\newcommand{\jxi}{\jap{\xi}}

\newcommand{\jD}{\jap{D}}
\newcommand{\hf}{\frac{1}{2}}
\newcommand{\thf}{\frac{3}{2}}
\newcommand{\Hess}{\mathrm{Hess}}
\newcommand{\bv}{\bar{v}}

\newcommand{\wtcalI}{\widetilde{\calI}}

\newcommand{\ud}{\mathrm{d}}

\def\wh{\widehat}
\def\R{\mathbb{R}}
\def\eps{\varepsilon}
\def\nn{\nonumber}
\def\ol{\overline}

\def\les{\lesssim}
\def\one{\mathbbm{1}}
\def\les{\lesssim} 
\def\calL{\mathcal{L}}
\def\calS{\mathcal{S}}
\def\pih{\frac{\pi}{2}}

\def\PV{\mathrm{PV}}
\def\wzwpi{\sqrt{2\pi}\,}
\def\ftn{\frac{1}{\wzwpi}}

\def\ol{\overline}
\def\sfa{\sqrt{\frac{\pi}{2}}\,}
\def\sfainv{\sqrt{\frac{2}{\pi}}\,}
\newcommand{\EQ}[1]{\begin{equation}\begin{split} #1 \end{split}\end{equation}}

\def\calI{\mathcal{I}}
\def\Iop{\calI}
\def\jD{\jap{D}}
\def\jxi{\jap{\xi}}

\def\calC{\mathcal{C}}
\def\calH{\mathcal{H}}
\def\bxi{\boldsymbol\xi}
\def\calT{\mathcal{T}}
\def\calJ{\mathcal{J}}
\def\calF{\mathcal{F}}

\def\one{\mathbbm{1}}

\def\FT{\mathcal{F}}

\newcommand{\bZ}{{\mathbb Z}}

\renewcommand{\Im}{\mathrm{Im}}
\renewcommand{\Re}{\mathrm{Re}}
\DeclareMathOperator{\sech}{sech}
\DeclareMathOperator{\cosech}{cosech}
\DeclareMathOperator{\Sech}{Sech}
\DeclareMathOperator{\Cosech}{Cosech}

\begin{document}

\title[Asymptotic stability of the sine-Gordon kink under odd perturbations]{Asymptotic stability of the sine-Gordon kink \\ under odd perturbations}

\author[J. L\"uhrmann]{Jonas L\"uhrmann}
\address{Department of Mathematics \\ Texas A\&M University \\ College Station, TX 77843, USA}
\email{luhrmann@math.tamu.edu}

\author[W. Schlag]{Wilhelm Schlag}
\address{Department of Mathematics \\ Yale University \\ New Haven, CT 06511, USA}
\email{wilhelm.schlag@yale.edu}

\thanks{
J. L\"uhrmann was partially supported by NSF grant DMS-1954707. 
W. Schlag was partially supported by NSF grant DMS-1902691.
}

\begin{abstract}
 We establish the asymptotic stability of the sine-Gordon kink under odd perturbations that are sufficiently small in a weighted Sobolev norm.
 Our approach is perturbative and does not rely on the complete integrability of the sine-Gordon model.
 Key elements of our proof are a specific factorization property of the linearized operator around the sine-Gordon kink, a remarkable non-resonance property exhibited by the quadratic nonlinearity in the Klein-Gordon equation for the perturbation, and a variable coefficient quadratic normal form introduced in~\cite{LLS2}. 
 We emphasize that the restriction to odd perturbations does not bypass the effects of the odd threshold resonance of the linearized operator.
 Our techniques have applications to soliton stability questions for several well-known non-integrable models, for instance, to the asymptotic stability problem for the kink of the $\phi^4$ model as well as to the conditional asymptotic stability problem for the solitons of the focusing quadratic and cubic Klein-Gordon equations in one space dimension. 
\end{abstract}

\maketitle 

\tableofcontents

\newpage
\section{Introduction}

\subsection{The sine-Gordon model}

The sine-Gordon model is a classical nonlinear scalar field theory for real-valued fields $\phi \colon \bbR^{1+1} \to \bbR$. 
Its Lagrangian is given by
\begin{equation*}
 \iint_{\bbR^{1+1}} \Bigl( -\hf (\pt \phi)^2 + \hf (\px \phi)^2 + 1-\cos(\phi) \Bigr) \, \ud x \, \ud t
\end{equation*}
and the associated Euler-Lagrange equation reads
\begin{equation} \label{equ:euler_lagrange}
 (\pt^2 - \px^2) \phi = - \sin(\phi), \quad (t,x) \in \bbR \times \bbR. 
\end{equation}
The equation enjoys space-time translation invariance and Lorentz invariance. Its solutions formally conserve the energy 
\begin{equation*}
 E = \int_\bbR \Bigl( \hf (\pt \phi)^2 + \hf (\px \phi)^2 + 1-\cos(\phi) \Bigr) \, \ud x.
\end{equation*}

The sine-Gordon model was discovered as early as in the 1860s in the study of surfaces with constant negative curvature~\cite{Bour1862} and it arises in a diverse range of applications in physics. We refer to the monographs~\cite{DauxPey10, Lamb80, MantSut04, SG_Series} for more background.
Due to its complete integrability the sine-Gordon model assumes a special place among other well-known nonlinear scalar field theories on the line such as the $\phi^4$ model, the $P(\phi)_2$ theories, or the double sine-Gordon theory. 

\medskip 

Global existence of all finite energy solutions to~\eqref{equ:euler_lagrange} is a consequence of a standard fixed-point argument and energy conservation. The long-time behavior of solutions to~\eqref{equ:euler_lagrange} is therefore the main objective in the study of the dynamics of the sine-Gordon model. In this regard its soliton solutions such as kinks and breathers play a fundamental role. 

A static sine-Gordon kink is a stationary solution to~\eqref{equ:euler_lagrange} that connects the two constant finite energy solutions $0$ and $2\pi$, which are also referred to as vacuum solutions of~\eqref{equ:euler_lagrange}. In other words, a static kink corresponds to a heteroclinic orbit of the ordinary differential equation $f'' = \sin(f)$ and is, up to translation, explicitly given by 
\begin{equation*}
 K(x) = 4 \arctan(e^x).
\end{equation*}
The invariance of~\eqref{equ:euler_lagrange} under spatial translations and Lorentz transformations then gives rise to the family of moving kinks
\begin{equation*}
 K_{\ell, a}(t,x) = K\bigl( \gamma (x-\ell t-a) \bigr),
\end{equation*}
where $\ell \in (-1,1)$, $a \in \bbR$, and $\gamma := (1-\ell^2)^{-\hf}$.
Kinks are simple examples of topological solitons, i.e., solitary wave solutions that exhibit a non-trivial topological invariant.

Breathers are time-periodic, spatially localized solutions to~\eqref{equ:euler_lagrange}. An example is given by
\begin{equation*}
 B_\beta(t,x) = 4 \arctan \biggl( \frac{\beta}{\alpha} \frac{\sin(\alpha t)}{\cosh(\beta x)} \biggr), 
\end{equation*}
where $\alpha := \sqrt{1-\beta^2}$ and $\beta \in (-1,1) \backslash \{0\}$.

Kinks and breathers form the building blocks of the long-time dynamics of solutions to~\eqref{equ:euler_lagrange} in the sense that, loosely speaking, any generic global-in-time solution to~\eqref{equ:euler_lagrange} asymptotically decouples into a finite sum of weakly interacting kinks and breathers plus a radiation term that decays to zero.
Relying on the complete integrability of the sine-Gordon model, its initial value problem had been formally studied via inverse scattering techniques by Ablowitz-Kaup-Newell-Segur~\cite{AblKaupNewSeg73}, Takhtadzhyan~\cite{Takhtadzhyan74}, Kaup~\cite{Kaup75}, and Faddeev-Takhtadzhyan-Zakharov~\cite{ZakhTakhFadd74}. Recently, the soliton resolution for the sine-Gordon model was rigorously established by Chen-Liu-Lu~\cite[Theorem 1.1]{CLL20}.

\medskip 

In this work we investigate the asymptotic stability of the sine-Gordon kink $K(x) = 4 \arctan(e^x)$ under small perturbations.
Interestingly, the presence of exceptional periodic solutions called wobbling kinks poses an obstruction to the asymptotic stability of the sine-Gordon kink under small perturbations in the energy space, but not relative to perturbations that are small with respect to (sufficiently) weighted Sobolev norms. 
Wobbling kinks can be thought of as nonlinear superpositions of a kink and a breather. 
We refer to Alejo-Mu\~{n}oz-Palacios~\cite{AMP20} and Chen-Liu-Lu~\cite{CLL20} for detailed discussions of this aspect of the problem.

While the asymptotic stability problem for the sine-Gordon kink has been resolved via inverse scattering techniques, see Chen-Liu-Lu~\cite[Corollary 1.5]{CLL20}, we instead proceed perturbatively in this work and do not rely on the complete integrability of the sine-Gordon model. 
One finds that the evolution equation for odd perturbations $u(t,x) := \phi(t,x) - K(x)$ of the sine-Gordon kink $K(x)$ is given by
\begin{equation} \label{equ:intro_evol_equ_for_perturbation}
 \begin{aligned}
  \bigl( \pt^2 - \px^2 - 2 \sech^2(x) + 1 \bigr) u = - \sech(x) \tanh(x) u^2 + \bigl( {\textstyle \frac16} - {\textstyle \frac13} \sech^2(x) \bigr) u^3 + \{\text{higher order}\}.
 \end{aligned}
\end{equation}
Correspondingly, the proof of the asymptotic stability of the sine-Gordon kink under odd perturbations consists in proving the decay to zero of small solutions to~\eqref{equ:intro_evol_equ_for_perturbation}. We establish sharp decay estimates and asymptotics of solutions to~\eqref{equ:intro_evol_equ_for_perturbation} for odd initial data that are small with respect to a weighted Sobolev norm. 
Since arbitrary perturbed kink solutions may asymptotically converge to a slightly translated and boosted kink, our restriction to odd perturbations is only for technical reasons, namely to prevent the translational mode of the kink from entering the dynamics.
We stress that the restriction to odd initial data does not bypass the effects of the odd threshold resonance of the linearized operator around the sine-Gordon kink.

\medskip 

Our motivation is two-fold: 
\begin{itemize}
 \item First, we believe that our asymptotic stability proof highlights remarkable structures of the sine-Gordon model from a non-integrable point of view. 
 These should be of independent relevance in view of the recent interest in the study of the decay and the asymptotics of small solutions to one-dimensional quadratic Klein-Gordon equations with a potential such as~\eqref{equ:intro_evol_equ_for_perturbation}.
 
 \item Second, we introduce techniques to study long-range scattering problems for one-dimensional nonlinear Klein-Gordon equations with P\"oschl-Teller potentials.
 Our approach has applications to soliton stability questions for several well-known non-integrable models, for instance, to the asymptotic stability problem for the kink of the $\phi^4$ model as well as to the conditional asymptotic stability problem for the solitons of the focusing quadratic and cubic Klein-Gordon equations in one space dimension. 
\end{itemize}

\subsection{Previous results} \label{subsec:previous_results}

The study of the stability of solitons in nonlinear dispersive and hyperbolic equations is a rich and vast subject that we cannot review here in its entirety. In this subsection we only attempt to give an overview of previous results on the stability of kinks in classical nonlinear scalar field theories on the line and we discuss prior works on the closely related problem of proving decay and asymptotics of small solutions to one-dimensional Klein-Gordon equations. 

\medskip 

The orbital stability of kink solutions arising in nonlinear scalar field theories on the line for a large class of scalar double-well potentials was established by Henry-Perez-Wreszinski~\cite{HPW82}. 
The asymptotic stability of (moving) kinks was proved by Komech-Kopylova~\cite{KK11_1, KK11_2} for a certain class of nonlinear scalar field models under a sufficient flatness assumption on the double-wells of the scalar potential and under suitable spectral assumptions (no threshold resonances, possibility of a positive gap eigenvalue). 
Kowalczyk-Martel-Mu\~{n}oz~\cite{KMM17} established the asymptotic stability (in a local energy decay sense) of the kink in the $\phi^4$ model under odd perturbations.
A key difficulty in the analysis of the dynamics of perturbations of the $\phi^4$ kink is a positive gap eigenvalue (often called internal mode) of the linearized operator around the $\phi^4$ kink, see Sigal~\cite{Sigal93} and Soffer-Weinstein~\cite{SofWein99} for pioneering works in this direction. 
Delort-Masmoudi~\cite{DelMas20} proved $L^\infty_x$ decay estimates for odd perturbations of the $\phi^4$ kink up to times $\varepsilon^{-4}$, where $0 < \varepsilon \ll 1$ measures the size of the perturbation in a weighted Sobolev norm. 
As an application of general results on the decay of solutions to one-dimensional quadratic Klein-Gordon equations with potentials described further below, Germain-Pusateri~\cite{GP20} obtained the asymptotic stability under odd perturbations of kinks that occur in the double sine-Gordon theory within a certain range of the deformation parameter.
A general sufficient condition for asymptotic stability (in a local energy decay sense) of moving kinks under arbitrary small perturbations was found by Kowalczyk-Martel-Mu\~{n}oz-Van den Bosch~\cite{KMMV20} under certain spectral assumptions (no threshold resonances, no internal modes).
As already emphasized earlier, the asymptotic stability of the sine-Gordon kink under arbitrary small perturbations, and more generally the soliton resolution for the dynamics of the sine-Gordon model, was established by Chen-Liu-Lu~\cite{CLL20} using inverse scattering techniques.
Asymptotic stability of the sine-Gordon kink (in a local energy decay sense) for perturbations with symmetry had previously been proved by Alejo-Mu\~{n}oz-Palacios~\cite{AMP20}. 

Similar (conditional) asymptotic stability questions for solitons in focusing nonlinear Klein-Gordon models in one space dimensions have been addressed by Bizo\'{n}-Chmaj-Szpak~\cite{BizTadSzp11},  by the second author in joint work with Krieger and Nakanishi~\cite{KNS12}, and by Kowalczyk-Martel-Mu\~{n}oz~\cite{KMM19}.

We also refer to the works of Alejo-Mu\~{n}oz-Palacios~\cite{AlejMunPal17} on the stability of sine-Gordon breathers, to Mu\~{n}oz-Palacios~\cite{MunPal19} on the stability of 2-solitons in the sine-Gordon equation, and to Jendrej-Kowalczyk-Lawrie~\cite{JKL19} on the dynamics of kink-antikink pairs. 
    
Finally, we refer to the surveys~\cite{Tao09, KMM17_1} and references therein for results on closely related asymptotic stability questions for solitary waves in nonlinear Schr\"odinger equations, generalized KdV equations, and other nonlinear wave equations.

\medskip 

In a perturbative approach to the asymptotic stability problem for kinks in one-dimensional scalar field models one needs to prove that small perturbations of kinks decay to zero in a suitable sense. Omitting modulation theory aspects, the evolution equation for a perturbation of a static kink is a one-dimensional nonlinear Klein-Gordon equation of the general schematic form 
\begin{equation} \label{equ:intro_schematic_nlkg_for_perturbations}
 (\pt^2 - \px^2 + V(x) + m^2) u = \alpha(x) u^2 + \beta_0 u^3, \quad (t,x) \in \bbR \times \bbR,
\end{equation}
where $V(x)$ is a smooth localized potential, $m > 0$ is a mass parameter, $\alpha(x)$ is a smooth variable coefficient, and $\beta_0 \in \bbR$ is a constant coefficient. 
Despite the apparent simplicity of the Klein-Gordon model~\eqref{equ:intro_schematic_nlkg_for_perturbations}, the analysis of the long-time behavior of small solutions to~\eqref{equ:intro_schematic_nlkg_for_perturbations} features a surprising number of interesting difficulties. 
Due to the slow dispersive decay of Klein-Gordon waves in one space dimension, quadratic and cubic nonlinearities exhibit long-range effects that lead to modified asymptotics of small solutions in comparison to the free Klein-Gordon flow. 
Moreover, the variable coefficient $\alpha(x)$ and the potential $V(x)$ in~\eqref{equ:intro_schematic_nlkg_for_perturbations} cause a decorrelation of (distorted) input and output frequencies in the nonlinear interactions, which may lead to the occurrence of delicate resonance phenomena in the quadratic nonlinearity. 
In addition, the linear operator in~\eqref{equ:intro_schematic_nlkg_for_perturbations} may exhibit a threshold resonance, which is responsible  for slow local decay properties of the solutions that can significantly complicate the analysis of the long-time behavior of solutions to~\eqref{equ:intro_schematic_nlkg_for_perturbations}. The linearized operators around the sine-Gordon kink and the $\phi^4$ kink both have threshold resonances. It is worth pointing out a peculiar feature in one space dimension: in contrast to higher odd space dimensions, the flat linear Klein-Gordon operator (with $V(x) = 0$) exhibits a threshold resonance, namely the constant function $1$. 
Finally, the linearized operator may possess a positive gap eigenvalue (internal mode). The prime example for this phenomenon is the $\phi^4$ model. At the linear level such a positive gap eigenvalue would be an obstruction to decay. However, at the nonlinear level a coupling of the oscillations of the internal mode to the continuous spectrum may occur through the so-called nonlinear Fermi Golden Rule, see~\cite{Sigal93, SofWein99}, leading to the attenuation of the internal mode.

The investigation of the decay of small solutions to nonlinear Klein-Gordon equations in higher space dimensions was pioneered by Shatah~\cite{Sh85} and Klainerman~\cite{Kl93}. 

The seminal work of Delort~\cite{Del01, Del06} established sharp decay estimates and asymptotics for small solutions to one-dimensional nonlinear Klein-Gordon equations of the form
\begin{equation} \label{equ:intro_nlkg_const_coeff}
 (\pt^2 - \px^2 + 1) u = \alpha_0 u^2 + \beta_0 u^3, \quad \alpha_0, \beta_0 \in \bbR.
\end{equation}
See Lindblad-Soffer~\cite{LS05_1, LS05_2}, Hayashi-Naumkin~\cite{HN08, HN12}, Delort\cite{Del16_KG}, Stingo~\cite{Stingo18}, Candy-Lindblad~\cite{CL18} for subsequent results in this direction\footnote{The papers~\cite{Del01, Stingo18} pertain to more general quasilinear nonlinearities. With an eye towards the model~\eqref{equ:intro_schematic_nlkg_for_perturbations} for the kink stability problem, here we only discuss their applicability to~\eqref{equ:intro_nlkg_const_coeff}.}.
Due to the slow dispersive decay of Klein-Gordon waves in one space dimension, the quadratic and cubic nonlinearities produce long-range effects in the sense that small solutions to~\eqref{equ:intro_nlkg_const_coeff} have the same $L^\infty_x$ decay rate $t^{-\hf}$ as free Klein-Gordon waves, but exhibit logarithmic phase corrections with respect to the free flow. An intriguing number of different techniques have been devised in order to capture this modified scattering behavior. Oversimplifying slightly, one generally combines an ODE argument with the derivation of slowly growing energy estimates for a Lorentz boost $Z = t\px + x \pt$ (or the closely related operator $L = \jD x - it\px$) applied to the solution.

The first results on the long-time behavior of small solutions to the following Klein-Gordon equation with an additional variable coefficient cubic nonlinearity were obtained by Lindblad-Soffer~\cite{LS15} and Sterbenz~\cite{Sterb16},
\begin{equation} \label{equ:intro_nlkg_variable_cubic}
 (\pt^2 - \px^2 + 1) u = \alpha_0 u^2 + \beta_0 u^3 + \beta(x) u^3, \quad \alpha_0, \beta_0 \in \bbR, \quad \beta(\cdot) \in \calS(\bbR).
\end{equation}
The new difficulty caused by the variable coefficient $\beta(x)$ is related to the need for deriving slowly growing energy estimates for a Lorentz boost of the solution. Indeed, when the vector field $Z$ falls onto the variable coefficient, it produces a strongly divergent factor of $t$ that can be difficult to counteract. \cite{LS15, Sterb16} devised a variable coefficient cubic normal form to overcome this issue. More recently, the first author in joint work with Lindblad and Soffer~\cite{LLS1} introduced the use of local decay estimates as a robust way to handle this difficulty.

The study of Klein-Gordon equations with variable coefficient quadratic nonlinearities was recently initiated by the first author in joint work with Lindblad and Soffer~\cite{LLS2}, where the following model is considered
\begin{equation} \label{equ:intro_nlkg_LLS2part1}
 (\pt^2 - \px^2 + 1) u = \alpha(x) u^2, \quad \alpha(\cdot) \in \calS(\bbR).
\end{equation}
Due to the spatial localization of the variable coefficient $\alpha(x)$, the asymptotic behavior of small solutions to~\eqref{equ:intro_nlkg_LLS2part1} is governed by the local decay properties of the solutions. Already in the linear case, the local decay of free Klein-Gordon waves in one space dimension is slow and only of the order of $t^{-\hf}$ owing to the threshold resonance of the flat linear Klein-Gordon operator in one space dimension. This results in the dynamic formation of a source term of the schematic form $\alpha(x) e^{2it} t^{-1}$ on the right-hand side of~\eqref{equ:intro_nlkg_LLS2part1}, which exhibits a striking resonant interaction between the temporal oscillations $e^{2it}$ and the frequencies $\xi = \pm \sqrt{3}$ of the coefficient $\alpha(x)$. The latter leads to a logarithmic slow-down of the decay rate along the associated rays $x = \mp \frac{\sqrt{3}}{2} t$, so that the free $L^\infty_x$ decay rate $t^{-\hf}$ is not propagated by the nonlinear flow. See Subsections~1.4 in~\cite{LLS2} and \cite{LLSS} for a more detailed heuristic explanation of this phenomenon. 
This type of resonance can present a fundamental difficulty in the analysis of the long-time behavior of small solutions to quadratic Klein-Gordon equations of the form~\eqref{equ:intro_schematic_nlkg_for_perturbations}.
The possibility of a logarithmic slow-down of the free decay rate due to the presence of a space-time resonance had been demonstrated by Bernicot-Germain~\cite{BerGerm13} in the context of analyzing bilinear interactions of free dispersive waves in one space dimension. See also Deng-Ionescu-Pausader~\cite{DIP17} and Deng-Ionescu-Pausader-Pusateri~\cite{DIPP17} for higher-dimensional examples, where the free decay rate cannot be propagated by the nonlinear flow.

The analysis of~\eqref{equ:intro_nlkg_LLS2part1} in~\cite{LLS2} crucially relies on the spatial localization of the coefficient $\alpha(x)$, and it is not straightforward to include a constant coefficient cubic nonlinearity or more ambitiously, a coefficient $\alpha(x)$ with non-zero limits as $x \to \pm \infty$, if no symmetry assumptions are made.
In the special case where $\widehat{\alpha}(\pm \sqrt{3}) = 0$ and the above resonance phenomenon is correspondingly suppressed, \cite[Theorem 1.7]{LLS2} establishes $L^\infty_x$ decay estimates at the rate $t^{-\hf}$ with logarithmic phase corrections in the asymptotics for small solutions to
\begin{equation} \label{equ:intro_nlkg_LLS2part2}
 (\pt^2 - \px^2 + 1) u = \alpha(x) u^2 + \beta_0 u^3, \quad \alpha(\cdot) \in \calS(\bbR), \quad \widehat{\alpha}(\pm \sqrt{3}) = 0,
\end{equation}
via the introduction of a variable coefficient quadratic normal form. The latter also plays a key role in this work. 

Delort-Masmoudi~\cite{DelMas20} studied the long-time behavior of odd perturbations of the $\phi^4$ kink and obtained $L^\infty_x$ decay estimates up to times $\sim \varepsilon^{-4}$, where $0 < \varepsilon \ll 1$ measures the size of the perturbation in a weighted Sobolev norm. Odd perturbations of the $\phi^4$ kink satisfy the following nonlinear Klein-Gordon equation
\begin{equation} \label{equ:intro_nlkg_Delort_Masmoudi}
 \bigl( \pt^2 - \px^2 - 3 \sech^2( {\textstyle \frac{x}{\sqrt{2}} }) + 2 \bigr) u = - 3 \tanh( {\textstyle \frac{x}{\sqrt{2}} } ) u^2 - u^3.
\end{equation}
The linear operator on the left-hand side of~\eqref{equ:intro_nlkg_Delort_Masmoudi} has an even zero eigenfunction and an even threshold resonance (which are not relevant due to the odd parity assumption), but also an odd internal mode. 
Thus, the evolution equation~\eqref{equ:intro_nlkg_Delort_Masmoudi} in fact becomes a coupled system of a nonlinear Klein-Gordon equation for the projection of $u(t)$ onto the continuous spectral subspace and an ODE for the projection of $u(t)$ onto the internal mode.
The analysis of this involved coupled PDE/ODE system in~\cite{DelMas20} includes, among other aspects, the use of the wave operator of the linearized Klein-Gordon operator to conjugate to the flat linear Klein-Gordon equation, new normal forms, and an implementation of the Fermi Golden Rule. 
It appears that the limitation up to times $\sim \varepsilon^{-4}$ in~\cite{DelMas20} stems from a source term that the internal mode creates in the Klein-Gordon equation for the projection of $u(t)$ onto the continuous spectral subspace. To a certain extent, this source term bears a striking resemblance to the source term created by the threshold resonance in the dynamics of the equation~\eqref{equ:intro_nlkg_LLS2part1}, see also~\eqref{equ:intro_nlkg_LLSS} below. It could potentially lead to a logarithmic slow-down of the nonlinear solution along certain rays as well. It is worth noting that such a phenomenon cannot be detected by the local-in-space analysis in Kowalczyk-Martel-Mu\~{n}oz~\cite{KMM17}. We refer to Section~1.10 in~\cite{DelMas20} for a more elaborate discussion of this aspect. See also Remark (x) on Theorem~1.1 in~\cite{LLSS}.

The most general results on the long-time behavior of small solutions to Klein-Gordon models of the form~\eqref{equ:intro_schematic_nlkg_for_perturbations} were obtained by Germain-Pusateri \cite{GP20} who considered the equation
\begin{equation} \label{equ:intro_nlkg_Germain_Pusateri}
 (\pt^2 - \px^2 + V(x) + 1) u = a(x) u^2, \quad \lim_{x\to\pm\infty} a(x) = \ell_{\pm \infty} \in \bbR,
\end{equation}
where $V(x)$ is a Schwartz class potential and where $-\px^2 + V(x)$ is assumed to have no bound states. The key spectral assumption in~\cite{GP20} is that the distorted Fourier transform of the nonlinear solution vanishes at zero frequency at all times, i.e., $\widetilde{\calF}[u(t,\cdot)](0) = 0$. This condition is automatically satisfied for generic potentials (no threshold resonance) and can be enforced for non-generic potentials by improsing suitable parity conditions. Under these assumptions, \cite[Theorem 1.1]{GP20} establishes that small solutions to~\eqref{equ:intro_nlkg_Germain_Pusateri} decay in $L^\infty_x$ at the free rate $t^{-\hf}$ and exhibit logarithmic phase corrections ``caused by the non-zero limits'' $\ell_{\pm \infty}$ of the coefficient $a(x)$ as $x \to \pm \infty$. The approach in~\cite{GP20} is based on the distorted Fourier transform adapted to the Schr\"odinger operator $-\px^2 + V(x)$ along with new quadratic normal forms and a refined functional framework to capture the modified scattering behavior of small solutions to~\eqref{equ:intro_nlkg_Germain_Pusateri}. \cite{GP20} further highlights that the above mentioned special frequencies $\xi = \pm \sqrt{3}$ are the (distorted) output frequencies of a nonlinear space-time resonance, which is generally expected to occur for quadratic interactions in one space dimension in the presence of a linear potential $V(x)$ and/or a variable coefficient quadratic nonlinearity. 
We note that a slow-down of the decay rate of solutions to~\eqref{equ:intro_nlkg_Germain_Pusateri} under the assumptions of~\cite[Theorem 1.1]{GP20} does not occur because of the improved local decay of the solutions due to the vanishing assumption $\widetilde{\calF}[u(t,\cdot)](0) = 0$.

Finally, we mention the authors' recent joint work with Lindblad and Soffer~\cite{LLSS} concerning the Klein-Gordon model 
\begin{equation} \label{equ:intro_nlkg_LLSS}
 (\pt^2 - \px^2 + V(x) + 1) u = P_c \bigl( \alpha(\cdot) u^2 \bigr), \quad \alpha(\cdot) \in \calS(\bbR),
\end{equation}
where the potential $V(x)$ is assumed to be non-generic. In other words, the Schr\"odinger operator $-\px^2 + V(x)$ is assumed to exhibit a threshold resonance, i.e., a non-trivial bounded solution to $(-\px^2 + V(x)) \varphi = 0$ satisfying $\varphi(x) \to 1$ as $x \to \infty$.
\cite[Theorem 1.1]{LLSS} establishes an analogous modified scattering behavior to~\cite[Theorem 1.1]{LLS2} for the special case~\eqref{equ:intro_nlkg_LLS2part1}, involving a logarithmic slow-down along the rays $x = \pm \frac{\sqrt{3}}{2} t$, if $\widetilde{\calF}[\alpha \varphi^2](\pm \sqrt{3}) \neq 0$. This further highlights the role that the threshold resonances of the linear operator and the associated local decay properties of the solutions play for the long-time behavior of solutions to Klein-Gordon models of the form~\eqref{equ:intro_nlkg_LLSS}. 
Moreover, \cite[Remark 1.2]{LLSS} uncovered that for the linearized equation around the sine-Gordon kink the remarkable non-resonance property 
$\widetilde{\calF}[\alpha \varphi^2](\pm \sqrt{3}) = 0$ holds. A related observation plays a key role in this work. 

The techniques that have been developed for the analysis of the long-time behavior of solutions to Klein-Gordon equations such as~\eqref{equ:intro_schematic_nlkg_for_perturbations} are of course closely related to (and were at times preceded by) similar developments in the study of long-range scattering problems for other nonlinear dispersive equations. We specifically mention Hayashi-Naumkin~\cite{HN98}, Lindblad-Soffer~\cite{LS06}, Kato-Pusateri~\cite{KatPus11}, and Ifrim-Tataru~\cite{IT15} on the modified scattering of small solutions to the one-dimensional cubic nonlinear Schr\"odinger equation. Moreover, we refer to the following long-range scattering results for the one-dimensional cubic nonlinear Schr\"odinger equation with a generic potential (or in some cases with a non-generic potential under symmetry assumptions) by Naumkin~\cite{Naum16, Naum18}, Germain-Pusateri-Rousset~\cite{GermPusRou18}, Delort~\cite{Del16}, Masaki-Murphy-Segata~\cite{MasMurphSeg19}, and Chen-Pusateri~\cite{ChenPus19}. 
See also~\cite{DZ03, GHW15, MurphPus17, GPR16, Leg18, Leg19, PusSof20} and references therein. 

Finally, we point the reader to the introductory chapter of the recent work of Delort-Masmoudi~\cite{DelMas20} on the stability of the $\phi^4$ kink for a thorough overview of, and historical perspective on, previous results on soliton stability and the long-time behavior of small solutions to Klein-Gordon models related to~\eqref{equ:intro_schematic_nlkg_for_perturbations}.

\subsection{Main result}

We are now in the position to state the main result of this work.

\begin{theorem} \label{thm:main}
 The sine-Gordon kink $K(x) = 4 \arctan(e^x)$ is asymptotically stable under small odd perturbations in the following sense:
 There exists a small constant $0 < \varepsilon_0 \ll 1$ such that for any odd initial data $(u_0, u_1)$ with
 \begin{equation*}
  \varepsilon := \| \langle x \rangle (u_0, u_1) \|_{H^3_x \times H^2_x} \leq \varepsilon_0,
 \end{equation*} 
 the solution to 
 \begin{equation*}
  \left\{ \begin{aligned}
           (\partial_t^2 - \partial_x^2) \phi &= - \sin(\phi), \quad (t,x) \in \mathbb{R} \times \mathbb{R}, \\
           (\phi, \partial_t \phi)|_{t=0} &= (K + u_0, u_1),
          \end{aligned} \right.
 \end{equation*}
 satisfies 
 \begin{equation} \label{equ:main_thm_decay_perturbation}
  \| \phi(t,\cdot) - K(\cdot) \|_{L^\infty_x} \lesssim \frac{\varepsilon}{\jt^{\frac12}}, \quad t \in \bbR.
 \end{equation}
 Moreover, there exists an even asymptotic profile $\widehat{W} \in L^\infty$ and a small constant $0 < \delta \ll 1$ such that the perturbation 
 \begin{equation*}
  u(t,x) := \phi(t,x) - K(x)
 \end{equation*}
 exhibits the asymptotics 
 \begin{equation} \label{equ:main_thm_asymptotics}
 \begin{aligned}
  \biggl| u(t,x) + 2 \Re \, \biggl( \frac{e^{i\frac{\pi}{4}}}{t^\hf} \int_0^x \frac{\cosh(y)}{\cosh(x)} e^{i\rho} e^{-i \psi(\frac{y}{\rho}) \log(t)} &\widehat{W}\bigl( {\textstyle \frac{y}{\rho}} \bigr) \one_{(-1,1)}({\textstyle \frac{y}{t}}) \, \ud y \biggr) \biggr| \lesssim \frac{\varepsilon}{t^{\frac23-\delta}}, \quad t \geq 1,
 \end{aligned}
 \end{equation} 
 where $\rho := \sqrt{t^2-y^2}$ and 
 \begin{equation*} 
 \begin{aligned}
  \psi(\xi) := \frac14 \jxi^{-7} \bigl( 1 + 3 \xi^2 \bigr) \bigl|\widehat{W}(\xi) \bigr|^2.
 \end{aligned}
 \end{equation*}
 An analogous expression for the asymptotics of $u(t,x)$ holds for negative times $t \leq -1$.
\end{theorem}

As noted above, the restriction to odd perturbations in Theorem~\ref{thm:main} does not bypass the effects of the odd threshold resonance of the linearized operator around the sine-Gordon kink. The oddness assumption prevents the translational mode of the kink from entering the dynamics. We expect to be able to prove the asymptotic stability of the sine-Gordon kink under arbitrary small perturbations by incorporating modulation theory.
 
\begin{remark} 
 Our proof of Theorem~\ref{thm:main} relies in a crucial way on a remarkable factorization property of the linearized Klein-Gordon operator 
 \begin{equation} \label{equ:intro_remark_sG_linear_op}
  -\px^2 - 2\sech^2(x) + 1
 \end{equation}
 around the sine-Gordon kink. Its potential belongs to the family of reflectionless P\"oschl-Teller potentials~\cite{PoschlTeller}.
 Introducing the first-order differential operator 
 \begin{equation*}
  \calD := \px - \tanh(x)
 \end{equation*}
 and its adjoint 
 \begin{equation*}
  \calD^\ast := -\px - \tanh(x),
 \end{equation*}
 we may write
 \begin{equation*}
  \calD \calD^\ast = -\px^2 - 2\sech^2(x) + 1.
 \end{equation*}
 It turns out that the conjugate operator is just the flat linear operator 
 \begin{equation} \label{equ:intro_remark_flat_KG_op}
  \calD^\ast \calD = -\px^2 + 1. 
 \end{equation}
 Thus, upon differentiating by $\calD^\ast$ the Klein-Gordon equation~\eqref{equ:intro_evol_equ_for_perturbation} for the perturbation $u$ of the sine-Gordon kink, we obtain a new evolution equation with a flat Klein-Gordon operator for the new dependent variable $\calD^\ast u$.
 Observe that the linearized operator~\eqref{equ:intro_remark_sG_linear_op} exhibits the odd threshold resonance $\varphi(x) = \tanh(x)$, while ~\eqref{equ:intro_remark_flat_KG_op} has the even threshold resonance $\varphi(x) = 1$.
 
 Such factorization ideas have for instance previously been used in the study of blowup for energy-critical wave maps by Rodnianski-Sterbenz~\cite{RodSterb10}, Rapha\"{e}l-Rodnianski~\cite{RaphRod12}, Krieger-Miao~\cite{KriegerMiao20}, and by the second author in joint work with Krieger and Miao~\cite{KrMiSchl20}.
 
 In the context of studying the conditional asymptotic stability of solitons in 1D focusing nonlinear Klein-Gordon equations by Kowalczyk-Martel-Mu\~{n}oz~\cite{KMM19} and of kinks in scalar field theories by Kowalczyk-Martel-Mu\~{n}oz-Van den Bosch~\cite{KMMV20}, such factorization ideas have been key for the derivation of local energy decay estimates. See also Chang-Gustafson-Nakanishi-Tsai~\cite{ChangGustNakTsai}. 
 
 In fact, it follows from Subsection~5.2 in \cite{KMMV20} that among all scalar field theories on the line with double-well potentials supporting kink solutions, up to invariances, the sine-Gordon model is unique with the property that the conjugate of the linearized operator around its kink is just the flat linear Klein-Gordon operator. 
    
 To the best of our knowledge, our proof of Theorem~\ref{thm:main} appears to be the first instance where such factorizations are used to derive sharp $L^\infty_x$ decay estimates and asymptotics for solutions in the context of a long-range scattering problem. 
 Our approach has applications to soliton stability questions for several well-known non-integrable models, where the linearized operators feature P\"oschl-Teller potentials. Examples include the asymptotic stability problem for the kink of the $\phi^4$ model as well as the conditional asymptotic stability problem for the solitons of the focusing quadratic and cubic Klein-Gordon equations in one space dimension.
\end{remark}

\subsection{Proof ideas and overview}

In this subsection we describe the main ideas that enter the proof of Theorem~\ref{thm:main} and we provide an overview of the structure of this paper.

We show in Subsection~\ref{subsec:evol_equ_perturbation} that the evolution equation for an odd perturbation 
\begin{equation*}
 u(t,x) := \phi(t,x) - K(x)
\end{equation*} 
of the static sine-Gordon kink $K(x) = 4 \arctan(e^x)$ is given by 
\begin{equation} \label{equ:intro_ideas_evol_equ_pert}
 \bigl( \pt^2 - \px^2 - 2 \sech^2(x) + 1 \bigr) u = - \sech(x) \tanh(x) u^2 + \Bigl( \frac16 - \frac13 \sech^2(x) \Bigr) u^3 + \{\text{higher order}\}.
\end{equation}
This is a one-dimensional nonlinear Klein-Gordon equation featuring a linearized operator that has an even zero eigenfunction $Y(x) := \sech(x)$ and an odd threshold resonance $\varphi(x) := \tanh(x)$. 
In~\eqref{equ:intro_ideas_evol_equ_pert} only the quadratic and cubic nonlinearities are displayed, and we ignore the contributions of the milder higher order nonlinearities in this discussion. 
Since odd perturbations $u(t,x)$ automatically belong to the continuous spectral subspace of $L^2_x$ relative to the linearized operator, the proof of the asymptotic stability of the sine-Gordon kink under odd perturbations therefore consists in establishing the decay to zero of small solutions to~\eqref{equ:intro_ideas_evol_equ_pert}. 

A major difficulty in the analysis of~\eqref{equ:intro_ideas_evol_equ_pert} is to capture the long-range effects of the quadratic and cubic nonlinearities in the presence of the linearized operator, whose odd threshold resonance cannot be avoided by considering only odd perturbations. To our knowledge, no general techniques have yet been developed to study long-range scattering problems with linear operators that feature non-trivial, non-generic potentials (without imposing symmetry constraints to avoid the threshold resonance). 
Our solution tailored to~\eqref{equ:intro_ideas_evol_equ_pert} is to exploit a specific factorization property of the linearized operator, which allows us to transform the equation~\eqref{equ:intro_ideas_evol_equ_pert} into a more favorable form. Specifically, introducing the first-order differential operator $\calD := \px - \tanh(x)$ and its adjoint $\calD^\ast := -\px - \tanh(x)$,
the linear operator on the left-hand side of~\eqref{equ:intro_ideas_evol_equ_pert} factorizes as
\begin{equation*}
 \calD \calD^\ast = -\px^2 - 2\sech^2(x) + 1.
\end{equation*}
It turns out that its conjugate operator 
\begin{equation*}
 \calD^\ast \calD = -\px^2 + 1
\end{equation*}
is just the flat linear Klein-Gordon operator (which exhibits the even threshold resonance $1$).  
Thus, upon differentiating~\eqref{equ:intro_ideas_evol_equ_pert} by $\calD^\ast$, we find that the new dependent variable $w := \calD^\ast u$ satisfies the following nonlinear Klein-Gordon equation with the flat linear Klein-Gordon operator on the left-hand side
\begin{equation} \label{equ:intro_ideas_transformed_equ_w}
 (\pt^2 - \px^2 + 1) w = \calQ(w) + \calC(w) + \{\text{higher order}\},
\end{equation}
and with quadratic and cubic nonlinearities given by 
\begin{align*}
 \calQ(w) &= \bigl( -2 \sech(x) + 3 \sech^3(x) \bigr) \bigl( \calI[w] \bigr)^2 - 2 \sech(x) \tanh(x) \calI[w] w, \\
 \calC(w) &= \frac12 \bigl( \calI[w] \bigr)^2 w + \frac13 \tanh(x) \bigl( \calI[w] \bigr)^3 - \frac43 \sech^2(x) \tanh(x) \bigl( \calI[w] \bigr)^3 - \sech^2(x) \bigl( \calI[w] \bigr)^2 w,
\end{align*}
where 
\begin{equation} \label{equ:intro_ideas_calI_def}
 \calI[w(t,\cdot)](x) := - \sech(x) \int_0^x \cosh(y) w(t,y) \, \ud y.
\end{equation}
See Subsection~\ref{subsec:transformed_equation} for the details of the derivation of the transformed equation~\eqref{equ:intro_ideas_transformed_equ_w} for $w$. 
There we use that the odd dependent variable $u(t,x)$ can be expressed in terms of $w(t,x)$ via the integral expression 
\begin{equation} \label{equ:intro_ideas_calI}
 u(t,x) = \calI[w(t,\cdot)](x).
\end{equation}
In fact, $\calI[\cdot]$ is a right-inverse operator for $\calD^\ast$ and the kernel of $\calD^\ast$ is spanned by the even zero eigenfunction $Y(x)$.

At this point our task is to deduce decay and asymptotics for the solution $w(t)$ to~\eqref{equ:intro_ideas_transformed_equ_w}, from which we can then infer the desired decay and asymptotics for the original variable $u(t,x)$ via~\eqref{equ:intro_ideas_calI}. 
Observe that all quadratic nonlinearities on the right-hand side of~\eqref{equ:intro_ideas_transformed_equ_w} are spatially localized, while the cubic nonlinearities have both localized and non-localized parts.
In view of the general discussion in the preceding Subsection~\ref{subsec:previous_results}, the most problematic contributions could in principle stem from the quadratic nonlinearities. Due to their spatial localization, the local decay of the solution $w(t)$ is decisive for their analysis. Recall that $u(t,x)$ is odd, whence $w(t,x) = (\calD^\ast u)(t, x)$ is an even function. We therefore cannot hope for improved local decay of $w(t)$, because the threshold resonance of the flat linear Klein-Gordon operator is also even. However, we can expect $(\px w)(t)$ to have better local decay since the derivative cancels the effect of the threshold resonance. This leads us to integrate by parts in the definition of the integral operator~\eqref{equ:intro_ideas_calI_def} and to correspondingly rewrite the quadratic nonlinearities as 
\begin{align*}
 \calQ(w) = \calQ_1(w) + \calQ_2(w) + \calQ_3(w),
\end{align*}
where 
\begin{equation*} 
\begin{aligned}
 \calQ_1(w) := \alpha_1(x) w^2, \quad \calQ_2(w) := \alpha_2(x) \wtcalI[\px w] w, \quad \calQ_3(w) := \alpha_3(x) \bigl( \wtcalI[\px w] \bigr)^2,
\end{aligned}
\end{equation*}
for some spatially localized coefficients $\alpha_1, \alpha_2, \alpha_3 \in \calS(\bbR)$, and where
\begin{equation*}
 \widetilde{\calI}[\px w(t)](x) := \sech(x) \int_0^x \sinh(y) (\px w)(t,y) \, \ud y.
\end{equation*}
Thus, the critical quadratic contribution now comes from $\calQ_1(w)$, while the contributions of $\calQ_2(w)$ and $\calQ_3(w)$ are less severe due to the improved local decay of $(\px w)(t)$. In Lemma~\ref{lem:nonresonance} we make the key observation that the Fourier transform of the coefficient $\alpha_1(x)$ vanishes at the frequencies $\xi = \pm \sqrt{3}$. This suppresses the occurrence of a resonant interaction in $\calQ_1(w)$ and we are in the position to use a variable coefficient quadratic normal form introduced in~\cite{LLS2} to recast the quadratic nonlinearity $\calQ_1(w)$ into a better form.
The fundamental non-resonance property $\widehat{\alpha}_1(\pm \sqrt{3}) = 0$ is a remarkable feature of the sine-Gordon model. In fact, in view of \cite[Theorem 1.1]{LLS2} we would not expect to be able to propagate the $L^\infty_x$ decay rate $t^{-\hf}$ for the perturbation of the sine-Gordon kink if $\widehat{\alpha}_1(\pm \sqrt{3}) = 0$ did not hold.
We then arrive at the following first-order nonlinear Klein-Gordon equation for the renormalized variable $v + B(v,v)$,
\begin{equation} \label{equ:intro_ideas_renorm_equ_v}
  (\pt - i\jD) \bigl( v + B(v,v) \bigr) = \frac{1}{2i} \jD^{-1} \bigl( \calQ_{ren}(v, v) + \calC(v+\bar{v}) + \{\text{higher order}\} \bigr),
\end{equation}
where we set $v(t) := w(t) - i\jD^{-1} (\pt w)(t)$, the variable coefficient quadratic normal form $B(v,v)(t)$ is defined in \eqref{equ:def_variable_coeff_normal_form}, and the renormalized quadratic nonlinearity $\calQ_{ren}(v, v)$ is defined in~\eqref{equ:def_renormalized_quad_nonlinearities}. See Subsection~\ref{subsec:normal_form} for the details of the derivation of~\eqref{equ:intro_ideas_renorm_equ_v}.

At this point we follow the general approach of the space-time resonances method by Germain-Masmoudi-Shatah~\cite{GMS12_Ann, GMS12_JMPA, GMS09} and Gustafson-Nakanishi-Tsai~\cite{GNT09} to infer sharp decay estimates and asymptotics for small solutions to~\eqref{equ:intro_ideas_renorm_equ_v}. Specifically, we seek to obtain via a continuity argument an a priori bound on the following quantity 
\begin{equation*} 
 \begin{aligned}
  N(T) &:= \sup_{0 \leq t \leq T} \, \biggl\{ \jt^{\frac{1}{2}} \|v(t)\|_{L^\infty_x} + \jt^{-\delta} \| \jD^2 v(t) \|_{L^2_x} + \jt^{-\delta} \| \jD L v(t) \|_{L^2_x} \\
  &\qquad \qquad \qquad \qquad \qquad \qquad \qquad \qquad + \jt^{-1-\delta} \|x v(t)\|_{L^2_x} + \bigl\| \jap{\xi}^{\frac{3}{2}} \hat{f}(t,\xi)  \bigr\|_{L^\infty_\xi} \biggr\},
 \end{aligned}
\end{equation*}
where $f(t) := e^{-it\jD} v(t)$ is the profile of the solution $v(t)$ to~\eqref{equ:intro_ideas_renorm_equ_v} on some time interval $[0,T]$, $L := \jD x - it\px$, and $0 < \delta \ll 1$ is a small absolute constant. 
In view of the asymptotics for the linear Klein-Gordon evolution from Lemma~\ref{lem:asymptotics_KG}, the main components of $N(T)$ are a slowly growing bound for $\| \jD L v(t) \|_{L^2_x}$ and a uniform-in-time bound for $\| \jap{\xi}^{\frac{3}{2}} \hat{f}(t,\xi)\|_{L^\infty_\xi}$. 

In Section~\ref{sec:energy_estimates} we carry out the energy estimates for all $L^2_x$-based norms in $N(T)$. The derivation of the slowly growing estimate for $\| \jD L v(t) \|_{L^2_x}$ is the most delicate task in this part. Here the contributions of all spatially localized nonlinearities with cubic-type decay $\jt^{-(\thf-\delta)}$ can be handled using a version of an argument from~\cite{LLS1, LLS2} based on local decay, see Proposition~\ref{prop:key_slow_growth}. All renormalized quadratic nonlinearities and all spatially localized cubic nonlinearities fall into that category. The non-localized cubic nonlinearities also exhibit favorable structures to close the energy estimate for $\jD L v(t)$, see the treatment of~\eqref{equ:growth_H1Lv_def_calCnl} in Proposition~\ref{prop:growth_H1Lv} in conjunction with the identity~\eqref{equ:Z_action_wtcalI} from Lemma~\ref{lem:Z_action_calIs} and the bound~\eqref{equ:growth_L2_Z_action_wtcalIv} from Corollary~\ref{cor:growth_L2_Z_action_calIs}.

In Section~\ref{sec:pointwise_estimates} we perform a stationary phase analysis of the Duhamel expression for the Fourier transform $\hatf(t,\xi)$ to derive a differential equation that governs the asymptotic behavior of $\hatf(t,\xi)$. From the latter we infer the uniform-in-time bound on $\| \jap{\xi}^{\frac{3}{2}} \hat{f}(t,\xi)\|_{L^\infty_\xi}$ via an ODE argument. The Fourier analysis of the nonlinearities on the right-hand side of~\eqref{equ:intro_ideas_renorm_equ_v} is complicated by the non-local character of the integral operator $\calI[\cdot]$ defined in~\eqref{equ:intro_ideas_calI}. However, their analysis can be carried out explicitly, see Subsection~\ref{subsec:fourier_analysis_nonlin}. 

Finally, in Section~\ref{sec:proof_of_thm} we tie together all the preceding steps and conclude the proof of Theorem~\ref{thm:main}.

\medskip 

\noindent {\it Acknowledgments:} The authors are grateful to the referees for their careful proof-reading of the manuscript and for valuable comments.

\section{Preliminaries}

\subsection{Notation and conventions}

We denote by $C > 0$ an absolute constant whose value may change from line to line. For non-negative $X, Y$ we write $X \lesssim Y$ if $X \leq C Y$ and we use the notation $X \ll Y$ to indicate that the implicit constant should be regarded as small. We write $X \simeq Y$ if $X \lesssim Y \lesssim X$. 
Moreover, for non-negative $X$ and arbitrary $Z$, we use the short-hand notation $Z = \calO(X)$ if $|Z| \leq C X$.
Throughout we use the Japanese bracket notation
\begin{align*}
 \jt := (t^2 + 1)^\hf, \quad \jx := (x^2 + 1)^\hf, \quad \jxi := (\xi^2 + 1)^\hf.
\end{align*}
We denote by $\one_I(\cdot)$ the characteristic function of an interval $I \subset \bbR$.

Our conventions for the Fourier transform of a Schwartz function $g \in \calS(\bbR)$ are 
\begin{equation} \label{eq:FT}
 \begin{aligned}
  \calF[g](\xi) &= \hatg(\xi) = \ftn \int_{\bbR} e^{-ix\xi} g(x) \, \ud x, \\
  \calF^{-1}[g](x) &= \check{g}(x) = \ftn \int_{\bbR}  e^{ix\xi} g(\xi) \, \ud \xi.
 \end{aligned}
\end{equation}
Then the convolution laws are given by
\begin{equation*}
 \calF[g \ast h] = \wzwpi \hat{g}\hat{h}, \quad \calF[g h] = \frac{1}{\sqrt{2\pi}} \hat{g}\ast\hat{h}
\end{equation*}
for all $g, h \in \calS(\bbR)$.
We use the standard notations for the Lebesgue spaces $L^p_x$ as well as the Sobolev spaces $H^k_x$ and $W^{k,p}_x$.

We set $D := -i \px$ and define the operator $\jD$ in terms of its symbol $\calF[\jD f](\xi) = \jxi \hatf(\xi)$. Similarly, we introduce the Klein-Gordon propagator $e^{it\jD}$ via $\calF[e^{it\jD} f](\xi) = e^{it\jxi} \hatf(\xi)$. 
Finally, 
\begin{equation*}
 L := \jD x - i t \px,
\end{equation*}
which conjugates to $\jD x$ via $e^{it\jD}$ in the sense that
\begin{equation} \label{equ:relation_L_partial_xi}
 L = \jD x - it\px = e^{it\jD} \jD x e^{-it\jD} = \calF^{-1} e^{it\jxi} \jap{\xi} i \pxi e^{-it\jxi} \calF.
\end{equation}
The closely related Lorentz boost is denoted by $Z = t \px + x \pt$.

We will repeatedly use the following commutator identities
\begin{equation} \label{equ:commutators}
 \begin{aligned}
  {[x, \jD^k]} &= k \jD^{k-2} \px, \quad k \in \bbZ, \\ 
  [\jD, L] &= - \px, \\
  [\jD, Z] &= - \jD^{-1} \px \pt, \\
  \bigl[ (\pt - i \jD), L \bigr] &= 0, \\
  \bigl[ (\pt - i \jD), Z \bigr] &= i \jD^{-1} \px (\pt - i\jD).
 \end{aligned}
\end{equation}

\subsection{Decay estimates for the linear Klein-Gordon evolution}

In this subsection we recall decay estimates and asymptotics for the linear Klein-Gordon evolution. 
For the convenience of the reader we provide complete proofs.

\begin{lemma} \label{lem:asymptotics_KG}
We have for all $t\ge1$ that
 \begin{equation} \label{eq:stern}
  \biggl\| \bigl( e^{it\jap{ D}} f \bigr)(x) - \frac{1}{t^\hf} e^{i\frac{\pi}{4}} e^{i\rho} \jap{\xi_0}^\thf \hat{f}(\xi_0) \one_{(-1,1)}({\textstyle \frac{x}{t}}) \biggr\|_{L^\infty_x} \leq \frac{C}{t^{\frac23}} \|\jap{x} f\|_{H^2_x}, 
 \end{equation}
 where $\rho \equiv \rho(t,x) := \sqrt{t^2 - x^2}$ and $\frac{\xi_0}{\jap{\xi_0}} = - \frac{x}{t}$, or equivalently $ \xi_0 = - \frac{x}{\rho}$, $\jap{\xi_0} = \frac{t}{\rho}$. 
\end{lemma}
\begin{proof}
We write 
\EQ{ \label{eq:osc intrep}
\big( e^{i t \jD} f \big)(x) 
& =  \frac{1}{\sqrt{2\pi}} \int_\bbR  e^{it\jap{\xi}}e^{ix\xi} \hat{f}(\xi)\, \ud\xi = \frac{1}{\sqrt{2\pi}} \int_\bbR e^{it\phi(\xi,u)} \hat{f}(\xi)\, \ud\xi 
}
with 
\begin{equation*}
 \phi(\xi,u) := \jap{\xi} + u\xi, \quad u := \frac{x}{t}, 
\end{equation*}
and note that the phase $\phi(\xi, u)$ satisfies 
\begin{equation*}
 \pxi \phi(\xi, u) = \xi \jxi^{-1} + u, \quad \pxi^2 \phi(\xi, u) = \jxi^{-3}.
\end{equation*}
Thus, if $|x|\ge t\ge1$, then 
\EQ{ \label{eq:phase der}
|\partial_\xi \phi  (\xi,u )| &= |\xi\jap{\xi}^{-1} + u | \ge 1 - |\xi|\jap{\xi}^{-1}\ge  \jap{\xi}^{-2}/2.
}
We break up the integration in \eqref{eq:osc intrep} by means of the smooth partition of unity $1 = \chi_1(\xi^2/t) + \chi_0(\xi^2/t)$, where $\chi_0(\cdot)$ is a smooth cutoff to the interval $[-1,1]$, and integrate by parts in the latter integral. 
Using \eqref{eq:phase der} yields 
\begin{equation} \label{eq:E+}
\begin{aligned}
 \bigl| \big( e^{i t \jD} f \big)(x) \bigr| &\les \int_\bbR \chi_1(\xi^2/t) |\hat{f}(\xi)|\, \ud \xi + t^{-1} \int_\bbR \frac{|\partial_\xi^2\phi (\xi,u)|}{ |\partial_\xi \phi  (\xi,u )|^2}  \chi_0(\xi^2/t)   |\hat{f}(\xi)| \, \ud \xi  \\
 &\quad + t^{-1} \int_\bbR | \partial_\xi \phi  (\xi,u )|^{-1}  \big| \partial_\xi \big( \chi_0(\xi^2/t) \hat{f}(\xi)\big) \big| \, \ud \xi  \\
 &\les t^{-\frac34} \bigl( \| \jap{\xi}^2 \hat{f}(\xi)\|_{L^2_\xi} + \| \jap{\xi}^2 \partial_\xi \hat{f}(\xi)\|_{L^2_\xi} \bigr), 
\end{aligned}
\end{equation}
which is better than~\eqref{eq:stern}. 
Now suppose $|x|< t$. The phase $\phi (\xi_0,u)$ has a unique stationary point at 
\begin{equation*}
 \xi_0=-\jap{\xi_0} u, \quad \text{or equivalently} \quad  \xi_0=-\frac{u}{\sqrt{1-u^2}}.
\end{equation*}
One has $\phi (\xi_0, u)=\sqrt{1-u^2}$ which implies $t \phi (\xi_0, u)=\sqrt{t^2-x^2}=\rho$. 

We now claim that the bound~\eqref{eq:phase der} continues to hold (up to multiplicative constants) for all $\xi\in\R\setminus I(\xi_0)$ where 
\[
I(\xi_0) := \bigl[ \xi_0 - \jap{\xi_0}/100, \xi_0 + \jap{\xi_0}/100 \bigr].
\] 
In fact,
\EQ{\label{eq:p phi}
|\partial_\xi \phi (\xi, u) | &= |\partial_\xi \phi (\xi, u)- \partial_\xi \phi (\xi_0, u)| \\ 
&= | \xi\jap{\xi}^{-1} - \xi_0\jap{\xi_0}^{-1}| \\
&= \frac{|\xi^2-\xi_0^2|}{(\jap{\xi}\jap{\xi_0})^2 | \xi\jap{\xi}^{-1} + \xi_0\jap{\xi_0}^{-1}|}.
}
Without loss of generality, we assume $\xi_0\ge0$. Then, on the one hand,  if $\xi\ge    \xi_0 + \jap{\xi_0}/100$, \eqref{eq:p phi} implies that
\[
|\partial_\xi \phi (\xi, u) |  \gtrsim \jap{\xi_0}^{-2}\gtrsim \jap{\xi}^{-2}
\]
This follows from $\xi\geq \max(1,101\xi_0)/100$, which ensures that  the absolute value in the denominator is $\simeq1$, while the numerator is $\gtrsim \jap{\xi}^2$. 
On the other hand, if  $\xi\le    \xi_0 - \jap{\xi_0}/100$ we claim that 
$
|\partial_\xi \phi (\xi, u) |  \gtrsim \jap{\xi}^{-2}
$.
Indeed, consider first the case $\xi\le0$. Then from the first line of \eqref{eq:p phi}, $|\partial_\xi \phi (\xi, u) | \ge \xi_0\jap{\xi_0}^{-1}$ which is $\gtrsim 1$ unless $0\le\xi_0\ll1$. In that latter case, however,  $\xi\le    \xi_0 - \jap{\xi_0}/100$  implies that $\xi\le -1/200$, say. In that case the first line of~\eqref{eq:p phi} implies that $|\partial_\xi \phi (\xi, u) |\gtrsim 1$. 
If $0\le\xi\le    \xi_0 - \jap{\xi_0}/100$, then $\xi_0\gtrsim 1$ and  the absolute value in the denominator is again $\simeq1$, while the numerator is $\gtrsim \jap{\xi_0}^2$. In summary, the claim holds. 
Setting $$\omega_u(\xi):= \chi_0\big(C_0(\xi-\xi_0)\jap{\xi_0}^{-1}\big)$$ for some large constant $C_0$, and repeating the arguments leading to~\eqref{eq:E+} therefore yields 
\EQ{ \label{eq:zwisch}
\bigg|  \big( e^{i t \jD}f\big)(x) - \frac{1}{\sqrt{2\pi}} \int_\bbR  e^{it\phi(\xi,u)}\omega_u(\xi) \hat{f}(\xi)\, \ud\xi  \biggr| &\les t^{-\frac34} \bigl( \| \jap{\xi}^2 \hat{f}(\xi)\|_{L^2_\xi} + \| \jap{\xi}^2 \partial_\xi \hat{f}(\xi)\|_{L^2_\xi} \bigr),
}
which holds uniformly in $t\ge1$ and $x\in\R$. To analyze the main term here, which we denote by $\Psi(t)f$, we write
\EQ{\nn
\phi (\xi,u) - \phi (\xi_0,u) &= \jap{\xi} + u\xi - \jap{\xi_0} - u\xi_0 = \frac{(\xi-\xi_0)^2}{\jap{\xi_0} (1+\xi\xi_0+\jap{\xi}\jap{\xi_0})}=:\eta^2.
}
The change of variables $\xi\mapsto\eta$ is a diffeomorphism on the support of $\omega_u(\xi)$ given by 
\EQ{\nn
\eta &= \frac{\xi-\xi_0}{\sqrt{\jap{\xi_0} (1+\xi\xi_0+\jap{\xi}\jap{\xi_0})} }, \quad  \frac{d\eta}{d\xi} \simeq \jap{\xi}^{-\frac32},\quad  \frac{d\eta}{d\xi}(\xi_0)= \jap{\xi_0}^{-\frac32}/\sqrt{2}.
}
Hence we have 
\begin{equation*} \label{eq:G}
(\Psi(t)f)(x) = \frac{e^{i\rho}}{\sqrt{2\pi}} \int_\bbR e^{it\eta^2}\; \ol{G}(\eta;t,x) \, \ud \eta = \frac{e^{i\rho}}{\sqrt{2\pi}} \frac{e^{i\frac{\pi}{4}}}{\sqrt{2t}} \int_\bbR e^{-i\frac{y^2}{4t}}\; \ol{\widehat{G}(y;t,x)} \, \ud y,
\end{equation*}
where we write 
\begin{equation*}
 \ol{G}(\eta;t,x) = \omega_u(\xi)\hat{f}(\xi) \frac{d\xi}{d\eta}.
\end{equation*}
Note that
\begin{equation*}
\frac{1}{\sqrt{2\pi}} \int_\bbR  \ol{\widehat{G}(y;t,x)} \, \ud  y = \ol{G}(0;t,x) = \omega_u(\xi_0)\hat{f}(\xi_0) \frac{d\xi}{d\eta}(\xi_0) = \sqrt{2} \jap{\xi_0}^{\frac32}\hat{f}(\xi_0) \label{eq:rt 2},
\end{equation*}
which further implies  
\EQ{ \label{eq:schr2}
 (\Psi(t)f)(x) &= \frac{e^{i\rho}e^{i\frac{\pi}{4}}}{t^\hf}\jap{\xi_0}^{\frac32}\hat{f}(\xi_0)  + \calO_{L^\infty_x} \biggl( t^{-\frac12} \int_\bbR \big| e^{-i\frac{y^2}{4t}} -1 \big| \, |{\widehat{G}(y;t,x)}  |\, \ud y \biggr).
}
The integral in the last line is estimated as follows
\EQ{\nn
 \int_\bbR \big| e^{-i\frac{y^2}{4t}} -1 \big| \, \bigl|{\widehat{G}(y;t,x)}  \bigr| \, \ud y &\quad \lesssim t^{-\frac12}\int_{\{|y|^2\le t\}} |y| \, \bigl| {\widehat{G}(y;t,x)} \bigr| \, \ud y  + \int_{\{ |y|^2\ge t \}} \, \bigl| {\widehat{G}(y;t,x)} \bigr| \, \ud y  \\
 &\quad \lesssim t^{-\frac14} \bigl\| y \, {\widehat{G}(y;t,x)} \bigr\|_{L^2_y} \\
 &\quad \lesssim t^{-\frac14} \bigl\|\partial_\eta {G(\eta;t,x)} \bigr\|_{L^2_\eta}.
}
By definition, 
\EQ{\nn 
\int_\bbR \big|\partial_\eta {{G}(\eta;t,x)}\big|^2\, d\eta &= \int_\bbR \Big|\frac{d\xi}{d\eta}\Big| \, \Big| \partial_\xi \Big( \omega_u(\xi)\hat{f}(\xi) \frac{d\xi}{d\eta} \Big) \Big|^2 \, \ud \xi \lesssim  \int_\bbR \Big| \partial_\xi \Big( \omega_u(\xi)\hat{f}(\xi) \frac{d\xi}{d\eta} \Big) \Big|^2 \jap{\xi}^{\frac32} \,\ud \xi.
}
Now we note that by complex interpolation of the preceding bound with 
\begin{equation*}
 \int_\bbR \big| {{G}(\eta;t,x)}\big|^2 \, d\eta \lesssim \int_\bbR \Big| \omega_u(\xi)\hat{f}(\xi) \frac{d\xi}{d\eta} \Big|^2 \jap{\xi}^{-\frac32} \, \ud \xi,
\end{equation*}
we obtain that for all $\frac12<\beta\le1$, 
\begin{equation*}
 \begin{aligned}
   \int_\bbR \big| e^{-i\frac{y^2}{4t}} -1 \big| \, \bigl|{\widehat{G}(y;t,x)} \bigr| \, \ud y &\quad \lesssim t^{\frac14-\frac\beta2} \big\| |y|^{\beta} \, {\widehat{G}(y;t,x)}  \big\|_{L^2_y} \\
   &\quad \lesssim t^{\frac14-\frac\beta2} \bigl\| \bigl(-\partial_\eta^2\bigr)^{\frac{\beta}{2}} {G(\eta;t,x)} \bigr\|_{L^2_\eta}  \\
   &\quad \lesssim t^{\frac14-\frac\beta2} \biggl( \int_\bbR \Big|(-\partial_\xi^2)^{\frac{\beta}{2}} \Big( \omega_u(\xi)\hat{f} (\xi) \frac{d\xi}{d\eta} \Big) \Big|^2 \jap{\xi}^{-\frac32+3\beta} \,\ud \xi \bigg)^\hf.  
 \end{aligned}
\end{equation*}
On the one hand,
\begin{equation*}
 \bigg( \int_\bbR \Big| \omega_u(\xi)\hat{f}(\xi) \frac{d\xi}{d\eta} \Big|^2 \, \ud \xi \Big)^{\frac12} \lesssim \bigg( \int_\bbR |\hat{f}(\xi)  |^2 \jap{\xi}^3 \, \ud \xi \bigg)^{\frac12}
\end{equation*}
and, on the other hand, 
\EQ{\nn
 \bigg( \int_\bbR  \Big| \bigl(-\partial_\xi^2)^{\frac{1}{2}} \Big( \omega_u(\xi)\hat{f}(\xi) \frac{d\xi}{d\eta} \Big) \Big|^2 \, \ud \xi \bigg)^{\frac12}
&\lesssim \bigg( \int_\bbR \big( |\hat{f}(\xi) |^2 \jap{\xi} + |\partial_\xi \hat{f}(\xi) |^2 \jap{\xi}^3 \big) \, \ud \xi \bigg)^{\frac12}.
}
In conclusion, we can bound with $\beta=\frac56$, 
\EQ{\label{eq:schr3}
  \int_\bbR \big| e^{-i\frac{y^2}{4t}} -1 \big| \, \big|{\widehat{G}(y;t,x)} \big| \, \ud y &\lesssim t^{-\frac16} \bigg( \int_\bbR \bigl( |\hat{f}(\xi) |^2 + |\partial_\xi \hat{f}(\xi)|^2 \big) \jap{\xi}^4 \, \ud \xi \bigg)^{\frac12}.
}
Combining \eqref{eq:E+}, \eqref{eq:zwisch}, \eqref{eq:schr2}, and \eqref{eq:schr3} yields 
\EQ{\label{eq:schr4}
 \biggl| \bigl( e^{it\jap{ D}} f \bigr)(x) - \frac{1}{t^\hf} e^{i\frac{\pi}{4}} e^{i\rho} \jap{\xi_0}^\thf \hat{f}(\xi_0) \one_{(-1,1)}({\textstyle \frac{x}{t}}) \biggr| \lesssim t^{-\frac23} \bigl( \| \jap{\xi}^2 \hat{f}(\xi)\|_{L^2_\xi} + \| \jap{\xi}^2 \partial_\xi \hat{f}(\xi)\|_{L^2_\xi} \bigr),
}
which holds uniformly in $t\ge1$ and $x\in\R$.
\end{proof}

Next, we establish a pointwise bound on the evolution for all energies. 

\begin{lemma} \label{lem:pw decay}
Fix $\mu > 0$. Then we have for all $t > 0$ that
\begin{equation} \label{equ:dispersive_decay_propagator}
  \bigl\| e^{i t\jap{ D}} f \bigr\|_{L^\infty_x} \le \frac{C(\mu)}{t^{\hf}} \bigl\| \jap{D}^{\thf+\mu} f \bigr\|_{L^1_x}.
\end{equation}
\end{lemma}
\begin{proof} 
Let $\chi$ be a bump function compactly supported on $\R\setminus\{0\}$ and fix any $\lambda\ge 1$. Consider the evolution 
\begin{equation} \label{eq:Phitpm}
 \begin{aligned}
  \big( e^{i t \jap{D}}  \chi(D/\lambda) f \big)(x) = \frac{1}{\sqrt{2\pi}} \int_\bbR  e^{i(t\jap{\xi}+x\xi)} \chi(\xi/\lambda) \hat{f}(\xi)\, \ud\xi = \int_\bbR K_\lambda(t,x-y) f(y) \, \ud y
 \end{aligned}
\end{equation}
with 
\begin{equation*}
 K_\lambda(t,x) = \int_\bbR e^{it( \jap{\xi}+\xi x/t)} \chi(\xi/\lambda) \, \ud\xi = {\lambda} \int_\bbR e^{i\lambda t (\lambda^{-1} \jap{\lambda \xi} + \xi x/t)} \chi(\xi) \, \ud \xi.
\end{equation*}
We have the bound $|K_\lambda(t,x)|\le C\lambda$ uniformly in $x\in\R$, $t>0$, and $\lambda \geq 1$. 
If $t\ge\lambda$, then we claim the stronger bound 
\begin{equation} \label{eq:claim K}
|K_\lambda(t,x)|\le C\lambda^{\frac32} t^{-\frac12}.
\end{equation}
We write  
\begin{equation} \label{eq:Ks}
K_\lambda(t,x) = {\lambda} \int_\bbR e^{is\varphi_\lambda(\xi; t,x)} \chi(\xi) \, \ud\xi  
\end{equation}
with $s:=\lambda^{-1}  t $ and phase $\varphi_\lambda(\xi; t,x) := \lambda^2 (\lambda^{-1} \jap{\lambda \xi} + \xi x/t)$.
Then 
\begin{equation} \label{eq:derphi} 
 \begin{aligned}
  \partial_\xi \varphi_\lambda(\xi; t,x) &=  \lambda^2 \Bigl( \frac{\lambda\xi}{\jap{\lambda\xi}} + \frac{x}{t} \Bigr), \\
  \partial_\xi^2 \varphi_\lambda(\xi; t,x) &=  \frac{\lambda^3}{\jap{\lambda\xi}^3} \simeq 1, \\
  |\partial_\xi^3 \varphi_\lambda(\xi; t,x)| &=  3\frac{\lambda^5|\xi|}{\jap{\lambda\xi}^5} \simeq 1,
 \end{aligned}
\end{equation}
on the support $I_0\subset[-\xi_2,-\xi_1]\cup [\xi_1,\xi_2]\subset \R\setminus\{0\}$ of $\chi$ (recall $\lambda\ge1$).  Without loss of generality, we assume $I_0\subset [\xi_1,\xi_2]$, the reflected part being symmetric. We distinguish the following two cases, for fixed $x,t,\lambda$ as above: 
\begin{itemize}
\item[(a)] $\min | \partial_\xi \varphi_\lambda(\xi; t,x)|\gtrsim s^{-\frac12}$ on $I_0$,
\item[(b)] $\min | \partial_\xi \varphi_\lambda(\xi; t,x)|\ll s^{-\frac12}$ on $I_0$.
\end{itemize}
In case (a), we deduce from the second derivative in  \eqref{eq:derphi} that 
\[
 | \partial_\xi \varphi_\lambda(\xi; t,x)|\gtrsim s^{-\frac12} + \min \bigl\{ |\xi-\xi_1|, |\xi-\xi_2| \bigr\} \quad \forall\; \xi \in I_0.
\]
Integrating by parts once in \eqref{eq:Ks} yields 
\begin{equation*}
 |K_\lambda(t,x)| \le C\lambda s^{-1} \int_{I_0} \biggl( \frac{|\partial_\xi^2 \varphi_\lambda(\xi; t,x)|}{(\partial_\xi \varphi_\lambda(\xi; t,x))^2} +  \frac{1}{|\partial_\xi \varphi_\lambda(\xi; t,x)|} \biggr) \, \ud \xi \le C\lambda s^{-\frac12},
\end{equation*}
as claimed by \eqref{eq:claim K}. On the other hand, in case (b) suppose the minimum of $\min | \partial_\xi \varphi_\lambda(\xi; t,x)|$ is attained at $\xi_*\in I_0$. 
Then we infer from  the second derivative that 
\[
| \partial_\xi \varphi_\lambda(\xi; t,x)|\gtrsim |\xi-\xi_*| \text{\ \ on\ \ } \xi\in I_0, \; |\xi-\xi_*|\ge s^{-\frac12}
\]
Let $\psi$ be a smooth bump function that equals $1$ on $[-1,1]$. Then with $\calL:= \frac{1}{i\partial_\xi \varphi_\lambda}\partial_\xi$, 
we write
\begin{equation*}
 \begin{aligned}
  |K_\lambda(t,x)| & \le  {\lambda} \bigg| \int_{0}^\infty   e^{is\varphi_\lambda(\xi; t,x)} \chi(\xi) \psi((\xi-\xi_*)s^{\frac12})\, \ud\xi \bigg|  \\
  &\quad + {\lambda}s^{-2} \bigg| \int_{0}^\infty e^{is\varphi_\lambda(\xi; t,x)} (\calL^*)^2 \Bigl( \chi(\xi) \bigl( 1 - \psi((\xi-\xi_*)s^{\frac12}) \bigr) \Bigr) \, \ud\xi \bigg| \\
  &\lesssim \lambda s^{-\frac12} + \lambda s^{-2}  \int_{I_0} \one_{[|\xi-\xi_*|\ge s^{-\frac12}]} \Big( |\xi-\xi_*|^{-4}+ |\xi-\xi_*|^{-2} s\Big)\, \ud\xi \\
  &\lesssim  \lambda s^{-\frac12},  
 \end{aligned}
\end{equation*}
which establishes the claim \eqref{eq:claim K}.  In summary, for all $\lambda\ge1$ and $t>0$, 
\begin{equation} \label{eq:awayf0}
\|e^{i t\jap{ D}}  \chi(D/\lambda) f\|_{L^\infty_x} \le C t^{-\frac12} \lambda^{\frac32} \|f\|_{L^1_x}.
\end{equation}
Let $\chi_0$ be a bump function supported near $0$. Then by essentially the same analysis as above (albeit with $s=t$ and $\lambda=1$) we obtain 
\begin{equation} \label{eq:dispt}
\|e^{i t\jap{ D}}  \chi_0(D) f\|_{L^\infty_x} \le C t^{-\frac12} \|f\|_{L^1_x}.
\end{equation}
Performing a dyadic decomposition of energies, and adding up all contributions from~\eqref{eq:awayf0} and~\eqref{eq:dispt} yields
\begin{equation} \label{eq:LPdec}
 \begin{aligned}
  \| e^{i t\jap{ D}} f \|_{L^\infty_x} &\le C t^{-\hf} \Big( \|\chi_0(D) f\|_{L^1_x} + \sum_{j=0}^\infty 2^{3j/2} \| \chi(D/2^{j}) f\|_{L^1_x} \Big) \\
  &=  C t^{-\hf} \Big( \|\chi_0(D)f\|_{L^1_x} + \sum_{j=0}^\infty 2^{-j\mu} \| \psi_j(D)  |D|^{\frac32+\mu}  f\|_{L^1_x} \Big)
 \end{aligned}
\end{equation}
with $\mu > 0$ arbitrary and 
\begin{equation*}
\psi_j(D) :=  2^{(\thf+\mu)j} | D|^{-\frac32-\mu} \chi(D/2^{j}), \quad j \geq 0.
\end{equation*}
Summing up~\eqref{eq:LPdec} will complete the proof provided we have the operator bounds
\begin{equation} \label{eq:multop}
\|\chi_0(D) f\|_{L^1_x} \lesssim \|f\|_{L^1_x}, \quad \sup_{j\ge0} \| \psi_j(D) f\|_{L^1_x} \lesssim \| f \|_{L^1_x},
\end{equation}
as desired. 
\end{proof}

Finally, we derive local decay estimates for the linear Klein-Gordon evolution.

\begin{lemma} \label{lem:loc L2}
 Let $a >\frac12$ and $b \geq 0$. We have uniformly for all $t \in \bbR$ that 
 \begin{align}
  \bigl\| \jap{x}^{-a} \jap{D}^{-b} e^{i t \jap{D}} \jap{x}^{-a} \bigr\|_{L^2_x \to L^2_x} &\lesssim \frac{1}{\jap{t}^{\frac{1}{2}}}, \label{equ:local_decay} \\
  \bigl\| \jap{x}^{-1-a} \partial_x \jD^{-1} e^{i t \jap{D}} \jap{x}^{-1-a} \bigr\|_{L^2_x \to L^2_x} &\lesssim \frac{1}{\jap{t}^{\frac{3}{2}}}, \label{equ:local_decay_van1} \\
  \bigl\| \jap{x}^{-1} \partial_x \jD^{-1} e^{i t \jap{D}} \jap{x}^{-1} \bigr\|_{L^2_x \to L^2_x} &\lesssim \frac{1}{\jap{t}}, \label{equ:local_decay_van2} \\
  \bigl\| \jap{x}^{-1} (-1+\jD) \jD^{-1} e^{i t \jap{D}} \jap{x}^{-1} \bigr\|_{L^2_x \to L^2_x} &\lesssim \frac{1}{\jap{t}}. \label{equ:local_decay_van3}
 \end{align}
\end{lemma}
\begin{proof}
By unitarity of the flow, it suffices in all of these estimates to take $t \geq 1$. We begin with the proof of~\eqref{equ:local_decay}. 
Let $\chi_0(D)$ be a smooth cutoff to frequencies in $[-1,1]$, say. Since $a > \hf$ we obtain from Lemma~\ref{lem:pw decay} that
\EQ{\nn 
   \bigl\| \jap{x}^{-a} \jap{D}^{-b} e^{i t \jap{D}} \chi_0(D) \jap{x}^{-a} f \bigr\|_{L^2_x} &\lesssim t^{-\frac12} \| \jap{D}^2 \chi_0(D) \jap{x}^{-a} f \bigr\|_{L^1_x} \lesssim t^{-\frac12} \|f\|_{L^2_x}.
}
On the other hand, with $\chi_1=1-\chi_0$, 
\EQ{\label{eq:ibp}
   \jap{D}^{-b}  e^{ i t \jap{D}} \chi_1(D) f (x) &=\frac{1}{\sqrt{2\pi}} \int_\R e^{ i t \jap{\xi}}e^{ix\xi}  \jap{\xi}^{-b} \chi_1(\xi) \hat{f}(\xi)\, \ud \xi \\
& =  \frac{i}{\sqrt{2\pi}\, t} \int_\R e^{ i t \jap{\xi}} \partial_\xi \Bigl( \frac{\jap{\xi}}{\xi} e^{ix\xi}  \jap{\xi}^{-b} \chi_1(\xi) \hat{f}(\xi) \Bigr) \, \ud \xi. 
}
By inspection, 
\[
 \bigl\| \jap{x}^{-1}   \jap{D}^{-b}  e^{ i t \jap{D}} \chi_1(D) f \bigr\|_{L^2_x} \lesssim t^{-1} \|\hat{f}\|_{H^1_\xi} \simeq t^{-1} \| \jap{x} f\|_{L^2_x},
\]
whence by complex interpolation 
\[
 \bigl\| \jap{x}^{-\frac12}   \jap{D}^{-b}  e^{ i t \jap{D}} \chi_1(D) f \bigr\|_{L^2_x} \lesssim t^{-\frac12} \| \jap{x}^{\frac12} f\|_{L^2_x}.
\]
Next, we prove~\eqref{equ:local_decay_van1}. For the contribution of the small frequencies to  \eqref{equ:local_decay_van1}, we write as in \eqref{eq:ibp} 
\begin{equation} \label{equ:local_decay_proof1}
\begin{aligned}
 D \jap{D}^{-1}  e^{ i t \jap{D}} \chi_0(D) f (x) &=\frac{1}{\sqrt{2\pi}} \int_\R e^{ i t \jap{\xi}}e^{ix\xi} \xi  \jap{\xi}^{-1} \chi_0(\xi) \hat{f}(\xi)\, \ud \xi \\
 & =  \frac{i}{\sqrt{2\pi}t} \int_\R e^{ i t \jap{\xi}} \partial_\xi \bigl(  e^{ix\xi}  \chi_0(\xi) \hat{f}(\xi) \bigr) \, \ud \xi.
\end{aligned}
\end{equation}
Applying Lemma~\ref{lem:pw decay} as before to the right-hand side implies  \eqref{equ:local_decay_van1} with a $\chi_0(D)$ inserted. Note that $i\partial_\xi\hat{f}(\xi)=\widehat{xf}(\xi)$ leads to a $\frac32+$ weight by Cauchy-Schwarz as stated.  On the other hand, for the large frequencies we insert $\chi_1(D)$ into this expression and apply integration by parts twice as in~\eqref{eq:ibp}. 

Finally, the estimates \eqref{equ:local_decay_van2} and \eqref{equ:local_decay_van3} follow by integration by parts as in~\eqref{equ:local_decay_proof1}, exploiting the vanishing of the symbol $D$, respectively of $-1+\jD$, at zero frequency.
\end{proof}

\section{Setting up the analysis} \label{sec:setting_up}

\subsection{Evolution equation for a perturbation of the sine-Gordon kink} \label{subsec:evol_equ_perturbation}

The goal of this subsection is to derive an evolution equation for odd perturbations of the static sine-Gordon kink. Recall that the equation of motion for the scalar field $\phi(t,x)$ in the sine-Gordon model is given by
\begin{equation} \label{equ:evol_equ_equation_of_motion}
 (\pt^2 - \px^2) \phi = - W'(\phi), \quad (t,x) \in \bbR \times \bbR, 
\end{equation}
where 
\begin{align*}
 W(\phi) = 1 - \cos(\phi).
\end{align*}
In what follows we consider small odd perturbations of the sine-Gordon kink 
\begin{equation*}
 K(x) = 4 \arctan(e^x)
\end{equation*}
in the sense that we decompose the scalar field as
\begin{equation} \label{equ:evol_equ_ansatz}
 \phi(t,x) = K(x) + u(t,x).
\end{equation}
By Taylor expansion we have 
\begin{equation} \label{equ:evol_equ_taylor_expansion}
  - W'(K+u) = - W'(K) - \sum_{k=1}^3 \frac{1}{k!} W^{(k+1)}(K) u^k + R_1(u) + R_2(u),
\end{equation}
where we use the short-hand notation
\begin{align*}
 R_1(u) &= - \frac{1}{4!} W^{(5)}(K) u^4, \\
 R_2(u) &= - \frac{1}{4!} \biggl( \int_0^1 (1-r)^4 W^{(6)}(K+ru) \, \ud r \biggr) u^5.
\end{align*}
Inserting the decomposition~\eqref{equ:evol_equ_ansatz} and the expansion~\eqref{equ:evol_equ_taylor_expansion} into the equation of motion~\eqref{equ:evol_equ_equation_of_motion} for the scalar field, and using that the static kink satisfies $-\px^2 K = - W'(K)$, we obtain the following evolution equation for the perturbation 
\begin{equation}
 \bigl( \pt^2 - \px^2 + W''(K) \bigr) u = - \frac12 W^{(3)}(K) u^2 - \frac16 W^{(4)}(K) u^3 + R_1(u) + R_2(u),
\end{equation}
or equivalently,
\begin{equation} \label{equ:evol_equ_for_u}
 \bigl(\pt^2 - \px^2 + \cos(K) \bigr) u = \frac12 \sin(K) u^2 + \frac16 \cos(K) u^3 + R_1(u) + R_2(u),
\end{equation}
where we have 
\begin{align*}
 R_1(u) &= - \frac{1}{4!} \sin(K) u^4, \\
 R_2(u) &= - \frac{1}{4!} \biggl( \int_0^1 (1-r)^4 \cos(K+ru) \, \ud r \biggr) u^5.
\end{align*}
Finally, observing that
\begin{equation} \label{equ:evol_equ_cosKsinK}
 \begin{aligned}
  \cos(K) &= 1 - 2 \sech^2(x), \\
  \sin(K) &= - 2 \sech(x) \tanh(x),
 \end{aligned}
\end{equation}
we may write~\eqref{equ:evol_equ_for_u} more explicitly as
\begin{equation} \label{equ:evol_equ_explicit_for_u}
 \begin{aligned}
  \bigl( \pt^2 - \px^2 - 2 \sech^2(x) + 1 \bigr) u = - \sech(x) \tanh(x) u^2 + \frac16 u^3 - \frac13 \sech^2(x) u^3 + R_1(u) + R_2(u).
 \end{aligned}
\end{equation}
Only the quadratic and the cubic nonlinearities require a careful treatment in the study of the long-time behavior of small solutions to~\eqref{equ:evol_equ_explicit_for_u}. We will see that the spatial localization of the quartic nonlinearities $R_1(u)$ allows for a particularly simple analysis of their contributions, and the quintic remainder terms $R_2(u)$ can then be dealt with in a crude manner.

\subsection{Super-symmetric factorization and the transformed equation} \label{subsec:transformed_equation} 

The linearized operator in~\eqref{equ:evol_equ_explicit_for_u} admits the factorization
\begin{equation} \label{equ:evol_equ_factorization_lin_op}
 \calD \calD^\ast = -\px^2 - 2 \sech^2(x) + 1
\end{equation}
in terms of the first-order differential operator $\calD$ and its adjoint $\calD^\ast$ given by
\begin{align*}
 \calD := \px - \tanh(x), \quad \calD^\ast &:= -\px - \tanh(x).
\end{align*}
It turns out that the conjugate operator to~\eqref{equ:evol_equ_factorization_lin_op} is just the flat linear operator
\begin{equation} \label{equ:evol_equ_super_symm_partner}
 \calD^\ast \calD = -\px^2 + 1.
\end{equation}
Upon differentiating the Klein-Gordon equation~\eqref{equ:evol_equ_for_u} by $\calD^\ast$, we therefore find that the dependent variable $\calD^\ast u$ satisfies the following nonlinear equation 
\begin{equation} \label{equ:evol_equ_Dastu}
 \begin{aligned}
  (\pt^2 - \px^2 + 1) (\calD^\ast u) &= \calD^\ast \Bigl( \frac12 \sin(K) u^2 \Bigr) + \calD^\ast \Bigl( \frac16 \cos(K) u^3 \Bigr) + \calD^\ast \bigl( R_1(u) \bigr) + \calD^\ast \bigl( R_2(u) \bigr),
 \end{aligned}
\end{equation}
which just features the flat linear Klein-Gordon operator on the left-hand side.
In the remainder of this subsection we rewrite~\eqref{equ:evol_equ_Dastu} as a nonlinear Klein-Gordon equation for the new dependent variable 
\begin{equation} \label{equ:def_new_variable_w}
 w(t,x) := (\calD^\ast u)(t,x).
\end{equation}
Observe that $w(t,x)$ is even since $u(t,x)$ is odd.
To this end we first need to detail how to pass back and forth between the variables $u$ and $w$. 
The linearized operator around the sine-Gordon kink has the even zero eigenfunction 
\begin{equation*}
 Y(x) = \sech(x).
\end{equation*}
Indeed, one readily verifies that $\calD^\ast Y = 0$. Correspondingly, the integral operator 
\begin{equation} \label{equ:def_calI}
 \begin{aligned}
  \calI[g](x) &:= - Y(x) \int_0^x (Y(y))^{-1} g(y) \, \ud y = - \sech(x) \int_0^x \cosh(y) g(y) \, \ud y
 \end{aligned}
\end{equation}
is a right-inverse operator for $\calD^\ast$, i.e.,
\begin{equation*}
 \calD^\ast \bigl( \calI[g] \bigr) = g.
\end{equation*}
We will occasionally use that integration by parts in the definition of $\calI[g]$ gives the identity
\begin{equation*}
 \calI[g](x) = - \tanh(x) g(x) + \wtcalI[ \px g ](x),
\end{equation*}
where 
\begin{equation} \label{equ:def_wtcalI}
 \wtcalI[ \px g ](x) := \sech(x) \int_0^x \sinh(y) (\partial_y g)(y) \, \ud y.
\end{equation}
Moreover, integrating by parts in the integral expression $\calI\bigl[ \calD^\ast g \bigr]$, we obtain for any sufficiently regular function $g(x)$ that
\begin{align*}
 g(x) = \calI\bigl[ \calD^\ast g \bigr](x) + g(0) Y(x).
\end{align*}
Since in this work we only consider \emph{odd} perturbations $u(t,x)$, whence $u(t,0) = 0$, we can simply express $u(t,x)$ in terms of the new variable $w(t,x)$ via
\begin{equation} \label{equ:evol_equ_uIw}
 u(t,x) = \calI\bigl[w(t)](x).
\end{equation}
Inserting the preceding relation~\eqref{equ:evol_equ_uIw} into \eqref{equ:evol_equ_Dastu}, we now pass to the following nonlinear Klein-Gordon equation for the new variable $w$,
\begin{equation} \label{equ:nlkg_for_w}
 (\pt^2 - \px^2 + 1) w = \calQ(w) + \calC(w) + \calR_1(w) + \calR_2(w),
\end{equation}
with initial data 
\begin{equation*}
 (w, \pt w)|_{t=0} = (w_0, w_1) := (\calD^\ast u_0, \calD^\ast u_1),
\end{equation*}
and where 
\begin{equation} \label{equ:nonlinearities_in_w_equation_definition}
\begin{aligned}
 \calQ(w) &:= \calD^\ast \Bigl( \frac12 \sin(K) u^2 \Bigr), \\
 \calC(w) &:= \calD^\ast \Bigl( \frac16 \cos(K) u^3 \Bigr), \\
 \calR_1(w) &:= \calD^\ast \bigl( R_1(u) \bigr), \\
 \calR_2(w) &:= \calD^\ast \bigl( R_2(u) \bigr).
\end{aligned}
\end{equation}
In the remainder of this subsection we use~\eqref{equ:evol_equ_uIw} to express the nonlinearities~\eqref{equ:nonlinearities_in_w_equation_definition} in terms of $w$.

\subsubsection{Transformed quadratic nonlinearity} 

We begin by computing that
\begin{align*}
 \calD^\ast \bigl( \sin(K) u^2 \bigr) &= \bigl( - \px - \tanh(x) \bigr) \bigl( \sin(K) u^2 \bigr) \\
 &= - (\px K) \cos(K) u^2 + \sin(K) 2 u (-\px u) - \tanh(x) \sin(K) u^2 \\
 &= \bigl( -(\px K) \cos(K) + \tanh(x) \sin(K) \bigr) u^2 + 2 \sin(K) u (\calD^\ast u).
\end{align*}
In view of~\eqref{equ:evol_equ_cosKsinK} and the fact that $(\px K)(x) = 2 \sech(x)$, it follows that
\begin{align*}
 \calD^\ast \biggl( \frac12 \sin(K) u^2 \biggr) 
 &= \bigl( -2 \sech(x) + 3 \sech^3(x) \bigr) u^2 - 2 \sech(x) \tanh(x) u (\calD^\ast u).
\end{align*}
Passing to the new variable $w = \calD^\ast u$ and using that $u = \calI[w]$, we obtain 
\begin{equation*}
 \calD^\ast \biggl( \frac12 \sin(K) u^2 \biggr) = \bigl( -2 \sech(x) + 3 \sech^3(x) \bigr) \bigl( \calI[w] \bigr)^2 - 2 \sech(x) \tanh(x) \calI[w] w.
\end{equation*}
Since all coefficients on the right-hand side of the preceding line are spatially localized, it is useful to insert the relation
\begin{align*}
 \calI[w] = - \tanh(x) w + \wtcalI[\partial_x w],
\end{align*}
so that we can exploit the expected improved local decay of $\px w$. We find that
\begin{align*}
 \calD^\ast \biggl( \frac12 \sin(K) u^2 \biggr)  &= \bigl( -2 \sech(x) + 3 \sech^3(x) \bigr) \bigl( - \tanh(x) w + \wtcalI[\px w] \bigr)^2 \\
 &\quad \quad \quad - 2 \sech(x) \tanh(x) \bigl( - \tanh(x) w + \wtcalI[\px w] \bigr) w \\
 &= 3 \sech^3(x) \tanh^2(x) w^2 + \bigl( 2 \sech(x) - 6 \sech^3(x) \bigr) \tanh(x) \wtcalI[\px w] w \\
 &\quad \quad \quad + \bigl( -2 \sech(x) + 3 \sech^3(x) \bigr) \bigl( \wtcalI[\px w] \bigr)^2.
\end{align*}
In conclusion, we obtain 
\begin{align*}
 \calQ(w) = \calD^\ast \biggl( \frac12 \sin(K) u^2 \biggr) = \calQ_1(w) + \calQ_2(w) + \calQ_3(w),
\end{align*}
where we set
\begin{equation} \label{equ:def_Q_1to3}
\begin{aligned}
 \calQ_1(w) &:= \alpha_1(x) w^2, \\
 \calQ_2(w) &:= \alpha_2(x) \wtcalI[\px w] w, \\
 \calQ_3(w) &:= \alpha_3(x) \bigl( \wtcalI[\px w] \bigr)^2,
\end{aligned}
\end{equation}
for spatially localized coefficients $\alpha_1, \alpha_2, \alpha_3 \in \calS(\bbR)$ that are explicitly given by
\begin{align*}
 \alpha_1(x) &:= 3 \sech^3(x) \tanh^2(x), \\
 \alpha_2(x) &:= \bigl( 2 \sech(x) - 6 \sech^3(x) \bigr) \tanh(x), \\
 \alpha_3(x) &:= -2 \sech(x) + 3 \sech^3(x).
\end{align*}

\subsubsection{Transformed cubic nonlinearity}

Passing to the new variable $w = \calD^\ast u$ and using that $u = \calI[w]$, we first compute 
\begin{align*}
 \calD^\ast \bigl( u^3 \bigr) &= 3 u^2 (\calD^\ast u) + 2 \tanh(x) u^3 = 3 \bigl( \calI[w] \bigr)^2 w + 2 \tanh(x) \bigl( \calI[w] \bigr)^3 
\end{align*}
as well as 
\begin{align*}
 \calD^\ast \bigl( \sech^2(x) u^3 \bigr) &= 4 \sech^2(x) \tanh(x) u^3 + 3 \sech^2(x) u^2 (\calD^\ast u) \\
 &= 4 \sech^2(x) \tanh(x) \bigl( \calI[w] \bigr)^3 + 3 \sech^2(x) \bigl( \calI[w] \bigr)^2 w.
\end{align*}
Thus, we find that
\begin{align*}
 \calC(w) = \calD^\ast \Bigl( \frac16 \cos(K) u^3 \Bigr) &=  \frac16 \calD^\ast \bigl( u^3 \bigr) - \frac13 \calD^\ast \bigl( \sech^2(x) u^3 \bigr) \\
 &= \frac12 \bigl( \calI[w] \bigr)^2 w + \frac13 \tanh(x) \bigl( \calI[w] \bigr)^3 \\
 &\quad - \frac43 \sech^2(x) \tanh(x) \bigl( \calI[w] \bigr)^3 - \sech^2(x) \bigl( \calI[w] \bigr)^2 w.
\end{align*}
In order to distinguish those parts of the cubic nonlinearities that exhibit obvious spatial localization and those that do not, in what follows we will use the notation
\begin{equation*}
 \calC(w) = \calC_{nl}(w) + \calC_l(w),
\end{equation*}
where 
\begin{align}
 \calC_{nl}(w) &:= \frac12 \bigl( \calI[w] \bigr)^2 w + \frac13 \tanh(x) \bigl( \calI[w] \bigr)^3, \label{equ:def_nonlinearities_Cnl} \\
 \calC_l(w) &:= - \frac43 \sech^2(x) \tanh(x) \bigl( \calI[w] \bigr)^3 - \sech^2(x) \bigl( \calI[w] \bigr)^2 w. \label{equ:def_nonlinearities_Cl}
\end{align}

\subsubsection{Transformed quartic nonlinearity}

Here we compute
\begin{align*}
 \calD^\ast \bigl( \sin(K) u^4 \bigr) &= \bigl( -(\px K) \cos(K) + 3 \tanh(x) \sin(K) \bigr) u^4 + 4 \sin(K) u^3 (\calD^\ast u) \\
 &= \bigl( - 8 \sech(x) + 10 \sech^3(x) \bigr) \bigl( \calI[w] \bigr)^4 - 8 \sech(x) \tanh(x) \bigl( \calI[w] \bigr)^3 w.
\end{align*}
Correspondingly, we arrive at the expression
\begin{equation} \label{equ:def_nonlinearities_quartic}
\begin{aligned}
 \calR_1(w) &= \calD^\ast \Bigl( - \frac{1}{4!} \sin(K) u^4 \Bigr) \\
 &= \frac{1}{12} \bigl( 4 \sech(x) -5 \sech^3(x) \bigr) \bigl( \calI[w] \bigr)^4 + \frac13 \sech(x) \tanh(x) \bigl( \calI[w] \bigr)^3 w.
\end{aligned}
\end{equation}

\subsubsection{Transformed quintic nonlinearity}

Finally, analogous computations as in the preceding subsections yield that the quintic remainder term can be written as
\begin{equation} \label{equ:def_nonlinearities_quintic}
 \begin{aligned}
 \calR_2(w) &= \calD^\ast \bigl( R_2(u) \bigr) \\
 &= - \frac{2}{4!} \sech(x) \biggl( \int_0^1 (1-r)^4 \sin\bigl( K + r \calI[w] \bigr) \, \ud r \biggr) \bigl( \calI[w] \bigr)^5 \\
 &\quad + \frac{1}{4!} \biggl( \int_0^1 (1-r)^4 r \sin\bigl( K + r \calI[w] \bigr) \, \ud r \biggr) \bigl( \calI[w] \bigr)^5 w \\
 &\quad + \frac{1}{4!} \tanh(x) \biggl( \int_0^1 (1-r)^4 r \sin\bigl( K + r \calI[w] \bigr) \, \ud r \biggr) \bigl( \calI[w] \bigr)^6 \\
 &\quad - \frac{5}{4!} \biggl( \int_0^1 (1-r)^4 \cos\bigl( K + r \calI[w] \bigr) \, \ud r \biggr) \bigl( \calI[w] \bigr)^4 w \\
 &\quad - \frac{1}{3!} \tanh(x) \biggl( \int_0^1 (1-r)^4 \cos\bigl( K + r \calI[w] \bigr) \, \ud r \biggr) \bigl( \calI[w] \bigr)^5 \\
 &\equiv \sum_{k=1}^5 \calR_{2,k}(w).
 \end{aligned}
\end{equation}

\subsection{Normal form transformation} \label{subsec:normal_form}

We now pass to the variable 
\begin{equation} \label{equ:def_new_variable_v}
 v(t) := \frac12 \bigl( w(t) - i \jD^{-1} \pt w(t) \bigr),
\end{equation}
which satisfies the first-order nonlinear Klein-Gordon equation
\begin{equation} \label{equ:nlkg_for_v}
 \left\{ \begin{aligned}
  (\pt - i \jD) v &= \frac{1}{2i} \jD^{-1} \bigl( \calQ(v+\bar{v}) + \calC(v+\bar{v}) + \calR_1(v+\bar{v}) + \calR_2(v+\bar{v}) \bigr), \\
  v(0) &= v_0, 
 \end{aligned} \right.
\end{equation}
with initial datum 
\begin{equation*}
 v_0 := \frac12 \bigl( w_0 - i \jD^{-1} w_1 \bigr).
\end{equation*}
Note that $v(t,x)$ is even since $w(t,x)$ is even.
In order to derive decay and asymptotics of the solution $w(t)$ to the flat nonlinear Klein-Gordon equation~\eqref{equ:nlkg_for_w}, it suffices to deduce these for the variable $v(t)$, because we have
\begin{equation} \label{equ:wisvplusvbar}
 w(t) = v(t) + \bar{v}(t).
\end{equation}
We will frequently use \eqref{equ:wisvplusvbar} as a convenient short-hand notation.

Before we begin with the analysis of the long-time behavior of the solution $v(t)$ to~\eqref{equ:nlkg_for_v}, we need to examine the quadratic nonlinearities on the right-hand side of~\eqref{equ:nlkg_for_v}, whose coefficients are spatially localized. Since $\calQ_2(v+\bar{v})$ and $\calQ_3(v+\bar{v})$ feature at least one factor of $\px v$, we expect these quadratic contributions to be better behaved due to the expected improved local decay of $\px v$ and the spatial localization furnished by the coefficients $\alpha_2(x)$ and $\alpha_3(x)$. In contrast, the quadratic contribution of $\calQ_1(v+\bar{v}) = \alpha_1(x) (v+\bv)^2$ appears more problematic at first sight. However, it turns out that the coefficient $\alpha_1(x)$ exhibits the miraculous non-resonance property $\widehat{\alpha}_1(\pm \sqrt{3}) = 0$ as the next lemma shows. 

\begin{lemma} \label{lem:nonresonance}
 The Fourier transform of 
 \begin{equation}
  \alpha_1(x) = 3 \sech^3(x) \tanh^2(x)
 \end{equation}
 is given by
 \begin{equation}
  \widehat{\alpha}_1(\xi) = - \frac{1}{8} \sqrt{\frac{\pi}{2}} (\xi^2-3) (\xi^2+1) \sech \Bigl( \frac{\pi \xi}{2} \Bigr).
 \end{equation}
 In particular, it follows that
 \begin{equation}
  \widehat{\alpha}_1(\pm \sqrt{3}) = 0,
 \end{equation}
 and that $(2-\jD)^{-1} \alpha_1 \in \calS(\bbR)$ is a Schwartz function. 
\end{lemma}
\begin{proof}
 By direct computation we find that
 \begin{equation*}
  \bigl( \px^4 + 2 \px^2 - 3 \bigr) \sech(x)  = -24 \sech^3(x) \tanh^2(x).
 \end{equation*}
 Using that 
 \begin{equation*}
  \widehat{\sech}(\xi) = \sqrt{\frac{\pi}{2}} \sech \Bigl( \frac{\pi \xi}{2} \Bigr),
 \end{equation*}
 we correspondingly obtain 
 \begin{align*}
  \widehat{\alpha}_1(\xi) &= -\frac18 \calF\bigl[ ( \px^4 + 2 \px^2 - 3 ) \bigl( \sech(\cdot) \bigr) \bigr](\xi) \\
  &= -\frac18 (\xi^4 - 2 \xi^2 - 3) \sqrt{\frac{\pi}{2}} \sech \Bigl( \frac{\pi \xi}{2} \Bigr) \\
  &= -\frac18 \sqrt{\frac{\pi}{2}} (\xi^2-3) (\xi^2+1) \sech \Bigl( \frac{\pi \xi}{2} \Bigr).
 \end{align*}
 Clearly, we have $\widehat{\alpha}_1(\pm \sqrt{3}) = 0$. Moreover, although $2-\jxi = 0$ for $\xi = \pm \sqrt{3}$, it follows that $(2-\jD)^{-1} \alpha_1 \in \calS(\bbR)$ is still a Schwartz function.
\end{proof}
This observation allows us to recast the worst parts of $\calQ_1(v+\bv)$ into a better form by implementing a variable coefficient quadratic normal form introduced in~\cite{LLS2}.
To this end it is useful to consider the equation satisfied by the Fourier transform of the profile $f(t) = e^{-it\jD} v(t)$ of the solution $v(t)$ to~\eqref{equ:nlkg_for_v} given by
\begin{equation} \label{equ:FT_profile_equation}
 \begin{aligned}
  \pt \hatf(t,\xi) &= \frac{1}{2i} \jxi^{-1} e^{-it\jxi} \calF\bigl[ \calQ(v + \bv)(t) \bigr](\xi) + \frac{1}{2i} \jxi^{-1} e^{-it\jxi} \calF\bigl[ \calC(v+\bar{v})(t) \bigr](\xi) \\
  &\quad + \frac{1}{2i} \jxi^{-1} e^{-it\jxi} \calF\bigl[ \calR_1(v+\bar{v})(t) + \calR_2(v+\bar{v})(t) \bigr](\xi).
 \end{aligned}
\end{equation}
Then we decompose the delicate quadratic contribution $\calQ_1(v+\bv)$ into 
\begin{equation} \label{equ:Q1_first_decomposition}
 \begin{aligned}
  \calQ_1(v+\bv)(t, x) &= \alpha_1(x) \bigl( v(t,x) + \bv(t,x) \bigr)^2 \\
  &= \alpha_1(x) \bigl( v(t,0) + \bv(t,0) \bigr)^2 \\
  &\quad + \alpha_1(x) \Bigl( \bigl( v(t,x) + \bv(t,x) \bigr)^2 - \bigl( v(t,0) + \bv(t,0) \bigr)^2 \Bigr).
 \end{aligned}
\end{equation}
The second term on the right-hand side of~\eqref{equ:Q1_first_decomposition} is of the schematic form $x \alpha_1(x) (\px v)(t) v(t)$ by the fundamental theorem of calculus, and is therefore expected to be better behaved due to the improved local decay of $\px v(t)$. In order to further analyze the contribution of the first term on the right-hand side of~\eqref{equ:Q1_first_decomposition} to~\eqref{equ:FT_profile_equation}, we insert $v(t,0) = e^{it} (e^{-it} v(t,0))$ to obtain 
\begin{equation} \label{equ:key_quadratic_contribution1}
 \begin{aligned}
  \frac{1}{2i} e^{-it\jxi} \jxi^{-1} \widehat{\alpha}_1(\xi) \bigl( v(t,0) + \bv(t,0) \bigr)^2 &= \frac{1}{2i} e^{it(2-\jxi)} \jxi^{-1} \widehat{\alpha}_1(\xi) \bigl( e^{-it} v(t,0) \bigr)^2 \\
  &\quad \quad + \frac{1}{i} e^{-it\jxi} \jxi^{-1} \widehat{\alpha}_1(\xi) \bigl( e^{-it} v(t,0) \bigr) \bigl( \overline{e^{-it} v(t,0)} \bigr) \\
  &\quad \quad + \frac{1}{2i} e^{-it(2+\jxi)} \jxi^{-1} \widehat{\alpha}_1(\xi) \bigl( \overline{e^{-it} v(t,0)} \bigr)^2.
 \end{aligned}
\end{equation}
Exploiting the oscillations and the crucial non-resonance property $\widehat{\alpha}_1(\pm \sqrt{3}) = 0$ established in Lemma~\ref{lem:nonresonance}, we recast~\eqref{equ:key_quadratic_contribution1} as
\begin{equation} \label{equ:key_quadratic_contribution2}
 \begin{aligned}
  &\frac{1}{2i} e^{-it\jxi} \jxi^{-1} \widehat{\alpha}_1(\xi) \bigl( v(t,0) + \bv(t,0) \bigr)^2 \\
  &\quad = \pt \biggl( - \hf e^{-it\jxi} \jxi^{-1} (2-\jxi)^{-1} \widehat{\alpha}_1(\xi) v(t,0)^2 \biggr) \\
  &\quad \quad + e^{-it\jxi} \jxi^{-1} (2-\jxi)^{-1} \widehat{\alpha}_1(\xi) e^{2it} \pt \bigl( e^{-it} v(t,0) \bigr) \bigl( e^{-it} v(t,0) \bigr) \\
  &\quad \quad + \pt \biggl( e^{-it\jxi} \jxi^{-2} \widehat{\alpha}_1(\xi) |v(t,0)|^2 \biggr) \\
  &\quad \quad - 2 e^{-it\jxi} \jxi^{-2} \widehat{\alpha}_1(\xi) \, \Re \Bigl( \pt \bigl( e^{-it} v(t,0) \bigr) \bigl( \overline{e^{-it} v(t,0)} \bigr) \Bigr) \\
  &\quad \quad + \pt \biggl( \frac12 e^{-it\jxi} \jxi^{-1} (2+\jxi)^{-1} \widehat{\alpha}_1(\xi) \bar{v}(t,0)^2 \Bigr) \\
  &\quad \quad - e^{-it\jxi} \jxi^{-1} (2+\jxi)^{-1} \widehat{\alpha}_1(\xi) e^{-2it} \pt \bigl( \overline{e^{-it} v(t,0)} \bigr) \bigl( \overline{e^{-it} v(t,0)} \bigr).
\end{aligned}
\end{equation}
Upon defining 
\begin{align*}
 \widehat{\alpha}_{11}(\xi) &:= \frac12 \jxi^{-1} (2-\jxi)^{-1} \widehat{\alpha}_1(\xi) , \\ 
 \widehat{\alpha}_{12}(\xi) &:= - \jxi^{-2} \widehat{\alpha}_1(\xi), \\
 \widehat{\alpha}_{13}(\xi) &:= - \frac12 \jxi^{-1} (2+\jxi)^{-1} \widehat{\alpha}_1(\xi),
\end{align*}
we introduce the variable coefficient quadratic normal form  
\begin{equation} \label{equ:def_variable_coeff_normal_form}
 B(v,v)(t) := \alpha_{11}(x) v(t,0)^2 + \alpha_{12}(x) |v(t,0)|^2 + \alpha_{13}(x) \bar{v}(t,0)^2. 
\end{equation}
Then we conclude from \eqref{equ:key_quadratic_contribution2} that 
\begin{equation} \label{equ:FT_profile_equation_renorm}
 \begin{aligned}
  &\pt \Bigl( \hatf(t,\xi) + e^{-it\jxi} \calF\bigl[ B(v,v)(t) \bigr](\xi) \Bigr) \\
  &\quad = \frac{1}{2i} \jxi^{-1} e^{-it\jxi} \calF\bigl[ \calQ_{ren}(v,v)(t) \bigr](\xi) + \frac{1}{2i} \jxi^{-1} e^{-it\jxi} \calF\bigl[ \calC(v+\bar{v})(t) \bigr](\xi) \\
  &\quad \quad + \frac{1}{2i} \jxi^{-1} e^{-it\jxi} \calF\bigl[ \calR_1(v+\bar{v})(t) + \calR_2(v+\bar{v})(t) \bigr](\xi),
 \end{aligned}
\end{equation}
where the renormalized quadratic nonlinearity is given by, see~\eqref{equ:def_Q_1to3},
\begin{equation} \label{equ:def_renormalized_quad_nonlinearities}
 \begin{aligned}
  \calQ_{ren}(v, v) := \calQ_{11}(v,v) + \calQ_{12}(v,v) + \calQ_{13}(v,v) + \calQ_{14}(v+\bar{v}) + \calQ_2(v+\bar{v}) + \calQ_3(v+\bar{v})
 \end{aligned}
\end{equation}
with 
\begin{equation} \label{equ:def_Q1k_quad_nonlinearities}
\begin{aligned}
 \calQ_{11}(v,v)(t,x) &= 2 (\jD \alpha_{11})(x) e^{2it} \partial_t \bigl( e^{-it} v(t,0) \bigr) \bigl( e^{-it} v(t,0) \bigr), \\
 \calQ_{12}(v,v)(t,x) &= 2 (\jD \alpha_{12})(x) \, \Re \Bigl(\pt \bigl( e^{-it} v(t,0) \bigr) \bigl( e^{it} \bar{v}(t,0) \bigr) \Bigr), \\
 \calQ_{13}(v,v)(t,x) &= 2 (\jD \alpha_{13})(x) e^{-2it} \pt \bigl( e^{it} \bar{v}(t,0) \bigr) \bigl( e^{it} \bar{v}(t,0) \bigr), \\
 \calQ_{14}(v+\bv)(t,x) &= \alpha_1(x) \Bigl( \bigl( v(t,x) + \bv(t,x) \bigr)^2 - \bigl( v(t,0) + \bv(t,0) \bigr)^2 \Bigr).
\end{aligned}
\end{equation}
Moreover, it follows that the renormalized variable $v + B(v,v)$ satisfies the equation
\begin{equation} \label{equ:nlkg_for_v_renorm}
 \begin{aligned}
  (\pt - i\jD) \bigl( v + B(v,v) \bigr) = \frac{1}{2i} \jD^{-1} \Bigl( \calQ_{ren}(v, v) + \calC(v+\bar{v}) + \calR_1(v+\bar{v}) + \calR_2(v+\bar{v}) \Bigr).
 \end{aligned}
\end{equation}
The latter can be written in Duhamel form as 
\begin{equation} \label{equ:duhamel_v_renorm}
 \begin{aligned}
  v(t) &= e^{it\jD} \bigl( v_0 + B(v, v)(0) \bigr) - B(v,v)(t) \\
  &\quad + \frac{1}{2i} \int_0^t e^{i(t-s)\jD} \jD^{-1} \calQ_{ren}(v, v)(s) \, \ud s \\
  &\quad + \frac{1}{2i} \int_0^t e^{i(t-s)\jD} \jD^{-1} \calC(v + \bar{v})(s) \, \ud s \\
  &\quad + \frac{1}{2i} \int_0^t e^{i(t-s)\jD} \jD^{-1} \calR_1(v+\bar{v})(s) \, \ud s \\
  &\quad + \frac{1}{2i} \int_0^t e^{i(t-s)\jD} \jD^{-1} \calR_2(v+\bar{v})(s) \, \ud s.
 \end{aligned}
\end{equation}

Having recast the quadratic nonlinearity into a more favorable form via the variable coefficient quadratic normal form, 
we are now prepared to determine the decay and the asymptotics of small solutions $v(t)$ to~\eqref{equ:nlkg_for_v}.
By time-reversal symmetry it suffices to consider only positive times.
We seek to establish an a priori bound on the quantity
\begin{equation} \label{equ:def_NT_setting_up}
 \begin{aligned}
  N(T) &:= \sup_{0 \leq t \leq T} \, \biggl\{ \jt^{\frac{1}{2}} \|v(t)\|_{L^\infty_x} + \jt^{-\delta} \| \jD^2 v(t) \|_{L^2_x} + \jt^{-\delta} \| \jD L v(t) \|_{L^2_x} \\
  &\qquad \qquad \qquad \qquad \qquad \qquad \qquad \qquad + \jt^{-1-\delta} \|x v(t)\|_{L^2_x} + \bigl\| \jap{\xi}^{\frac{3}{2}} \hat{f}(t,\xi) \bigr\|_{L^\infty_\xi} \biggr\},
 \end{aligned}
\end{equation}
where $T > 0$ is arbitrary and where $0 < \delta \ll 1$ is a small absolute constant whose size will be specified later.
In the next Section~\ref{sec:energy_estimates} we derive bounds on the  $L^2_x$-based norms of $v(t)$, and in Section~\ref{sec:pointwise_estimates} we control the weighted $L^\infty_\xi$-norm of the Fourier transform of the profile $f(t)$ of $v(t)$. 
We then combine these estimates in the proof of Theorem~\ref{thm:main} in Section~\ref{sec:proof_of_thm} to infer the desired a priori bound on $N(T)$ via a standard continuity argument. This gives a sharp decay estimate and asymptotics for $v(t)$, which in turn imply the asserted decay estimate and asymptotics for the perturbation $u(t,x)$ of the sine-Gordon kink. 
Since we only consider small initial data for $v(t)$, throughout we may freely assume that $T \geq 1$ and that $N(T) \leq 1$, which simplifies the bookkeeping of some of the estimates.

\section{Energy estimates} \label{sec:energy_estimates}

In this section we derive a priori estimates for all $L^2_x$-based norms that are part of the bootstrap quantity~\eqref{equ:def_NT_setting_up}. 

\subsection{Preparations}

Before we turn to the proofs of the main energy estimates, we first need to make several technical preparations.
We begin with several $L^\infty_x$- and $L^2_x$-bounds on quantities involving the integral operators $\calI$ and $\wtcalI$ defined in~\eqref{equ:def_calI}, respectively in \eqref{equ:def_wtcalI}.

\begin{lemma} \label{lem:aux_bounds_Ioperators}
Let $T > 0$ and let $N(T)$ be defined as in~\eqref{equ:def_NT_setting_up}. Then we have uniformly for all $0 \leq t \leq T$ that
 \begin{align}
  \|\calI[v(t)]\|_{L^\infty_x} + \|\px \calI[v(t)]\|_{L^\infty_x} &\lesssim \|v(t)\|_{L^\infty_x} \lesssim N(T) \jt^{-\hf}, \label{equ:aux_bound_Linfty_Iv} \\
  \| \wtcalI[\px v(t)] \|_{L^\infty_x} &\lesssim \|v(t)\|_{L^\infty_x} \lesssim N(T) \jt^{-\hf}, \label{equ:aux_bound_Linfty_wtIpxv} \\
  \| \calI[v(t)] \|_{L^2_x} + \| \px \calI[v(t)] \|_{L^2_x} &\lesssim \|v(t)\|_{L^2_x} \lesssim N(T) \jt^\delta, \label{equ:aux_bound_L2_Iv} \\
  \| \wtcalI[\px v(t)] \|_{L^2_x} &\lesssim \|v(t)\|_{L^2_x} \lesssim N(T) \jt^\delta, \label{equ:aux_bound_L2_wtIpxv} \\
  \| \wtcalI[\px \pt v(t)] \|_{L^2_x} &\lesssim \|\pt v(t)\|_{L^2_x} \lesssim N(T) \jt^\delta, \label{equ:aux_bound_L2_wtIpxptv} \\
  \|\jx \calI[v(t)]\|_{L^2_x} &\lesssim \|\jx v(t)\|_{L^2_x} \lesssim N(T) \jt^{1+\delta}, \label{equ:aux_bound_L2_jxIv} \\
  \|\jx \wtcalI[\px v(t)]\|_{L^2_x} &\lesssim \|\jx v(t)\|_{L^2_x} \lesssim N(T) \jt^{1+\delta}. \label{equ:aux_bound_L2_jxwtcalIpxv}
 \end{align}
\end{lemma}
\begin{proof}
 The asserted bounds all follow in a straightforward manner from the exponential localization of the kernels in the definition of the integral operators $\calI[v(t)](x)$, respectively $\wtcalI[\px v(t)](x)$. We remark that for the proofs of~\eqref{equ:aux_bound_Linfty_wtIpxv}, \eqref{equ:aux_bound_L2_wtIpxv}, \eqref{equ:aux_bound_L2_wtIpxptv}, and \eqref{equ:aux_bound_L2_jxwtcalIpxv}, we first integrate by parts. 
 Moreover, the asserted bound $\|\pt v(t)\|_{L^2_x} \lesssim N(T) \jt^\delta$ on the right-hand side of \eqref{equ:aux_bound_L2_wtIpxptv} follows from \eqref{equ:aux_slow_growth_jDptv} below.
\end{proof}

On occasion we will also need the following auxiliary slow growth estimates.
\begin{lemma}[Auxiliary slow energy growth bounds] \label{lem:auxiliary_slow_growth_bounds}
Let $T > 0$ and let $N(T)$ be defined as in~\eqref{equ:def_NT_setting_up}. Then we have uniformly for all $0 \leq t \leq T$ that
 \begin{align}
  \|\jD \pt v(t)\|_{L^2_x} &\lesssim N(T) \jt^\delta, \label{equ:aux_slow_growth_jDptv}  \\
  \|\jD Z v(t)\|_{L^2_x} &\lesssim N(T) \jt^\delta. \label{equ:aux_slow_growth_jDZv}
 \end{align}
\end{lemma}
\begin{proof}
We first prove the estimate~\eqref{equ:aux_slow_growth_jDptv}. Writing $\pt v = i \jD v + (\pt - i\jD) v$ and inserting the equation~\eqref{equ:nlkg_for_v} for $v(t)$, we obtain 
\begin{equation*}
 \begin{aligned}
  \|\jD \pt v(t)\|_{L^2_x} &\lesssim \|\jD^2 v(t)\|_{L^2_x} + \|\jD (\pt - i\jD) v(t)\|_{L^2_x} \\
  &\lesssim N(T) \jt^\delta + \| \calQ(v+\bv)(t) \|_{L^2_x} + \| \calC(v+\bv)(t)\|_{L^2_x} \\
  &\quad + \|\calR_1(v+\bv)(t)\|_{L^2_x} + \|\calR_2(v+\bv)(t)\|_{L^2_x}.
 \end{aligned}
\end{equation*}
The contributions of the nonlinearities on the right-hand side can now be estimated quite crudely. Using~\eqref{equ:aux_bound_Linfty_wtIpxv}, we may bound the quadratic nonlinearities by
\begin{equation*}
 \| \calQ(v+\bv)(t) \|_{L^2_x} \lesssim \sum_{j=1}^3 \|\alpha_j\|_{L^2_x} \|v(t)\|_{L^\infty_x}^2 \lesssim N(T)^2 \jt^{-1}.
\end{equation*}
Then owing to~\eqref{equ:aux_bound_Linfty_Iv} and~\eqref{equ:aux_bound_L2_Iv}, the cubic nonlinearities can be estimated by 
\begin{equation*}
 \| \calC(v+\bv)(t) \|_{L^2_x} \lesssim \|v(t)\|_{L^2_x} \|v(t)\|_{L^\infty_x}^2 \lesssim N(T)^3 \jt^{-(1-\delta)},
\end{equation*}
and the contributions of the quartic and quintic nonlinearities can be treated in a similar manner. Combining the preceding estimates establishes~\eqref{equ:aux_slow_growth_jDptv}.

Next, we deduce the estimate~\eqref{equ:aux_slow_growth_jDZv}. To this end we write 
\begin{equation*}
 \jD Z = i\jD L + i\px - \jD^{-1} \px (\pt - i\jD) + x \jD (\pt - i\jD)
\end{equation*}
and then insert the equation~\eqref{equ:nlkg_for_v} to find that
\begin{equation*}
 \begin{aligned}
  \|\jD Z v(t)\|_{L^2_x} &\lesssim \|\jD L v(t)\|_{L^2_x} + \|\px v(t)\|_{L^2_x} + \|\jx\jD(\pt-i\jD)v(t)\|_{L^2_x} \\
  &\lesssim N(T) \jt^\delta + \| \jx \calQ(v+\bv)(t) \|_{L^2_x} + \| \jx \calC(v+\bv)(t)\|_{L^2_x} \\
  &\quad + \|\jx \calR_1(v+\bv)(t)\|_{L^2_x} + \|\jx \calR_2(v+\bv)(t)\|_{L^2_x}.
 \end{aligned}
\end{equation*}
Using~\eqref{equ:aux_bound_Linfty_wtIpxv} we bound the quadratic nonlinearities by
\begin{equation*}
 \| \jx \calQ(v+\bv)(t) \|_{L^2_x} \lesssim \sum_{j=1}^3 \|\jx \alpha_j\|_{L^2_x} \|v(t)\|_{L^\infty_x}^2 \lesssim N(T)^2 \jt^{-1},
\end{equation*}
and invoking~\eqref{equ:aux_bound_Linfty_Iv} as well as~\eqref{equ:aux_bound_L2_jxIv}, we estimate the cubic nonlinearities by
\begin{equation*}
 \| \jx \calC(v+\bv)(t) \|_{L^2_x} \lesssim \|\jx v(t)\|_{L^2_x} \|v(t)\|_{L^\infty_x}^2 \lesssim N(T)^3 \jt^{\delta}.
\end{equation*}
Finally, the quartic and quintic nonlinearities can be treated analogously. Putting together the preceding bounds yields the estimate~\eqref{equ:aux_slow_growth_jDZv}, and thus finishes the proof of the lemma.
\end{proof}

The following improved local decay estimates for the solution $v(t)$ to~\eqref{equ:nlkg_for_v} play a key role in multiple places in the derivation of the main energy estimates.

\begin{lemma}[Improved local decay] \label{lem:improved_local_decay}
Let $T > 0$ and let $N(T)$ be defined as in~\eqref{equ:def_NT_setting_up}. Then we have uniformly for all $0 \leq t \leq T$ that
 \begin{align}
  \bigl\| \jx^{-1} \px v(t) \bigr\|_{H^1_x} &\lesssim N(T) \jt^{-(1-\delta)}, \label{equ:improved_local_decay_px_v} \\
  \bigl\| \jx^{-1} (-1+\jD) v(t) \bigr\|_{H^1_x} &\lesssim N(T) \jt^{-(1-\delta)}, \label{equ:improved_local_decay_minusoneplusjapD} \\
  \bigl\| \jx^{-1} \wtcalI[\px v(t)] \bigr\|_{L^2_x} &\lesssim N(T) \jt^{-(1-\delta)}, \label{equ:improved_local_decay_wtilIpxv} \\
  \bigl\| \jx^{-1} \px \wtcalI[\px v(t)] \bigr\|_{L^2_x} &\lesssim N(T) \jt^{-(1-\delta)}, \label{equ:improved_local_decay_pxwtilIpxv} \\
  \bigl\| \jx^{-2} \bigl( w(t,x)^2 - w(t,0)^2 \bigr) \bigr\|_{L^2_x} &\lesssim N(T)^2 \jt^{-(\thf-\delta)}. \label{equ:improved_local_decay_w_minus_w_origin}
 \end{align}
\end{lemma}
\begin{proof}
We begin with the proof of~\eqref{equ:improved_local_decay_px_v}. Writing the solution $v(t)$ in terms of its profile and using the improved local decay estimate~\eqref{equ:local_decay_van2} for the linear Klein-Gordon evolution, we find uniformly for all $0 \leq t \leq T$ that
\begin{equation*}
 \begin{aligned}
  \bigl\| \jx^{-1} \px v(t) \bigr\|_{H^1_x} &= \bigl\| \jx^{-1} \px e^{it\jD} f(t) \bigr\|_{H^1_x} \\
  &\lesssim \bigl\| \jx^{-1} \px \jD^{-1} e^{it\jD} \jx^{-1} \bigr\|_{L^2_x \to L^2_x} \bigl\| \jx \jD^2 f(t) \bigr\|_{L^2_x} \\
  &\lesssim \jt^{-1} \bigl( \bigl\|\jD^2 v(t)\bigr\|_{L^2_x} + \bigl\|\jD L v(t)\bigr\|_{L^2_x} \bigr) \\
  &\lesssim N(T) \jt^{-(1-\delta)}.
 \end{aligned}
\end{equation*}
The proof of \eqref{equ:improved_local_decay_minusoneplusjapD} proceeds analogously using the improved local decay estimate~\eqref{equ:local_decay_van3} for the linear Klein-Gordon evolution.

To prove~\eqref{equ:improved_local_decay_wtilIpxv} we write
\begin{equation*}
 \jx^{-1} \wtcalI[\px v(t)](x) = \int_0^x \jx^{-1} \jap{y} \sech(x) \sinh(y) \jap{y}^{-1} (\px v)(t,y) \, \ud y.
\end{equation*}
Then the estimate~\eqref{equ:improved_local_decay_wtilIpxv} follows from~\eqref{equ:improved_local_decay_px_v} and Schur's test for the kernel 
\begin{equation*}
 K(x,y) := \bigl( \one_{[0,\infty)}(x) \one_{[0,x]}(y) - \one_{(-\infty,0)}(x) \one_{[x,0]}(y) \bigr) \jx^{-1} \jap{y} \sech(x) \sinh(y).
\end{equation*}
For the proof of~\eqref{equ:improved_local_decay_pxwtilIpxv} we first compute that
\begin{equation*}
 \px \bigl( \wtcalI[\px v(t)](x) \bigr) = - \tanh(x) \wtcalI[\px v(t)](x) + \tanh(x) \px v(t,x),
\end{equation*}
whence~\eqref{equ:improved_local_decay_pxwtilIpxv} is an immediate consequence of the estimates~\eqref{equ:improved_local_decay_px_v} and~\eqref{equ:improved_local_decay_wtilIpxv}. Finally, see~\cite[Lemma 4.3]{LLS2} for the proof of the estimate~\eqref{equ:improved_local_decay_w_minus_w_origin}.
\end{proof}

The following improved decay estimates of the solution $v(t)$ to~\eqref{equ:nlkg_for_v} at the origin $x=0$ are crucial for estimating the renormalized quadratic nonlinearities as well as for obtaining a slow energy growth estimate for the action of a Lorentz boost on the integral operator $\wtcalI[\px v(t)]$ in Corollary~\ref{cor:growth_L2_Z_action_calIs} below.

\begin{lemma}[Improved decay at the origin]
Let $T > 0$ and let $N(T)$ be defined as in~\eqref{equ:def_NT_setting_up}. Then we have uniformly for all $0 \leq t \leq T$ that
 \begin{align}
  |\px v(t,0)| &\lesssim N(T) \jt^{-(1-\delta)}, \label{equ:improved_decay_pxv_at_zero} \\
  \bigl|\pt \bigl( e^{-it} v(t,0) \bigr) \bigr| &\lesssim N(T) \jt^{-(1-\delta)}. \label{equ:improved_decay_pt_phase_filtered_v}
 \end{align}
\end{lemma}
\begin{proof}
The estimate~\eqref{equ:improved_decay_pxv_at_zero} is just a consequence of Sobolev embedding and the improved local decay estimate~\eqref{equ:improved_local_decay_px_v}.

In order to deduce the estimate~\eqref{equ:improved_decay_pt_phase_filtered_v}, we proceed as in the proof of \cite[Lemma 4.1]{LLS2}. 
We write 
\begin{equation} \label{equ:pt_phase_filtered_v_written_out}
 \pt \bigl( e^{-it} v(t,0) \bigr) = i e^{-it} \bigl( (-1+\jD) v \bigr)(t,0) + e^{-it} \bigl( (\pt - i\jD) v \bigr)(t,0).
\end{equation}
Then the desired bound $|\bigl( (-1+\jD) v \bigr)(t,0)| \lesssim N(T) \jt^{-(1-\delta)}$ for the first term on the right-hand side of \eqref{equ:pt_phase_filtered_v_written_out} is a consequence of Sobolev embedding and the improved local decay estimate~\eqref{equ:improved_local_decay_minusoneplusjapD}.
For the second term on the right-hand side of \eqref{equ:pt_phase_filtered_v_written_out}, we obtain the desired improved decay easily by inserting the equation~\eqref{equ:nlkg_for_v} for $v(t)$ and using the estimates~\eqref{equ:aux_bound_Linfty_Iv}--\eqref{equ:aux_bound_Linfty_wtIpxv}.
\end{proof}

In the next lemma we determine how a Lorentz boost $Z$ acts on the integral operators $\calI$ and $\wtcalI$ defined in~\eqref{equ:def_calI}, respectively in~\eqref{equ:def_wtcalI}.

\begin{lemma} \label{lem:Z_action_calIs}
 The following identities hold
 \begin{align}
  Z \bigl( \calI[v(t)](x) \bigr) &= - t \sech^2(x) v(t,x) + \int_0^x K_1(x,y) (Zv)(t,y) \, \ud y \notag \\
   &\quad \quad + \int_0^x K_2(x,y) (\pt v)(t,y) \, \ud y, \label{equ:Z_action_calI} \\
  Z \bigl( \wtcalI[\px v(t)](x) \bigr) &= t \sech(x) \tanh(x) (\px v)(t,0) + \int_0^x K_3(x,y) (Z \partial_y v)(t,y) \, \ud y \notag \\
   &\quad \quad + \int_0^x K_4(x,y) (\pt \partial_y v)(t,y) \, \ud y, \label{equ:Z_action_wtcalI}
 \end{align}
 with smooth kernels $K_j(x,y)$, $1 \leq j \leq 4$, satisfying 
 \begin{equation} \label{equ:Z_action_kernel_bounds}
  |K_j(x,y)| \leq C e^{- c |x-y|} \quad \text{for} \quad 0 \leq y \leq x \quad \text{or} \quad x \leq y \leq 0
 \end{equation}
 for some absolute constants $C, c > 0$.
\end{lemma}
\begin{proof} 
We begin with the proof of the identity~\eqref{equ:Z_action_calI}. To this end we compute
\begin{equation} \label{equ:compute_Z_action_calI1}
 \begin{aligned}
  Z \bigl( \calI[v(t)](x) \bigr) &= (t \px + x \pt) \biggl( - \sech(x) \int_0^x \cosh(y) v(t,y) \, \ud y \biggr) \\
  &= - t v(t,x) + t \tanh(x) \sech(x) \int_0^x \cosh(y) v(t,y) \, \ud y \\
  &\quad - x \sech(x) \int_0^x \cosh(y) (\pt v)(t,y) \, \ud y.
 \end{aligned}
\end{equation}
Then we integrate by parts in the second term on the right-hand side
\begin{equation*}
 \begin{aligned}
  &t \tanh(x) \sech(x) \int_0^x \cosh(y) v(t,y) \, \ud y \\
  &\quad = t \tanh^2(x) v(t,x) - \tanh(x) \sech(x) \int_0^x \sinh(y) t (\partial_y v)(t,y) \, \ud y,
 \end{aligned}
\end{equation*}
and use that $\tanh^2(x) = 1 -\sech^2(x)$, in order to rewrite \eqref{equ:compute_Z_action_calI1} as
\begin{equation} \label{equ:compute_Z_action_calI2}
 \begin{aligned}
  Z \bigl( \calI[v(t)](x) \bigr) =& - t \sech^2(x) v(t,x) - \tanh(x) \sech(x) \int_0^x \sinh(y) t (\partial_y v)(t,y) \, \ud y \\
  & - \sech(x) \int_0^x \cosh(y) y (\pt v)(t,y) \, \ud y \\
  & - \sech(x) \int_0^x \cosh(y) (x-y) (\pt v)(t,y) \, \ud y.
 \end{aligned}
\end{equation}
Finally, inserting the relation $t (\py v)(t,y) = (Zv)(t,y) - y (\pt v)(t,y)$ in the integrand of the second term on the right-hand side of~\eqref{equ:compute_Z_action_calI2} and using the subtraction formula for the hyperbolic cosine function to combine terms, we conclude that
\begin{equation*}
 \begin{aligned}
  Z \bigl( \calI[v(t)](x) \bigr) &= - t \sech^2(x) v(t,x) + \int_0^x K_1(x,y) (Zv)(t,y) \, \ud y + \int_0^x K_2(x,y) (\pt v)(t,y) \, \ud y
 \end{aligned}
\end{equation*}
with smooth kernels $K_j(x,y)$, $j = 1, 2$, defined by 
\begin{equation} \label{equ:compute_Z_action_kernel_def1}
 \begin{aligned}
  K_1(x,y) &:= - \tanh(x) \sech(x) \sinh(y), \\
  K_2(x,y) &:= - \sech^2(x) \cosh(x-y) y - \sech(x) \cosh(y) (x-y).
 \end{aligned}
\end{equation}
 
Next, we establish the identity~\eqref{equ:Z_action_wtcalI}. We begin by computing  
 \begin{equation} \label{equ:compute_Z_action_wtcalI1}
  \begin{aligned}
   Z \bigl( \wtcalI[\px v(t)](x) \bigr) &= (t \px + x \pt) \biggl( \sech(x) \int_0^x \sinh(y) (\py v)(t,y) \, \ud y \biggr) \\
   &=  - t \sech(x) \tanh(x) \int_0^x \sinh(y) (\py v)(t,y) \, \ud y + t \tanh(x) (\px v)(t,x) \\
   &\quad + x \sech(x) \int_0^x \sinh(y) (\pt \py v)(t,y) \, \ud y.
  \end{aligned}
 \end{equation}
 Proceeding analogously to the preceding derivation, we integrate by parts in the first term on the right-hand side, viz.
 \begin{equation*}
 \begin{aligned}
  - t \sech(x) \tanh(x) \int_0^x \sinh(y) (\py v)(t,y) \, \ud y &= - t \tanh(x) (\px v)(t,x) + t \sech(x) \tanh(x)  (\px v)(t,0) \\
  &\quad \quad + t \sech(x) \tanh(x) \int_0^x \cosh(y) (\py^2 v)(t,y) \, \ud y,
 \end{aligned}
\end{equation*}
to rewrite~\eqref{equ:compute_Z_action_wtcalI1} as
\begin{equation} \label{equ:compute_Z_action_wtcalI2}
 \begin{aligned}
  Z \bigl( \wtcalI[\px v(t)](x) \bigr) &= t \sech(x) \tanh(x)  (\px v)(t,0) \\
  &\quad + \sech(x) \tanh(x) \int_0^x \cosh(y) t (\py^2 v)(t,y) \, \ud y \\
  &\quad + x \sech(x) \int_0^x \sinh(y) (\pt \py v)(t,y) \, \ud y.
 \end{aligned}
\end{equation}
Then we use the relation $t (\py^2 v)(t,y) = (Z \py v)(t,y) - y (\pt \py v)(t,y)$ and the substraction formula for the hyperbolic sine function in order to further rewrite the last two terms on the right-hand side of~\eqref{equ:compute_Z_action_wtcalI2} as
 \begin{equation*}
  \begin{aligned}
   &\sech(x) \tanh(x) \int_0^x \cosh(y) t (\py^2 v)(t,y) \, \ud y + x \sech(x) \int_0^x \sinh(y) (\pt \py v)(t,y) \, \ud y \\
   &= \sech(x) \tanh(x) \int_0^x \cosh(y) (Z \py v)(t,y) \, \ud y - \sech(x) \tanh(x) \int_0^x \cosh(y) y (\pt \py v)(t,y) \, \ud y \\
   &\quad + x \sech(x) \int_0^x \sinh(y) (\pt \py v)(t,y) \, \ud y \\
   &= \int_0^x \sech(x) \tanh(x) \cosh(y) (Z \py v)(t,y) \, \ud y - \int_0^x \sinh(x-y) \sech^2(x) y (\pt \py v)(t,y) \, \ud y \\
   &\quad + \int_0^x \sech(x) \sinh(y) (x-y) (\pt \py v)(t,y) \, \ud y.
  \end{aligned}
 \end{equation*}
 In this manner we arrive at the identity
 \begin{equation*}
  \begin{aligned}
   Z \bigl( \wtcalI[\px v(t)](x) \bigr) &= t \sech(x) \tanh(x) (\px v)(t,0) + \int_0^x K_3(x,y) (Z \py v)(t,y) \, \ud y \\
   &\quad + \int_0^x K_4(x,y) (\pt \py v)(t,y) \, \ud y,
  \end{aligned}
 \end{equation*}
 with smooth kernels $K_j(x,y)$, $j = 3,4$ defined by
 \begin{equation} \label{equ:compute_Z_action_kernel_def2}
  \begin{aligned}
   K_3(x,y) &:= \sech(x) \tanh(x) \cosh(y), \\
   K_4(x,y) &:= -\sinh(x-y) \sech^2(x) y + \sech(x) \sinh(y) (x-y).
  \end{aligned}
 \end{equation}
 
Clearly, in view of the definitions~\eqref{equ:compute_Z_action_kernel_def1} and~\eqref{equ:compute_Z_action_kernel_def2} of the kernels $K_j(x,y)$, there exist absolute constants $C, c > 0$ such that for $1 \leq j \leq 4$ we have 
\begin{equation*}
 |K_j(x,y)| \leq C e^{-c|x-y|} \quad \text{for} \quad 0 \leq y \leq x \quad \text{or} \quad x \leq y \leq 0.
\end{equation*}
This finishes the proof of the lemma. 
\end{proof}

As a consequence of the identities~\eqref{equ:Z_action_calI} and~\eqref{equ:Z_action_wtcalI}, we obtain the following growth estimates for the $L^2_x$-norm of a Lorentz boost $Z$ applied to $\calI[v(t)]$, respectively to $\wtcalI[\px v(t)]$.

\begin{corollary} \label{cor:growth_L2_Z_action_calIs}
Let $T > 0$ and let $N(T)$ be defined as in~\eqref{equ:def_NT_setting_up}. Then we have uniformly for all $0 \leq t \leq T$ that
 \begin{align}
  \bigl\| Z \bigl( \calI[v(t)] \bigr) \bigr\|_{L^2_x} &\lesssim N(T) \jt^\hf, \label{equ:growth_L2_Z_action_calIv}  \\
  \bigl\| Z \bigl( \wtcalI[\px v(t)] \bigr) \bigr\|_{L^2_x} &\lesssim N(T) \jt^\delta. \label{equ:growth_L2_Z_action_wtcalIv}
 \end{align}
\end{corollary}
\begin{proof}
From the identity~\eqref{equ:Z_action_calI}, the kernel bounds~\eqref{equ:Z_action_kernel_bounds}, and the auxiliary bounds from Lemma~\ref{lem:auxiliary_slow_growth_bounds}, we conclude for any time $0 \leq t \leq T$ that
\begin{equation*}
 \begin{aligned}
  \bigl\| Z \bigl( \calI[v(t)] \bigr) \bigr\|_{L^2_x} \lesssim t \|v(t)\|_{L^\infty_x} + \|Zv(t)\|_{L^2_x} + \| \pt v(t) \|_{L^2_x} \lesssim N(T) \jt^\hf.
 \end{aligned}
\end{equation*}
This proves~\eqref{equ:growth_L2_Z_action_calIv}. Similarly, we deduce~\eqref{equ:growth_L2_Z_action_wtcalIv} from the identity~\eqref{equ:Z_action_wtcalI}, the kernel bounds~\eqref{equ:Z_action_kernel_bounds}, the improved decay at the origin~\eqref{equ:improved_decay_pxv_at_zero}, as well as the auxiliary bounds from Lemma~\ref{lem:auxiliary_slow_growth_bounds}. Specifically, we obtain for any $0 \leq t \leq T$ that 
\begin{equation*}
 \begin{aligned}
  \bigl\| Z \bigl( \wtcalI[\px v(t)] \bigr) \bigr\|_{L^2_x} &\lesssim t |(\px v)(t,0)| + \| Z \px v(t)\|_{L^2_x} + \| \pt \px v(t) \|_{L^2_x} \\
  &\lesssim t |(\px v)(t,0)| + \| (\jD Z v)(t)\|_{L^2_x} + \| (\jD \pt v)(t) \|_{L^2_x} \\
  &\lesssim N(T) \jt^\delta,
 \end{aligned}
\end{equation*}
as desired.
\end{proof}

\subsection{Main energy growth estimates}

We are now prepared for the proofs of the main energy estimates. We begin with the derivation of a slow growth estimate for the $H^2_x$-norm of the solution $v(t)$ to~\eqref{equ:nlkg_for_v}.

\begin{proposition} \label{prop:growth_H2v}
Let $v(t)$ be the solution to~\eqref{equ:nlkg_for_v} on the time interval $[0,T]$. Let $N(T)$ be defined as in~\eqref{equ:def_NT_setting_up} and assume $N(T) \leq 1$. Then we have 
 \begin{equation}
  \sup_{0 \leq t \leq T} \, \jt^{-\delta} \| \jD^2 v(t) \|_{L^2_x} \lesssim \|v_0\|_{H^2_x} + \|v_0\|_{H^1_x}^2 + N(T)^2.  
 \end{equation}
\end{proposition}
\begin{proof}
 From the Duhamel representation~\eqref{equ:duhamel_v_renorm} of $v(t)$, we obtain for any $0 \leq t \leq T$ that
 \begin{equation} \label{equ:H2v_duhamel}
  \begin{aligned}
   \|\jD^2 v(t)\|_{L^2_x} &\lesssim \|v_0\|_{H^2_x} + \|\jD^2 B(v,v)(0)\|_{L^2_x} + \|\jD^2 B(v,v)(t)\|_{L^2_x} \\
   &\quad + \int_0^t \| \jD \calQ_{ren}(v,v)(s) \|_{L^2_x} \, \ud s + \int_0^t \| \jD \calC(v+\bv)(s) \|_{L^2_x} \, \ud s \\
   &\quad + \int_0^t \| \jD \calR_1(v+\bv)(s) \|_{L^2_x} \, \ud s + \int_0^t \| \jD \calR_2(v+\bv)(s) \|_{L^2_x} \, \ud s.
  \end{aligned}
 \end{equation}
 The contributions of the variable coefficient quadratic normal form $B(v,v)$ to the right-hand side of~\eqref{equ:H2v_duhamel} can be easily estimated by
 \begin{equation} \label{equ:H2v_normal_form_contr}
  \begin{aligned}
   \|\jD^2 B(v,v)(0)\|_{L^2_x} + \|\jD^2 B(v,v)(t)\|_{L^2_x} &\lesssim \sum_{k=1}^3 \|\jD^2 \alpha_{1k}\|_{L^2_x} \bigl( |v(0,0)|^2 + |v(t,0)|^2 \bigr) \\
   &\lesssim \|v_0\|_{L^\infty_x}^2 + \|v(t)\|_{L^\infty_x}^2 \\
   &\lesssim \|v_0\|_{H^1_x}^2 + N(T)^2 \jt^{-1}.
  \end{aligned}
 \end{equation}
 Next, we estimate the contributions of all nonlinear terms to the right-hand side of~\eqref{equ:H2v_duhamel}.
 In what follows we always consider times $0 \leq s \leq t \leq T$.

 \medskip 
 
 \noindent {\it Renormalized quadratic nonlinearities:} 
 In view of the definitions~\eqref{equ:def_Q1k_quad_nonlinearities} of $\calQ_{1k}$, $1 \leq k \leq 4$, we have by the improved decay estimate~\eqref{equ:improved_decay_pt_phase_filtered_v} that 
 \begin{equation*}
  \begin{aligned}
   \sum_{k=1}^3 \|\jD \calQ_{1k}(v,v)(s)\|_{L^2_x} &\lesssim \sum_{k=1}^3 \|\jD^2 \alpha_{1k}\|_{L^2_x} | \ps (e^{-is} v(s,0) ) | |v(s,0)| \lesssim N(T)^2 \js^{-(\thf-\delta)}.
  \end{aligned}
 \end{equation*}
 Moreover, using the improved local decay estimates~\eqref{equ:improved_local_decay_px_v} and \eqref{equ:improved_local_decay_w_minus_w_origin}, we obtain
 \begin{equation*}
  \begin{aligned}
   \|\jD \calQ_{14}(v+\barv)(s)\|_{L^2_x} &\lesssim \bigl\| \jD \bigl( \alpha_1(x) \bigl( w(s)^2 - w(s,0)^2 \bigr) \bigr) \bigr\|_{L^2_x} \\
   &\lesssim \|\jx^2 \alpha_1\|_{L^\infty_x} \bigl\|\jx^{-2} \bigl( w(s)^2 - w(s,0)^2 \bigr) \bigr\|_{L^2_x} \\
   &\quad + \|\jx^2 \px \alpha_1\|_{L^\infty_x} \bigl\|\jx^{-2} \bigl( w(s)^2 - w(s,0)^2 \bigr) \bigr\|_{L^2_x} \\
   &\quad + \| \jx \alpha_1\|_{L^\infty_x} \|\jx^{-1} \px v(s)\|_{L^2_x} \|v(s)\|_{L^\infty_x} \\
   &\lesssim N(T)^2 \js^{-(\thf-\delta)}.
  \end{aligned}
 \end{equation*}
 Similarly, in view of~\eqref{equ:def_Q_1to3}, by invoking the improved local decay estimates~\eqref{equ:improved_local_decay_wtilIpxv} and \eqref{equ:improved_local_decay_pxwtilIpxv} along with the estimate~\eqref{equ:aux_bound_Linfty_wtIpxv}, we find
 \begin{equation*}
  \begin{aligned}
   \|\jD \calQ_2(v+\bv)(s)\|_{L^2_x} &\lesssim \bigl\| \jD \bigl( \alpha_2(\cdot) \widetilde{\calI}[\px w(s)] w(s) \bigr)\bigr\|_{L^2_x} \\
   &\lesssim \|\jD \jx \alpha_2\|_{L^\infty_x} \bigl\|\jx^{-1} \widetilde{\calI}[\px v(s)] \bigr\|_{L^2_x} \|v(s)\|_{L^\infty_x} \\
   &\quad + \|\jx \alpha_2\|_{L^2_x} \bigl\| \jx^{-1} \px \widetilde{\calI}[\px v(s)] \bigr\|_{L^2_x} \|v(s)\|_{L^\infty_x} \\
   &\quad + \|\jx \alpha_2\|_{L^\infty_x} \bigl\| \widetilde{\calI}[\px v(s)]\bigr\|_{L^\infty_x} \|\jx^{-1} \px v(s)\|_{L^2_x} \\
   &\lesssim N(T)^2 \js^{-(\thf-\delta)}.
  \end{aligned}
 \end{equation*}
 In an analogous manner, we derive that
 \begin{equation*}
  \|\jD \calQ_3(v+\bv)(s)\|_{L^2_x} \lesssim N(T)^2 \js^{-(\thf-\delta)}.
 \end{equation*}
 Putting together the preceding bounds, we conclude that the quadratic contributions can be estimated by
 \begin{equation}
  \int_0^t \| \jD \calQ_{ren}(v,v)(s) \|_{L^2_x} \, \ud s \lesssim \int_0^t N(T)^2 \js^{-(\thf-\delta)} \, \ud s \lesssim N(T)^2.
 \end{equation}

 \medskip  
 
 \noindent {\it Cubic nonlinearities:}  
 The contributions of the cubic nonlinearities defined in~\eqref{equ:def_nonlinearities_Cnl} and~\eqref{equ:def_nonlinearities_Cl} can all be estimated in a straightforward manner using the following bounds established in Lemma~\ref{lem:aux_bounds_Ioperators},
 \begin{equation*}
  \|\calI[v(s)]\|_{L^\infty_x} \lesssim N(T) \js^{-\hf}
 \end{equation*}
 and 
 \begin{equation*}
  \bigl\| \calI[v(t)] \bigr\|_{L^2_x} + \bigl\| \px \calI[v(t)] \bigr\|_{L^2_x} \lesssim N(T) \js^\delta.
 \end{equation*}
  We obtain that
 \begin{align*}
  \|\jD \calC(v+\bv)(s)\|_{L^2_x} \lesssim \|\jD v(s)\|_{L^2_x} \|v(s)\|_{L^\infty_x}^2 \lesssim N(T)^3 \js^{-(1-\delta)},
 \end{align*}
 which implies 
 \begin{equation}
  \int_0^t \|\jD \calC(v+\bv)(s)\|_{L^2_x} \, \ud s \lesssim N(T)^3 \jt^\delta.
 \end{equation}

 \medskip  
 
 \noindent {\it Quartic and quintic nonlinearities:}
 We proceed analogously to the treatment of the cubic terms to bound the contributions of the quartic and quintic remainder terms~\eqref{equ:def_nonlinearities_quartic} and~\eqref{equ:def_nonlinearities_quintic} by
 \begin{equation*}
  \begin{aligned}
   \| \jD \calR_1(v+\bv)(s) \|_{L^2_x} &\lesssim \|v(s)\|_{L^\infty_x}^3 \|\jD v(s)\|_{L^2_x} \lesssim N(T)^4 \js^{-(\thf-\delta)}, \\
   \| \jD \calR_2(v+\bv)(s) \|_{L^2_x} &\lesssim \|v(s)\|_{L^\infty_x}^4 \|\jD v(s)\|_{L^2_x} \lesssim N(T)^5 \js^{-(2-\delta)},
  \end{aligned}
 \end{equation*}
 which implies that
 \begin{equation*}
  \int_0^t \| \jD \calR_1(v+\bv)(s) \|_{L^2_x} \, \ud s + \int_0^t \| \jD \calR_2(v+\bv)(s) \|_{L^2_x} \, \ud s \lesssim N(T)^4 + N(T)^5.
 \end{equation*}

 \medskip 
 
 Putting all of the preceding estimates together and recalling that we assume $N(T) \leq 1$, we obtain for any $0 \leq t \leq T$ that 
 \begin{equation*}
  \|\jD^2 v(t)\|_{L^2_x} \lesssim \|v_0\|_{H^2_x} + \|v_0\|_{H^1_x}^2 + N(T)^2 \jt^\delta.
 \end{equation*}
 This finishes the proof.
\end{proof}

Next, we deduce a growth estimate for the $L^2_x$-norm of $x v(t)$.

\begin{proposition} \label{prop:growth_xv}
Let $v(t)$ be the solution to~\eqref{equ:nlkg_for_v} on the time interval $[0,T]$. Let $N(T)$ be defined as in~\eqref{equ:def_NT_setting_up} and assume $N(T) \leq 1$. Then we have 
 \begin{equation} \label{equ:growth_xv}
  \sup_{0 \leq t \leq T} \, \jt^{-1-\delta} \| x v(t) \|_{L^2_x} \lesssim \|\jx v_0\|_{L^2_x} + \|v_0\|_{H^1_x}^2 + N(T)^2.
 \end{equation}
\end{proposition}
\begin{proof}
 Starting from the Duhamel representation~\eqref{equ:duhamel_v_renorm} for $v(t)$ and using that 
 \begin{equation*}
  x e^{it\jD} = e^{it\jD} (x+it\px\jD^{-1}),
 \end{equation*}
 we obtain for any time $0 \leq t \leq T$ that 
 \begin{equation*}
  \begin{aligned}
   \|xv(t)\|_{L^2_x} &\lesssim \|x v_0\|_{L^2_x} + t \|v_0\|_{L^2_x} + \|x B(v,v)(0)\|_{L^2_x} + t \|B(v,v)(0)\|_{L^2_x} + \|x B(v,v)(t)\|_{L^2_x} \\
   &\quad + \int_0^t \bigl\| x e^{i(t-s)\jD} \jD^{-1} \calQ_{ren}(v,v)(s)\bigr\|_{L^2_x} \, \ud s + \int_0^t \bigl\| x e^{i(t-s)\jD} \jD^{-1} \calC(v+\bv)(s)\bigr\|_{L^2_x} \, \ud s \\
   &\quad + \int_0^t \bigl\| x e^{i(t-s)\jD} \jD^{-1} \calR_1(v+\bv)(s)\bigr\|_{L^2_x} \, \ud s + \int_0^t \bigl\| x e^{i(t-s)\jD} \jD^{-1} \calR_2(v+\bv)(s)\bigr\|_{L^2_x} \, \ud s.
  \end{aligned}
 \end{equation*}
 For the contributions of the initial data and of the normal form, we have 
 \begin{equation*}
  \begin{aligned}
    &\|x v_0\|_{L^2_x} + t \|v_0\|_{L^2_x} + \|x B(v,v)(0)\|_{L^2_x} + t \|B(v,v)(0)\|_{L^2_x} + \|x B(v,v)(t)\|_{L^2_x} \\
    &\lesssim \jt \|\jx v_0\|_{L^2_x} + \jt \sum_{k=1}^3 \|\jx \alpha_{1k}\|_{L^2_x} |v(0,0)|^2 + \sum_{k=1}^3 \|x \alpha_{1k}\|_{L^2_x} |v(t,0)|^2 \\
    &\lesssim \jt \|\jx v_0\|_{L^2_x} + \jt \|v_0\|_{H^1_x}^2 + N(T)^2 \jt^{-1}.
  \end{aligned}
 \end{equation*}
 Next, we estimate the contributions of all nonlinearities. Throughout we only consider times $0 \leq s \leq t \leq T$. 

 \medskip  
 
 \noindent {\it Renormalized quadratic nonlinearities:}   
 Here we first crudely bound 
 \begin{align*}
  \int_0^t \bigl\| x e^{i(t-s)\jD} \jD^{-1} \calQ_{ren}(v,v)(s)\bigr\|_{L^2_x} \, \ud s &\lesssim \int_0^t \bigl\| \bigl( x + i(t-s) \px \jD^{-1}\bigr) \jD^{-1} \calQ_{ren}(v,v)(s)\bigr\|_{L^2_x} \, \ud s \\
  &\lesssim \jt \int_0^t \|\jx \calQ_{ren}(v,v)(s)\|_{L^2_x} \, \ud s.
 \end{align*}
 Then we exploit the spatial localization of all quadratic nonlinearities together with the improved decay estimate~\eqref{equ:improved_decay_pt_phase_filtered_v}, the estimate~\eqref{equ:aux_bound_Linfty_wtIpxv}, and the improved local decay estimates from Lemma~\ref{lem:improved_local_decay}, to obtain 
 \begin{align*}
  &\|\jx \calQ_{ren}(v,v)(s)\|_{L^2_x} \\
  &\lesssim \sum_{k=1}^3 \|\jx \jD \alpha_{1k}\|_{L^2_x} | \ps (e^{-is} v(s,0) ) | |v(s,0)| + \| \jx^2 \alpha_1(x) \|_{L^\infty_x} \bigl\| \jx^{-1} \bigl( w(s)^2 - w(s,0)^2 \bigr) \bigr\|_{L^2_x} \\
  &\quad + \| \jx^2 \alpha_2(x) \|_{L^\infty_x} \bigl\| \jx^{-1} \wtcalI[\px v(s)] \bigr\|_{L^2_x} \|v(s)\|_{L^\infty_x} + \|\jx^2\alpha_3(x)\|_{L^\infty_x} \bigl\| \jx^{-1} \wtcalI[\px v(s)] \bigr\|_{L^2_x} \|\wtcalI[\px v(s)]\|_{L^\infty_x} \\
  &\lesssim N(T)^2 \js^{-(\thf-\delta)}.
 \end{align*}
 Thus, we find that
 \begin{equation*}
  \begin{aligned}
   \int_0^t \bigl\| x e^{i(t-s)\jD} \jD^{-1} \calQ_{ren}(v,v)(s)\bigr\|_{L^2_x} \, \ud s 
   \lesssim \jt \int_0^t N(T)^2 \js^{-(\thf-\delta)} \, \ud s \lesssim N(T)^2 \jt.
  \end{aligned}
 \end{equation*}

 \medskip 
 
 \noindent {\it Cubic nonlinearities:}    
 We begin by estimating the contributions of the cubic nonlinearities by
 \begin{align*}
  \int_0^t \bigl\| x e^{i(t-s)\jD} \jD^{-1} \calC(v+\bv)(s) \bigr\|_{L^2_x} \, \ud s &\lesssim \int_0^t \bigl\| \bigl( x + i(t-s) \px \jD^{-1}\bigr) \jD^{-1} \calC(v+\bv)(s) \bigr\|_{L^2_x} \, \ud s \\
  &\lesssim \int_0^t \|\jx \calC(v+\bv)(s)\|_{L^2_x} \, \ud s + \jt \int_0^t \|\calC(v+\bv)(s)\|_{L^2_x} \, \ud s.
 \end{align*} 
 Using the following bounds 
 \begin{align*}
  \|\jx \calI[v(s)]\|_{L^2_x} &\lesssim \|\jx v(s)\|_{L^2_x} \lesssim N(T) \js^{1+\delta}, \\
  \|\calI[v(s)]\|_{L^\infty_x} &\lesssim \|v(s)\|_{L^\infty_x} \lesssim N(T) \js^{-\hf},
 \end{align*}
 established in Lemma~\ref{lem:aux_bounds_Ioperators}, we can then estimate all cubic terms in the same manner by
 \begin{align*}
  \|\jx \calC(v+\bv)(s)\|_{L^2_x} \lesssim \|\jx \calI[v(s)]\|_{L^2_x} \|v(s)\|_{L^\infty_x}^2 \lesssim N(T)^3 \js^\delta
 \end{align*}
 as well as 
 \begin{align*}
  \|\calC(v+\bv)(s)\|_{L^2_x} \lesssim \|\calI[v(s)]\|_{L^2_x} \|v(s)\|_{L^\infty_x}^2 \lesssim N(T)^3 \js^{-(1-\delta)}.
 \end{align*}
 Hence, we obtain that 
 \begin{equation}
  \begin{aligned}
   \int_0^t \bigl\| x e^{i(t-s)\jD} \jD^{-1} \calC(v+\bv)(s) \bigr\|_{L^2_x} \, \ud s &\lesssim \int_0^t N(T)^3 \js^\delta \, \ud s + \jt \int_0^t N(T)^3 \js^{-(1-\delta)} \, \ud s \\
   &\lesssim N(T)^3 \jt^{1+\delta}.
  \end{aligned}
 \end{equation}
 
 \medskip 
 
 \noindent {\it Quartic and quintic nonlinearities:}     
 Proceeding as in the treatment of the contributions of the cubic terms, we find that 
 \begin{align*}
  &\int_0^t \bigl\| x e^{i(t-s)\jD} \jD^{-1} \calR_1(v+\bv)(s)\bigr\|_{L^2_x} \, \ud s + \int_0^t \bigl\| x e^{i(t-s)\jD} \jD^{-1} \calR_2(v+\bv)(s)\bigr\|_{L^2_x} \, \ud s \\
  &\lesssim N(T)^4 \jt^{\hf+\delta} + N(T)^5 \jt^\delta.
 \end{align*}

 \medskip 
 
 Collecting all of the preceding estimates and recalling that we assume $N(T) \leq 1$, we conclude for any time $0 \leq t \leq T$ that 
 \begin{equation*}
  \|xv(t)\|_{L^2_x} \lesssim \jt \|\jx v_0\|_{L^2_x} + \jt \|v_0\|_{H^1_x}^2 + N(T)^2 \jt^{1+\delta},
 \end{equation*}
 which proves the asserted bound~\eqref{equ:growth_xv}.
\end{proof}

Finally, we turn to the most delicate energy estimate, namely the derivation of a slow growth estimate for the $L^2_x$-norm of the operator $\jD L$ applied to the solution $v(t)$ to~\eqref{equ:nlkg_for_v}.
A key ingredient for the proof is the following proposition that establishes a slow growth estimate for the contribution of any spatially localized nonlinearity with at least cubic-type decay $\jt^{-(\thf-\delta)}$. The idea of the proof is a version of an argument used in~\cite{LLS1, LLS2}. It crucially relies on improved local decay estimates for the Klein-Gordon propagator.

\begin{proposition} \label{prop:key_slow_growth}
 Let $T > 0$ and $0 < \delta \ll 1$. Assume that 
 \begin{equation} \label{equ:key_slow_growth_input_assumption}
  \sup_{0 \leq t \leq T} \, \jt^{\thf-\delta} \bigl\| \jx^2 \jD \calN(t) \bigr\|_{L^2_x} \leq A
 \end{equation}
 for some $A > 0$.
 Then we have 
 \begin{equation} \label{equ:key_slow_growth}
  \sup_{0 \leq t \leq T} \, \jt^{-\delta} \biggl\| \jD L \int_0^t e^{i(t-s)\jD} \jD^{-1} \calN(s) \, \ud s \biggr\|_{L^2_x} \lesssim A.
 \end{equation}
\end{proposition}
\begin{proof}
 By Plancherel's theorem and~\eqref{equ:relation_L_partial_xi}, we have for any $0 \leq t \leq T$ that
 \begin{equation} \label{equ:key_slow_growth_start}
 \begin{aligned}
  &\biggl\| \jD L \int_0^t e^{i(t-s)\jD} \jD^{-1} \calN(s) \, \ud s \biggr\|_{L^2_x} \\
  &= \biggl\| e^{it\jxi} \jxi^2 i\partial_\xi \int_0^t e^{-i s \jxi} \jxi^{-1} \widehat{\calN}(s,\xi) \, \ud s \biggr\|_{L^2_\xi} \\
  &\leq \biggl\| \int_0^t s \xi \jxi^{-1} e^{-i s \jxi} \jxi \widehat{\calN}(s,\xi) \, \ud s \biggr\|_{L^2_\xi} + \biggl\| \int_0^t e^{-i s \jxi} \jxi^2 i \pxi \bigl( \jxi^{-1} \widehat{\calN}(s,\xi) \bigr) \, \ud s \biggr\|_{L^2_\xi}.
 \end{aligned}
 \end{equation}
 The second term on the right-hand side can be bounded uniformly for all $0 \leq t \leq T$ by
 \begin{equation} \label{equ:key_slow_growth_1st_term_bound}
 \begin{aligned}
  \biggl\| \int_0^t e^{-i s \jxi} \jxi^2 i \pxi \bigl( \jxi^{-1} \widehat{\calN}(s,\xi) \bigr) \, \ud s \biggr\|_{L^2_\xi} &\lesssim \int_0^t \bigl\| e^{-i s \jxi} \jxi^2 i \pxi \bigl( \jxi^{-1} \widehat{\calN}(s,\xi) \bigr) \bigr\|_{L^2_\xi} \, \ud s \\
  &\lesssim \int_0^t \bigl\| \jx \jD \calN(s) \bigr\|_{L^2_x} \, \ud s \\
  &\lesssim \int_0^t \frac{A}{\js^{\thf-\delta}} \, \ud s \lesssim A.
 \end{aligned}
 \end{equation}
 In order to estimate the growth in time of the first term on the right-hand side of~\eqref{equ:key_slow_growth_start}, we compute
 \begin{align*}
  &\pt \, \Biggl( \biggl\| \int_0^t s \xi \jxi^{-1} e^{-i s \jxi} \jxi \widehat{\calN}(s,\xi) \, \ud s \biggr\|_{L^2_\xi}^2 \Biggr) \\
  &= 2 \Re \, \int_{\bbR} \biggl( \int_0^t \overline{ s \xi \jxi^{-1} e^{-is\jxi} \jxi \widehat{\calN}(s,\xi) } \, \ud s \biggr) \, t \xi \jxi^{-1} e^{-it\jxi} \jxi \widehat{\calN}(t,\xi) \, \ud \xi \\
  &= 2 \Re \, \int_0^t s t \biggl( \int_{\bbR} \overline{ \xi \jxi^{-1} e^{i(t-s)\jxi} \jxi \widehat{\calN}(s,\xi) } \, \xi \widehat{\calN}(t,\xi) \, \ud \xi \biggr) \, \ud s.
 \end{align*}
 Using Parseval's theorem, the Cauchy-Schwarz inequality, and the improved local decay estimate~\eqref{equ:local_decay_van1}, we obtain uniformly for all $0 \leq t \leq T$ that
 \begin{align*}
  &\Biggl| \pt \, \Biggl( \biggl\| \int_0^t s \xi \jxi^{-1} e^{-i s \jxi} \jxi \widehat{\calN}(s,\xi) \, \ud s \biggr\|_{L^2_\xi}^2 \Biggr) \Biggr| \\
  &\lesssim \int_0^t s t \, \bigl\| \jx^{-2} \px \jD^{-1} e^{i(t-s)\jD} \jD \calN(s) \bigr\|_{L^2_x} \bigl\| \jx^2 \px \calN(t) \bigr\|_{L^2_x} \, \ud s \\
  &\lesssim \int_0^t s t \, \frac{1}{\jap{t-s}^\thf} \bigl\| \jx^2 \jD \calN(s) \bigr\|_{L^2_x} \bigl\| \jx^2 \px \calN(t) \bigr\|_{L^2_x} \, \ud s. 
 \end{align*}
 Invoking the assumption~\eqref{equ:key_slow_growth_input_assumption}, we may further estimate the last line by
 \begin{align*}
  \int_0^t s t \, \frac{1}{\jap{t-s}^\thf} \frac{A}{\js^{\thf-\delta}} \frac{A}{\jt^{\thf-\delta}} \, \ud s \lesssim \frac{A^2}{\jt^{\hf-\delta}} \int_0^t \frac{1}{\jap{t-s}^\thf} \frac{1}{\js^{\hf-\delta}} \, \ud s \lesssim \frac{A^2}{\jt^{1-2\delta}}.
 \end{align*}
 Integrating in time we infer uniformly for all $0 \leq t \leq T$ that
 \begin{equation} \label{equ:key_slow_growth_2nd_term_bound}
  \begin{aligned}
   \biggl\| \int_0^t s \xi \jxi^{-1} e^{-i s \jxi} \jxi \widehat{\calN}(s,\xi) \, \ud s \biggr\|_{L^2_\xi}^2 \lesssim A^2 \jt^{2\delta}.    
  \end{aligned}
 \end{equation}
 Combining the bounds~\eqref{equ:key_slow_growth_1st_term_bound} and~\eqref{equ:key_slow_growth_2nd_term_bound} yields the asserted slow growth estimate~\eqref{equ:key_slow_growth}.
\end{proof}

\begin{proposition} \label{prop:growth_H1Lv}
Let $v(t)$ be the solution to~\eqref{equ:nlkg_for_v} on the time interval $[0,T]$. Let $N(T)$ be defined as in~\eqref{equ:def_NT_setting_up} and assume $N(T) \leq 1$. Then we have 
 \begin{equation} \label{equ:growth_H1Lv}
  \sup_{0 \leq t \leq T} \, \jt^{-\delta} \| \jD L v(t) \|_{L^2_x} \lesssim \| \jx v_0 \|_{H^2_x} + \|v_0\|_{H^1_x}^2 + N(T)^2.
 \end{equation}
\end{proposition}
\begin{proof}
Our strategy is to use the slow growth estimate from Proposition~\ref{prop:key_slow_growth} to estimate the contributions of all spatially localized nonlinearities that exhibit (at least) cubic-type decay $\jt^{-(\thf-\delta)}$. All renormalized quadratic nonlinearities fall into this category as well as all quartic nonlinearities and those parts of the cubic nonlinearities that are spatially localized. Thus, only the non-localized cubic nonlinearities and the quintic remainder terms require a more specific treatment.

To this end we first examine the structure of the cubic nonlinearities a bit more and peel off further spatially localized cubic terms.
It will turn out that the remaining non-localized cubic nonlinearities have a favorable structure.
Using the identity 
\begin{equation*}
 \calI[w(t)](x) = -\tanh(x) w(t,x) + \wtcalI[\px w(t)](x),
\end{equation*}
we can rewrite the (not obviously localized) cubic nonlinearities $\calC_{nl}(w)$ defined in~\eqref{equ:def_nonlinearities_Cnl} and peel off a few more localized cubic terms. Specifically, we find that 
\begin{align*}
 \calC_{nl}(w) &= \frac12 \bigl( \calI[w] \bigr)^2 w + \frac13 \tanh(x) \bigl( \calI[w] \bigr)^3 \\
 &= \frac12 \bigl( -\tanh(x) w + \wtcalI[\px w] \bigr)^2 w + \frac13 \tanh(x) \bigl( - \tanh(x) w + \wtcalI[\px w] \bigr)^3 \\
 &= \frac16 w^3 - \frac12 \bigl( \wtcalI[\px w] \bigr)^2 w + \frac13 \tanh(x) \bigl( \wtcalI[\px w] \bigr)^3 \\
 &\quad + \frac16 \sech^2(x) w^3 -\frac13 \sech^4(x) w^3 - \sech^2(x) \tanh(x) \wtcalI[\px w] w^2 + \sech^2(x) \bigl( \wtcalI[\px w] \bigr)^2 w.
\end{align*}
Clearly, the last four terms are spatially localized. 
Within this proof we therefore use the following decomposition of the cubic nonlinearities into
\begin{equation*}
 \calC(w) = \widetilde{\calC}_{nl}(w) + \widetilde{\calC}_l(w)
\end{equation*}
with 
\begin{equation} \label{equ:growth_H1Lv_def_calCnl}
 \begin{aligned}
  \widetilde{\calC}_{nl}(w) := \frac16 w^3 - \frac12 \bigl( \wtcalI[\px w] \bigr)^2 w + \frac13 \tanh(x) \bigl( \wtcalI[\px w] \bigr)^3
 \end{aligned}
\end{equation}
and
\begin{equation} \label{equ:growth_H1Lv_def_calCl}
 \begin{aligned}
  \widetilde{\calC}_l(w) &:= \frac16 \sech^2(x) w^3 - \frac13 \sech^4(x) w^3 - \sech^2(x) \tanh(x) \wtcalI[\px w] w^2 \\
 &\quad + \sech^2(x) \bigl( \wtcalI[\px w] \bigr)^2 w - \frac43 \sech^2(x) \tanh(x) \bigl( \calI[w] \bigr)^3 - \sech^2(x) \bigl( \calI[w] \bigr)^2 w.
 \end{aligned}
\end{equation}
The last two terms in the preceding definition of $\widetilde{\calC}_l(w)$ stem from $\calC_l(w)$ defined in~\eqref{equ:def_nonlinearities_Cl}.
We write the Duhamel representation~\eqref{equ:duhamel_v_renorm} of the solution $v(t)$ to~\eqref{equ:nlkg_for_v} as
\begin{equation*}
 \begin{aligned}
  v(t) &= e^{it\jD} \bigl( v_0 + B(v, v)(0) \bigr) - B(v,v)(t) \\
  &\quad + v_{\calQ_{ren}}(t) + v_{\widetilde{\calC}_{nl}}(t) + v_{\widetilde{\calC}_l}(t) + v_{\calR_1}(t) + v_{\calR_2}(t),
 \end{aligned}
\end{equation*}
where we denote the contributions of the nonlinearities by
\begin{equation*}
 \begin{aligned}
  v_{\calQ_{ren}}(t) &= \frac{1}{2i} \int_0^t e^{i(t-s)\jD} \jD^{-1} \calQ_{ren}(v, v)(s) \, \ud s, \\
  v_{\widetilde{\calC}_{nl}}(t) &= \frac{1}{2i} \int_0^t e^{i(t-s)\jD} \jD^{-1} \widetilde{\calC}_{nl}(v + \bar{v})(s) \, \ud s, \\
  v_{\widetilde{\calC}_l}(t) &= \frac{1}{2i} \int_0^t e^{i(t-s)\jD} \jD^{-1} \widetilde{\calC}_l(v + \bar{v})(s) \, \ud s, \\
  v_{\calR_j}(t) &= \frac{1}{2i} \int_0^t e^{i(t-s)\jD} \jD^{-1} \calR_j(v+\bar{v})(s) \, \ud s, \quad j = 1,2.
 \end{aligned}
\end{equation*}
Throughout this proof we only consider times $0 \leq s \leq t \leq T$. We have 
\begin{equation} \label{equ:growth_H1Lv_initial_bound}
 \begin{aligned}
  \| \jD L v(t) \|_{L^2_x} &\lesssim \bigl\| \jD L \bigl( e^{it\jD} \bigl( v_0 + B(v, v)(0) \bigr) - B(v,v)(t) \bigr) \bigr\|_{L^2_x} \\
  &\quad + \|\jD L v_{\calQ_{ren}}(t)\|_{L^2_x} + \|\jD L v_{\widetilde{\calC}_{nl}}(t)\|_{L^2_x} + \|\jD L v_{\widetilde{\calC}_l}(t)\|_{L^2_x} \\
  &\quad + \|\jD L v_{\calR_1}(t)\|_{L^2_x} + \|\jD L v_{\calR_2}(t)\|_{L^2_x}.
 \end{aligned}
\end{equation}
Using the identity $\jD L e^{it\jD} = e^{it\jD} \jD^2 x$ and recalling that $\jD L = \jD^2 x - it \jD \px$, we can easily bound the first term on the right-hand side of \eqref{equ:growth_H1Lv_initial_bound} by 
\begin{equation*}
 \begin{aligned}
  &\bigl\| \jD L \bigl( e^{it\jD} \bigl( v_0 + B(v, v)(0) \bigr) - B(v,v)(t) \bigr) \bigr\|_{L^2_x} \\
  &\lesssim \|xv_0\|_{H^2_x} + \sum_{k=1}^3 \|x\alpha_{1k}\|_{H^2_x} \|v_0\|_{L^\infty_x}^2 + \sum_{k=1}^3 \bigl( \|x \alpha_{1k}\|_{H^2_x} + t \|\alpha_{1k}\|_{H^2_x} \bigr) \|v(t)\|_{L^\infty_x}^2 \\
  &\lesssim \|x v_0\|_{H^2_x} + \|v_0\|_{H^1_x}^2 + N(T)^2.
 \end{aligned}
\end{equation*}
The main work now goes into estimating the contributions of the nonlinearities on the right-hand side of~\eqref{equ:growth_H1Lv_initial_bound}. 

\medskip 

\noindent {\it Renormalized quadratic nonlinearities:} 
To bound the contributions of the renormalized quadratic nonlinearities using the key slow growth estimate from Proposition~\ref{prop:key_slow_growth}, we only have to verify that the assumption~\eqref{equ:key_slow_growth_input_assumption} in the statement of Proposition~\ref{prop:key_slow_growth} is satisfied. We have 
\begin{equation} \label{equ:growth_H1Lv_verify_assump_quadr1}
 \begin{aligned}
  \bigl\| \jx^2 \calQ_{ren}(v,v)(t) \bigr\|_{H^1_x} &\lesssim \sum_{k=1}^3 \, \bigl\| \jx^2 \calQ_{1k}(v,v)(t) \bigr\|_{H^1_x} + \bigl\| \jx^2 \calQ_{14}(v + \bv)(t) \bigr\|_{H^1_x} \\
  &\quad \quad + \bigl\| \jx^2 \calQ_2(v + \bv)(t) \bigr\|_{H^1_x} + \bigl\| \jx^2 \calQ_3(v + \bv)(t) \bigr\|_{H^1_x}.
 \end{aligned}
\end{equation}
Then we use the improved decay estimate~\eqref{equ:improved_decay_pt_phase_filtered_v} to bound the first term on the right-hand side of~\eqref{equ:growth_H1Lv_verify_assump_quadr1} by 
\begin{equation}
 \begin{aligned}
  \sum_{k=1}^3 \, \bigl\| \jx^2 \calQ_{1k}(v,v)(t) \bigr\|_{H^1_x} \lesssim \sum_{k=1}^3 \|\jx^2 \alpha_{1k}\|_{H^1_x} |\pt (e^{-it} v(t,0))| |v(t,0)| \lesssim N(T) \jt^{-(\thf-\delta)}.
 \end{aligned}
\end{equation}
For the second term on the right-hand side of~\eqref{equ:growth_H1Lv_verify_assump_quadr1} we invoke the improved local decay estimates from Lemma~\ref{lem:improved_local_decay} to obtain
\begin{equation}
 \begin{aligned}
  \bigl\| \jx^2 \calQ_{14}(v + \bv)(t) \bigr\|_{H^1_x} &\lesssim \bigl\| \jx^2 \alpha_1(x) \bigl( w(t)^2 - w(t,0)^2 \bigr) \bigr\|_{H^1_x} \\
  &\lesssim \|\jx^4 \alpha_1\|_{W^{1,\infty}_x} \bigl\| \jx^{-2} \bigl( w(t)^2 - w(t,0)^2 \bigr) \bigr\|_{L^2_x} \\
  &\quad + \|\jx^3 \alpha_1\|_{L^\infty_x} \|\jx^{-1} \px v(t)\|_{L^2_x} \|v(t)\|_{L^\infty_x} \\
  &\lesssim N(T)^2 \jt^{-(\thf-\delta)}.
 \end{aligned}
\end{equation}
Similarly, the third term on the right-hand side of~\eqref{equ:growth_H1Lv_verify_assump_quadr1} can be bounded using the improved local decay estimates from Lemma~\ref{lem:improved_local_decay} as well as~\eqref{equ:aux_bound_Linfty_wtIpxv} by
\begin{equation}
 \begin{aligned}
  \bigl\| \jx^2 \calQ_2(v + \bv)(t) \bigr\|_{H^1_x} &\lesssim \bigl\| \jx^2 \alpha_2(x) \wtcalI[\px w(t)] w(t) \bigr\|_{H^1_x} \\
  &\lesssim \|\jx^3 \alpha_2\|_{W^{1,\infty}_x} \bigl\|\jx^{-1} \wtcalI[\px v(t)]\bigr\|_{L^2_x} \|v(t)\|_{L^\infty_x} \\
  &\quad + \|\jx^3\alpha_2\|_{L^\infty_x} \bigl\|\jx^{-1} \px \wtcalI[\px v(t)]\bigr\|_{L^2_x} \|v(t)\|_{L^\infty_x} \\
  &\quad + \|\jx^3\alpha_2\|_{L^\infty_x} \bigl\|\wtcalI[\px v(t)]\bigr\|_{L^\infty_x} \|\jx^{-1} \px v(t)\|_{L^2_x} \\
  &\lesssim N(T)^2 \jt^{-(\thf-\delta)},
 \end{aligned}
\end{equation}
and proceeding analogously, we obtain 
\begin{equation}
 \begin{aligned}
  \bigl\| \jx^2 \calQ_3(v + \bv)(t) \bigr\|_{H^1_x} \lesssim \bigl\| \jx^2 \alpha_3(x) \bigl(\wtcalI[\px w] \bigr)^2 \bigr\|_{H^1_x} \lesssim N(T)^2 \jt^{-(\thf-\delta)}.
 \end{aligned}
\end{equation}
Thus, we have 
\begin{equation*}
 \sup_{0 \leq t \leq T} \, \bigl\| \jx^2 \calQ_{ren}(v,v)(t) \bigr\|_{H^1_x} \lesssim N(T)^2 \jt^{-(\thf-\delta)}
\end{equation*}
and Proposition~\ref{prop:key_slow_growth} yields the desired slow growth estimate for the contributions of all renormalized quadratic nonlinearities
\begin{equation*}
 \sup_{0 \leq t \leq T} \, \jt^{-\delta} \bigl\| \jD L v_{\calQ_{ren}}(t) \bigr\|_{L^2_x} \lesssim N(T)^2.
\end{equation*}

\medskip 

\noindent {\it Localized cubic nonlinearities $\widetilde{\calC}_l(v+\bv)$:} 
Here we again only need to verify that the assumption~\eqref{equ:key_slow_growth_input_assumption} in the statement of Proposition~\ref{prop:key_slow_growth} is satisfied by all localized cubic nonlinearities. 
We have 
\begin{equation} \label{equ:growth_H1Lv_local_cubic1}
 \begin{aligned}
  &\bigl\| \jx^2 \widetilde{\calC}_l(v+\bv)(t)\bigr\|_{H^1_x} \\
  &\lesssim \bigl\| \jx^2 \sech^2(x) w(t)^3 \bigr\|_{H^1_x} + \bigl\| \jx^2 \sech^4(x) w(t)^3 \bigr\|_{H^1_x} \\ 
  &\quad + \bigl\| \jx^2 \sech^2(x) \tanh(x) \wtcalI[\px w(t)] w(t)^2 \bigr\|_{H^1_x} + \bigl\| \jx^2 \sech^2(x) \bigl( \widetilde{\calI}[\px w(t)] \bigr)^2 w(t) \bigr\|_{H^1_x} \\
  &\quad + \bigl\| \jx^2 \sech^2(x) \tanh(x) \bigl( \calI[w(t)] \bigr)^3 \bigr\|_{H^1_x} + \bigl\| \jx^2 \sech^2(x) \bigl( \calI[w(t)] \bigr)^2 w(t) \bigr\|_{H^1_x}. 
 \end{aligned}
\end{equation}
Using the improved local decay estimate~\eqref{equ:improved_local_decay_px_v} the first term on the right-hand side of~\eqref{equ:growth_H1Lv_local_cubic1} can then be bounded by
\begin{equation*}
 \begin{aligned}
  \bigl\| \jx^2 \sech^2(x) w(t)^3 \bigr\|_{H^1_x} &\lesssim \| \jx^2 \sech^2(x) \|_{H^1_x} \|v(t)\|_{L^\infty_x}^3 \\
  &\quad + \bigl\| \jx^3 \sech^2(x) \bigr\|_{L^\infty_x} \|\jx^{-1} \px v(t)\|_{L^2_x} \|v(t)\|_{L^\infty_x}^2 \\
  &\lesssim N(T)^3 \jt^{-(\thf-\delta)}.
 \end{aligned}
\end{equation*}
The bound for the second term on the right-hand side of~\eqref{equ:growth_H1Lv_local_cubic1} is the same.
Similarly, using~\eqref{equ:aux_bound_Linfty_wtIpxv} as well as the improved local decay estimates~\eqref{equ:improved_local_decay_wtilIpxv} and~\eqref{equ:improved_local_decay_pxwtilIpxv}, we estimate the third term on the right-hand side of~\eqref{equ:growth_H1Lv_local_cubic1} by 
\begin{equation*}
 \begin{aligned}
  &\bigl\| \jx^2 \sech^2(x) \tanh(x) \wtcalI[\px w(t)] w(t)^2 \bigr\|_{H^1_x} \\
  &\lesssim \| \jx^2 \sech^2(x) \tanh(x) \|_{H^1_x} \|\wtcalI[\px v(t)]\|_{L^\infty_x} \|v(t)\|_{L^\infty_x}^2 \\
  &\quad + \| \jx^3 \sech^2(x) \tanh(x) \|_{L^\infty_x} \| \jx^{-1} \px \wtcalI[\px v(t)] \|_{L^2_x} \|v(t)\|_{L^\infty_x}^2 \\
  &\quad + \| \jx^3 \sech^2(x) \tanh(x) \|_{L^\infty_x} \| \wtcalI[\px v(t)] \|_{L^\infty_x} \|\jx^{-1} \px v(t)\|_{L^2_x} \|v(t)\|_{L^\infty_x} \\
  &\lesssim N(T)^3 \jt^{-(\thf-\delta)}.
 \end{aligned}
\end{equation*}
In a similar manner we can derive the desired bounds on the remaining three terms on the right-hand side of~\eqref{equ:growth_H1Lv_local_cubic1}.
By the slow growth estimate from Proposition~\ref{prop:key_slow_growth}, we therefore obtain the desired bound
\begin{equation}
 \sup_{0\leq t \leq T} \, \jt^{-\delta} \|\jD L v_{\widetilde{\calC}_l}(t)\|_{L^2_x} \lesssim N(T)^3.
\end{equation}

\medskip 

\noindent {\it Non-localized cubic nonlinearities $\widetilde{\calC}_{nl}(v+\bv)$:} 
In order to bound the contributions of the non-localized cubic nonlinearities, we express the non-local operator $L$ in terms of $Z$. Specifically, we write
\begin{equation*}
 \begin{aligned}
  \jD L = -i \jD Z - \px + i x \jD (\pt - i\jD) - i \jD^{-1} \px (\pt - i\jD).
 \end{aligned}
\end{equation*}
The Lorentz boost $Z$ satisfies a product rule, which allows us to easily compute its precise action on the non-localized cubic nonlinearities. We have 
\begin{equation} \label{equ:growth_H1Lv_initial_bound_jdL_nonlocal_cubic}
 \begin{aligned}
  \|\jD L v_{\widetilde{\calC}_{nl}}(t)\|_{L^2_x} &\lesssim \|\jD Z v_{\widetilde{\calC}_{nl}}(t)\|_{L^2_x} + \|\px v_{\widetilde{\calC}_{nl}}(t)\|_{L^2_x} \\
  &\quad + \|x \widetilde{\calC}_{nl}(v+\bv)(t) \|_{L^2_x} + \|\widetilde{\calC}_{nl}(v+\bv)(t)\|_{L^2_x}.
 \end{aligned}
\end{equation}
The main work goes into estimating the first term on the right-hand side of~\eqref{equ:growth_H1Lv_initial_bound_jdL_nonlocal_cubic}. We begin by dispensing of the easier other terms. In view of the definition of $\widetilde{\calC}_{nl}(v+\bv)(t)$ in~\eqref{equ:growth_H1Lv_def_calCnl}, using~\eqref{equ:aux_bound_Linfty_wtIpxv} and \eqref{equ:aux_bound_L2_jxwtcalIpxv}, we may bound the third term on the right-hand side of~\eqref{equ:growth_H1Lv_initial_bound_jdL_nonlocal_cubic} by
\begin{equation*}
 \begin{aligned}
  \|x \widetilde{\calC}_{nl}(v+\bv)(t) \|_{L^2_x} &\lesssim \bigl( \|xv(t)\|_{L^2_x} + \|x \wtcalI[\px v(t)]\|_{L^2_x} \bigr) \bigl( \|v(t)\|_{L^\infty_x}^2 + \|\wtcalI[\px v(t)]\|_{L^\infty_x}^2 \bigr) \\
  &\lesssim N(T)^3 \jt^\delta.
 \end{aligned}
\end{equation*}
Similarly, invoking~\eqref{equ:aux_bound_Linfty_wtIpxv} and \eqref{equ:aux_bound_L2_wtIpxv}, we estimate the last term on the right-hand side of~\eqref{equ:growth_H1Lv_initial_bound_jdL_nonlocal_cubic} by
\begin{equation*}
 \begin{aligned}
  \|\widetilde{\calC}_{nl}(v+\bv)(t)\|_{L^2_x} &\lesssim \bigl( \|v(t)\|_{L^2_x} + \|\wtcalI[\px v(t)]\|_{L^2_x} \bigr) \bigl( \|v(t)\|_{L^\infty_x}^2 + \|\wtcalI[\px v(t)]\|_{L^\infty_x}^2 \bigr) \\
  &\lesssim N(T)^3 \jt^{-(1-\delta)},
 \end{aligned}
\end{equation*}
which additionally gives rise to the following bound on the second term on the right-hand side of~\eqref{equ:growth_H1Lv_initial_bound_jdL_nonlocal_cubic},
\begin{equation*}
 \begin{aligned}
  \|\px v_{\widetilde{\calC}_{nl}}(t)\|_{L^2_x} \lesssim \int_0^t \|\widetilde{\calC}_{nl}(v+\bv)(s)\|_{L^2_x} \, \ud s &\lesssim \int_0^t N(T)^3 \js^{-(1-\delta)} \, \ud s \lesssim N(T)^3 \jt^\delta.
 \end{aligned}
\end{equation*}
Thus, we can now turn to estimating the first term on the right-hand side of~\eqref{equ:growth_H1Lv_initial_bound_jdL_nonlocal_cubic}. 
By the standard energy estimate we have 
\begin{equation*}
 \begin{aligned}
  \|\jD Z v_{\widetilde{\calC}_{nl}}(t)\|_{L^2_x} &\lesssim \int_0^t \bigl\| \bigl((\pt - i \jD) \jD Z v_{\widetilde{\calC}_{nl}}\bigr)(s) \bigr\|_{L^2_x} \, \ud s.
 \end{aligned}
\end{equation*}
Using the commutators~\eqref{equ:commutators}, we compute 
\begin{equation} \label{equ:growth_H1Lv_evol_equ_jdZ_vCnl}
 \begin{aligned}
  &(\pt - i\jD) (\jD Z v_{\widetilde{\calC}_{nl}}) \\
  &= \jD Z (\pt - i \jD) v_{\widetilde{\calC}_{nl}} + \jD \bigl[ (\pt - i \jD), Z \bigr] v_{\widetilde{\calC}_{nl}} \\
  &= \frac{1}{2i} Z \widetilde{\calC}_{nl}(v+\bv) + \frac{1}{2i} [ \jD, Z ] \jD^{-1} \widetilde{\calC}_{nl}(v+\bv) + \jD \bigl[ (\pt - i \jD), Z \bigr] v_{\widetilde{\calC}_{nl}} \\
  &= \frac{1}{2i} Z \widetilde{\calC}_{nl}(v+\bv) - \frac{1}{2i} \jD^{-2} \px \pt \widetilde{\calC}_{nl}(v+\bv) + \hf \px \jD^{-1} \widetilde{\calC}_{nl}(v+\bv),
 \end{aligned}
\end{equation}
whence 
\begin{equation} \label{equ:growth_H1Lv_bound_jDZ_vCnl1}
 \begin{aligned}
  \|\jD Z v_{\widetilde{\calC}_{nl}}(t)\|_{L^2_x} &\lesssim \int_0^t \bigl\| \bigl( Z \widetilde{\calC}_{nl}(v+\bv) \bigr)(s) \bigr\|_{L^2_x} \, \ud s + \int_0^t \bigl\| \bigl( \pt \widetilde{\calC}_{nl}(v+\bv) \bigr)(s) \bigr\|_{L^2_x} \, \ud s \\
  &\quad \quad + \int_0^t \| \widetilde{\calC}_{nl}(v+\bv)(s) \|_{L^2_x} \, \ud s.
 \end{aligned}
\end{equation}
In view of the definition~\eqref{equ:growth_H1Lv_def_calCnl} of $\widetilde{\calC}_{nl}(v+\bv)(s)$,
using~\eqref{equ:aux_bound_Linfty_wtIpxv}, \eqref{equ:aux_bound_L2_wtIpxv} and~\eqref{equ:aux_bound_L2_wtIpxptv} along with the auxiliary slow growth estimate~\eqref{equ:aux_slow_growth_jDptv}, we may bound the last two terms on the right-hand side of~\eqref{equ:growth_H1Lv_bound_jDZ_vCnl1} by
\begin{equation}
 \begin{aligned}
  &\int_0^t \bigl\| \bigl( \pt \widetilde{\calC}_{nl}(v+\bv) \bigr)(s) \bigr\|_{L^2_x} \, \ud s + \int_0^t \| \widetilde{\calC}_{nl}(v+\bv)(s) \|_{L^2_x} \, \ud s \\
  &\quad \lesssim \int_0^t \bigl( \|v(s)\|_{L^2_x} + \|\pt v(s)\|_{L^2_x} \bigr) \|v(s)\|_{L^\infty_x}^2 \, \ud s \\
  &\quad \lesssim \int_0^t N(T)^3 \js^{-(1-\delta)} \, \ud s \lesssim N(T)^3 \jt^\delta.
 \end{aligned}
\end{equation}
Finally, we can turn to the heart of the matter, namely the estimate of the first term on the right-hand side of~\eqref{equ:growth_H1Lv_bound_jDZ_vCnl1}. To this end we compute that
\begin{equation}
 \begin{aligned}
  Z \bigl( \widetilde{\calC}_{nl}(w) \bigr) &= \frac12 w^2 (Zw) - w \wtcalI[\px w] Z \bigl( \wtcalI[\px w] \bigr) - \frac12 \bigl( \wtcalI[\px w] \bigr)^2 (Zw) \\
  &\quad + \frac13 t \sech^2(x) \bigl( \wtcalI[\px w] \bigr)^3 + \tanh(x) \bigl( \wtcalI[\px w] \bigr)^2 Z \bigl( \wtcalI[\px w] \bigr).
 \end{aligned}
\end{equation}
Invoking the crucial slow growth bound~\eqref{equ:growth_L2_Z_action_wtcalIv} from Corollary~\ref{cor:growth_L2_Z_action_calIs} along with the improved local decay estimate~\eqref{equ:improved_local_decay_wtilIpxv}, the decay estimate~\eqref{equ:aux_bound_Linfty_wtIpxv}, and the auxiliary slow growth bound~\eqref{equ:aux_slow_growth_jDZv}, we then infer that
\begin{equation*}
 \begin{aligned}
  &\bigl\| \bigl( Z \widetilde{\calC}_{nl}(v+\bv) \bigr)(s) \bigr\|_{L^2_x} \\
  &\lesssim \|v(s)\|_{L^\infty_x}^2 \|Zv(s)\|_{L^2_x} + \|v(s)\|_{L^\infty_x} \|\wtcalI[\px v(s)]\|_{L^\infty_x} \bigl\| Z \bigl( \wtcalI[\px v(s)] \bigr)\bigr\|_{L^2_x} \\
  &\quad + \|\wtcalI[\px v(s)]\|_{L^\infty_x}^2 \|Zv(s)\|_{L^2_x} + s \|\jx \sech^2(x)\|_{L^\infty_x} \| \jx^{-1} \wtcalI[\px v(s)]\|_{L^2_x} \|\wtcalI[\px v(s)]\|_{L^\infty_x}^2 \\
  &\quad + \|\wtcalI[\px v(s)]\|_{L^\infty_x}^2 \bigl\| Z \bigl( \wtcalI[\px v(s)] \bigr) \bigr\|_{L^2_x} \\
  &\lesssim N(T)^3 \js^{-(1-\delta)}.
 \end{aligned}
\end{equation*}
It follows that the first-term on the right-hand side of~\eqref{equ:growth_H1Lv_bound_jDZ_vCnl1} satisfies the desired bound
\begin{equation*}
 \begin{aligned}
  \int_0^t \bigl\| \bigl( Z \widetilde{\calC}_{nl}(v+\bv) \bigr)(s) \bigr\|_{L^2_x} \, \ud s &\lesssim \int_0^t N(T)^3 \js^{-(1-\delta)} \lesssim N(T)^3 \jt^\delta.
 \end{aligned}
\end{equation*}
Combining all of the preceding estimates we arrive at the desired slow growth estimate
\begin{equation}
 \begin{aligned}
  \sup_{0 \leq t \leq T} \, \jt^{-\delta} \|\jD L v_{\widetilde{\calC}_{nl}}(t)\|_{L^2_x} \lesssim N(T)^3.
 \end{aligned}
\end{equation}

\medskip 

\noindent {\it Quartic nonlinearities:} 
Here we again only have to verify that the assumption~\eqref{equ:key_slow_growth_input_assumption} in the statement of Proposition~\ref{prop:key_slow_growth} is satisfied by the quartic nonlinearities.
Proceeding analogously to the treatment of the localized cubic nonlinearities, we find that 
\begin{equation*}
 \begin{aligned}
  \bigl\| \jx^2 \calR_1(v+\bv)(t)\bigr\|_{H^1_x} &\lesssim N(T)^4 \jt^{-2}.
 \end{aligned}
\end{equation*}
Correspondingly, Proposition~\ref{prop:key_slow_growth} gives the desired bound
\begin{equation}
 \sup_{0\leq t \leq T} \, \jt^{-\delta} \|\jD L v_{\calR_1}(t)\|_{L^2_x} \lesssim N(T)^4.
\end{equation}

\medskip 

\noindent {\it Quintic nonlinearities:} 
Finally, we estimate the contributions of the quintic nonlinearities by proceeding analogously to the preceding treatment of the non-localized cubic nonlinearities. Here we only describe how to obtain the desired bound on the crucial contribution of
\begin{equation*}
 \int_0^t \bigl\| \bigl( Z \calR_2(v+\bv)\bigr)(s) \bigr\|_{L^2_x} \, \ud s,
\end{equation*}
which is the analogue of the first term on the right-hand side of~\eqref{equ:growth_H1Lv_bound_jDZ_vCnl1}. The treatment of all other terms is analogous to, and in fact even simpler than in the case of the non-localized cubic nonlinearities. Recall from \eqref{equ:def_nonlinearities_quintic} that
\begin{equation*} 
 \begin{aligned}
 \calR_2(v+\bv) &= \sum_{k=1}^5 \calR_{2,k}(v+\bv).
 \end{aligned}
\end{equation*}
We consider in detail the contribution of the first quintic nonlinearity $\calR_{2,1}(v+\bv)$ and compute
\begin{equation}
 \begin{aligned}
  &- \frac{4!}{2} \bigl( Z \calR_{2,1}(w) \bigr)(s) \\
  &= - s \tanh(x) \sech(x) \biggl( \int_0^1 (1-r)^4 \sin\bigl( K + r \calI[w(s)] \bigr) \, \ud r \biggr) \bigl( \calI[w(s)] \bigr)^5 \\
  &\quad + \sech(x) \biggl( \int_0^1 (1-r)^4 \cos\bigl( K + r \calI[w(s)] \bigr) \bigl( s (\px K) + r Z\bigl( \calI[w(s)] \bigr) \bigr)  \, \ud r \biggr) \bigl( \calI[w(s)] \bigr)^5 \\
  &\quad + \sech(x) \biggl( \int_0^1 (1-r)^4 \sin\bigl( K + r \calI[w(s)] \bigr) \, \ud r \biggr) 5 \bigl( \calI[w(s)] \bigr)^4 Z \bigl( \calI[w(s)] \bigr).
 \end{aligned}
\end{equation}
Hence, not even relying on any spatial localization properties, we can just use the growth bound~\eqref{equ:growth_L2_Z_action_calIv} from Corollary~\ref{cor:growth_L2_Z_action_calIs} along with the estimates~\eqref{equ:aux_bound_Linfty_Iv} and \eqref{equ:aux_bound_L2_Iv}, to crudely estimate
\begin{equation}
 \begin{aligned}
  &\bigl\| Z \bigl( \calR_{2,1}(v+\bv)(s) \bigr\|_{L^2_x} \\
  &\lesssim s \|\calI[v(s)]\|_{L^2_x} \|\calI[v(s)]\|_{L^\infty_x}^4 + \bigl\| Z \bigl( \calI[v(s)] \bigr) \bigr\|_{L^2_x} \bigl( \| \calI[v(s)]\|_{L^\infty_x}^5 + \| \calI[v(s)]\|_{L^\infty_x}^4 \bigr) \\
  &\lesssim s N(T) \js^\delta N(T)^4 \js^{-2} + N(T) \js^\hf \bigl( N(T)^5 \js^{-\frac52} + N(T)^4 \js^{-2} \bigr) \\
  &\lesssim N(T)^5 \js^{-(1-\delta)},
 \end{aligned}
\end{equation}
which suffices to obtain the desired bound 
\begin{equation*}
 \int_0^t \bigl\| \bigl( Z \calR_{2, 1}(v+\bv)\bigr)(s) \bigr\|_{L^2_x} \, \ud s \lesssim N(T)^5 \jt^\delta.
\end{equation*}
A careful examination of the structure of the other quintic nonlinearities $\calR_{2,k}(v+\bv)(t)$, $2 \leq k \leq 5$, shows that those can all be estimated in the same manner. This concludes the treatment of the quintic nonlinearities and thus finishes the proof of the proposition.
\end{proof}

\section{Pointwise estimates for the profile} \label{sec:pointwise_estimates}

In this section we establish an a priori bound on the $L^\infty_\xi$-norm of the Fourier transform of the profile $f(t) := e^{-it\jD} v(t)$ of the solution to~\eqref{equ:nlkg_for_v}. 

\begin{proposition} \label{prop:Linftyxi_bound_profile}
 Let $f(t) := e^{-it\jD} v(t)$ be the profile of the solution $v(t)$ to~\eqref{equ:nlkg_for_v} on the time interval $[0,T]$ for some $T \geq 1$. Let $N(T)$ be defined as in~\eqref{equ:def_NT_setting_up} and assume $N(T) \leq 1$. We have 
 \begin{equation} \label{equ:Linftyxi_bound_profile}
  \sup_{1 \leq t \leq T} \, \bigl\| \jxi^\thf \hatf(t,\xi) \bigr\|_{L^\infty_\xi} \lesssim \bigl\|\jxi^\thf \hatf(1,\xi) \bigr\|_{L^\infty_\xi} + N(T)^2.
 \end{equation}
 Moreover, we obtain for arbitrary times $1 \leq t_1 \leq t_2 \leq T$ that
 \begin{equation} \label{equ:Linftyxi_bound_difference_profile}
 \begin{aligned}
  \bigl\| \jxi^\thf \hatf(t_2, \xi) e^{i \Phi(t_2, \xi)} - \jxi^\thf \hatf(t_1, \xi) e^{i \Phi(t_1, \xi)} \bigr\|_{L^\infty_\xi} \lesssim N(T)^2 t_1^{-\frac15 + 3\delta},
 \end{aligned}
 \end{equation}
 where 
 \begin{equation} \label{equ:def_integrating_factor}
  \Phi(t, \xi) := \frac14 \jxi^{-7} (1+3\xi^2) \int_1^t \frac{1}{s} \bigl| \jxi^\thf \hatf(s,\xi) \bigr|^2 \, \ud s, \quad 1 \leq t \leq T.
 \end{equation}
\end{proposition}

The main part of the proof of Proposition~\ref{prop:Linftyxi_bound_profile} consists in deriving the following differential equation that captures the asymptotic behavior of the Fourier transform of the profile of the solution.

\begin{proposition} \label{prop:ODE_profile}
 Assume $T \geq 1$. Let $f(t) := e^{-it\jD} v(t)$ be the profile of the solution $v(t)$ to~\eqref{equ:nlkg_for_v} on the time interval $[0,T]$. Let $N(T)$ be defined as in~\eqref{equ:def_NT_setting_up} and assume $N(T) \leq 1$. Then we have for all $\xi \in \bbR$ and all $1 \leq t \leq T$ that 
 \begin{equation} \label{equ:ode_profile}
  \begin{aligned}
   \pt \Bigl( \jxi^\thf \hatf(t,\xi) + \hatr(t,\xi) \Bigr) &= \frac{1}{t} \frac{1}{36 \sqrt{3}} e^{it(-\jxi + 3 \jap{\frac{\xi}{3}})} \jxi^{\hf} \jap{{\textstyle \frac{\xi}{3}}}^{-3} (3+\xi^2) \hat{f}\bigl(t, {\textstyle \frac{\xi}{3}}\bigr)^3 \\
   &\quad \quad + \frac{1}{4i t} \jxi^{-7} (1+3\xi^2) \bigl| \jxi^\thf \hat{f}(t,\xi) \bigr|^2 \jxi^\thf \hat{f}(t,\xi) \\
   &\quad \quad + \frac{1}{4i t}  e^{-2 i t \jxi} \jxi^{-\frac52} (1+3\xi^2) |\hat{f}(t, -\xi)|^2 \bar{\hat{f}}(t,-\xi) \\
   &\quad \quad - \frac{1}{t} \frac{1}{36 \sqrt{3}} e^{-it(\jxi + 3 \jap{\frac{\xi}{3}})} \jap{\xi}^{\hf} \jap{{\textstyle \frac{\xi}{3}}}^{-3} (3+\xi^2) \hat{\bar{f}}\bigl(t, {\textstyle \frac{\xi}{3}}\bigr)^3 \\
   &\quad \quad + \calO_{L^\infty_\xi}\bigl( N(T)^2 t^{-\frac65 + 3\delta} \bigr),
  \end{aligned}
 \end{equation}
 where 
 \begin{equation*}
  \|\hatr(t,\cdot)\|_{L^\infty_\xi} \lesssim N(T)^2 \jt^{-1}.
 \end{equation*}
\end{proposition}

Proposition~\ref{prop:ODE_profile} implies Proposition~\ref{prop:Linftyxi_bound_profile} by a standard argument, which we briefly sketch next.
The remainder of this section is then devoted to the proof of Proposition~\ref{prop:ODE_profile}.

\begin{proof}[Proof of Proposition~\ref{prop:Linftyxi_bound_profile}]
 The basic idea is to just integrate the differential equation~\eqref{equ:ode_profile} in time. Among the four terms on the right-hand side of~\eqref{equ:ode_profile} that have non-integrable $t^{-1}$ decay, all but the second (resonant) term exhibit additional oscillations in time that allow for uniform-in-time bounds. The second term on the right-hand side of~\eqref{equ:ode_profile} can be removed via an integrating factor, which leads to logarithmic phase corrections in the asymptotics of $v(t)$. 
 Correspondingly, we multiply~\eqref{equ:ode_profile} by the integrating factor $e^{i \Phi(t,\xi)}$ with $\Phi(t,\xi)$ defined in~\eqref{equ:def_integrating_factor} to obtain that
 \begin{equation} \label{equ:ODE_multiplied_by_integr_factor}
  \begin{aligned}
   \pt \Bigl( \jxi^\thf \hatf(t,\xi) e^{i\Phi(t,\xi)} + \hatr(t,\xi) e^{i\Phi(t,\xi)} \Bigr) = \sum_{k=1}^3 \widehat{G}_k(t,\xi) + \calO_{L^\infty_\xi} \bigl( N(T)^2 t^{-\frac65 + 3\delta} \bigr),
  \end{aligned}
 \end{equation}
 where 
 \begin{align*}
  \widehat{G}_1(t,\xi) &:= \frac{1}{t} \frac{1}{36 \sqrt{3}} e^{it(-\jxi + 3 \jap{\frac{\xi}{3}})} \jxi^{\hf} \jap{{\textstyle \frac{\xi}{3}}}^{-3} (3+\xi^2) \hat{f}\bigl(t, {\textstyle \frac{\xi}{3}}\bigr)^3 e^{i\Phi(t,\xi)}, \\
  \widehat{G}_2(t,\xi) &:= \frac{1}{4i t}  e^{-2 i t \jxi} \jxi^{-\frac52} (1+3\xi^2) |\hat{f}(t, -\xi)|^2 \bar{\hat{f}}(t,-\xi) e^{i\Phi(t,\xi)}, \\
  \widehat{G}_3(t,\xi) &:= - \frac{1}{t} \frac{1}{36 \sqrt{3}} e^{-it(\jxi + 3 \jap{\frac{\xi}{3}})} \jap{\xi}^{\hf} \jap{{\textstyle \frac{\xi}{3}}}^{-3} (3+\xi^2) \hat{\bar{f}}\bigl(t, {\textstyle \frac{\xi}{3}}\bigr)^3 e^{i\Phi(t,\xi)}.
 \end{align*}
 Then upon showing for $k = 1, 2, 3$ that uniformly for all $1 \leq t_1 \leq t_2 \leq T$,
 \begin{equation} \label{equ:bound_integrated_widehatG}
 \begin{aligned}
  \biggl\| \int_{t_1}^{t_2} \widehat{G}_k(s,\xi) \, \ud s \biggr\|_{L^\infty_\xi} \lesssim \frac{1}{t_1^{\frac12-2\delta}} N(T)^3,
 \end{aligned}
 \end{equation}
 the asserted estimates~\eqref{equ:Linftyxi_bound_profile} and~\eqref{equ:Linftyxi_bound_difference_profile} follow from integrating~\eqref{equ:ODE_multiplied_by_integr_factor} in time and taking the $L^\infty_\xi$ norm. 
 We demonstrate in detail how to prove the bound~\eqref{equ:bound_integrated_widehatG} for $k=1$. To exploit the time oscillations in the term $\widehat{G}_1(t,\xi)$, we rewrite it as
 \begin{equation*}
 \begin{aligned}
  \widehat{G}_1(t,\xi) &=  \pt \biggl( \frac{1}{t} \frac{(-i)}{36 \sqrt{3}} e^{it(-\jxi + 3 \jap{\frac{\xi}{3}})} \bigl( -\jxi + 3 \jap{{\textstyle \frac{\xi}{3}}} \bigr)^{-1}   \jxi^{\hf} \jap{{\textstyle \frac{\xi}{3}}}^{-3} (3+\xi^2) \hat{f}\bigl(t, {\textstyle \frac{\xi}{3}}\bigr)^3 e^{i\Phi(t,\xi)} \biggr) \\
  &\quad + \frac{1}{t^2} \frac{(-i)}{36 \sqrt{3}} e^{it(-\jxi + 3 \jap{\frac{\xi}{3}})} \bigl( -\jxi + 3 \jap{{\textstyle \frac{\xi}{3}}} \bigr)^{-1} \jxi^{\hf} \jap{{\textstyle \frac{\xi}{3}}}^{-3} (3+\xi^2) \hat{f}\bigl(t, {\textstyle \frac{\xi}{3}}\bigr)^3 e^{i\Phi(t,\xi)} \\
  &\quad - \frac{1}{t} \frac{(-i)}{12 \sqrt{3}} e^{it(-\jxi + 3 \jap{\frac{\xi}{3}})} \bigl( -\jxi + 3 \jap{{\textstyle \frac{\xi}{3}}} \bigr)^{-1} \jxi^{\hf} \jap{{\textstyle \frac{\xi}{3}}}^{-3} (3+\xi^2) \hat{f}\bigl(t, {\textstyle \frac{\xi}{3}} \bigr)^2 \\
  &\qquad \qquad \qquad \qquad \qquad \qquad \qquad \qquad \qquad \qquad \qquad \qquad \times \pt \hat{f}\bigl(t, {\textstyle \frac{\xi}{3}}\bigr) e^{i\Phi(t,\xi)} \\
  &\quad + \frac{1}{t^2} \frac{1}{144 \sqrt{3}} e^{it(-\jxi + 3 \jap{\frac{\xi}{3}})} \bigl( -\jxi + 3 \jap{{\textstyle \frac{\xi}{3}}} \bigr)^{-1} \jxi^{-\frac{13}{2}} \jap{{\textstyle \frac{\xi}{3}}}^{-3} (3+\xi^2)  \\
  &\qquad \qquad \qquad \qquad \qquad \qquad \qquad \qquad \times (1+3\xi^2) \hat{f}\bigl(t, {\textstyle \frac{\xi}{3}}\bigr)^3 \bigl| \jxi^\thf \hatf(t,\xi) \bigr|^2 e^{i\Phi(t,\xi)}.
 \end{aligned}
 \end{equation*}
 Then using that $(-\jxi + 3 \jap{\frac{\xi}{3}})^{-1} \simeq \jxi$ and that~\eqref{equ:FT_profile_equation} implies the crude estimate 
 \begin{equation*}
  \bigl\| \pt \hatf(t,\xi) \bigr\|_{L^\infty_\xi} \lesssim N(T)^2 \jt^{-\hf + 2\delta}, \quad 0 \leq t \leq T,
 \end{equation*}
 we conclude for $1 \leq t_1 \leq t_2 \leq T$,
 \begin{equation*}
 \begin{aligned}
  \biggl\| \int_{t_1}^{t_2} \widehat{G}_1(s,\xi) \, \ud s \biggr\|_{L^\infty_\xi} &\lesssim \frac{1}{t_1} \sup_{1 \leq t \leq T} \, \bigl\| \jxi^\thf \hatf(t,\xi) \bigr\|_{L^\infty_\xi}^3 + \int_{t_1}^{t_2} \frac{1}{s^2} \bigl\| \jxi^\thf \hatf(s,\xi) \bigr\|_{L^\infty_\xi}^3 \, \ud s \\
  &\quad + \int_{t_1}^{t_2} \frac{1}{s^{\thf-2\delta}} \bigl\| \jxi^\thf \hatf(s,\xi) \bigr\|_{L^\infty_\xi}^2 \js^{\frac12-2\delta} \bigl\| \pt \hatf(s,\xi) \bigr\|_{L^\infty_\xi} \, \ud s \\
  &\quad + \int_{t_1}^{t_2} \frac{1}{s^2} \bigl\| \jxi^\thf \hatf(s,\xi) \bigr\|_{L^\infty_\xi}^3 \, \ud s \\
  &\lesssim \frac{1}{t_1^{\frac12-2\delta}} N(T)^3.
 \end{aligned}
 \end{equation*}
 We remark that there is some room in the preceding estimate regarding the frequency weights, and the stated upper bounds are not sharp.
 The bound~\eqref{equ:bound_integrated_widehatG} for $k = 2, 3$ can be derived in the same manner, which finishes the proof.
\end{proof}

\subsection{The ODE for the Fourier transform of the profile}

We begin with the proof of Proposition~\ref{prop:ODE_profile}. In order to deduce the asserted differential equation~\eqref{equ:ode_profile} for the Fourier transform of the profile $\hatf(t,\xi)$ of the solution $v(t)$ to~\eqref{equ:nlkg_for_v}, we multiply the differential equation~\eqref{equ:FT_profile_equation_renorm} for $\hatf(t,\xi)$ by the weight $\jxi^\thf$ to obtain that
\begin{equation} \label{equ:ode_profile_derive1}
 \begin{aligned}
  \pt \Bigl( \jxi^\thf \hatf(t,\xi) + \hatr(t,\xi) \Bigr) = \frac{1}{2i} \jxi^{\hf} e^{-it\jxi} \calF\bigl[ \calC_{nl}(v+\bar{v})(t) \bigr](\xi) + \widehat{\calE}(t,\xi).
 \end{aligned}
\end{equation}
Here we use the short-hand notations 
\begin{align}
 \hatr(t,\xi) &:= \jxi^\thf e^{-it\jxi} \calF\bigl[ B(v,v)(t) \bigr](\xi), \notag \\
 \widehat{\calE}(t,\xi) &:= \frac{1}{2i} \jxi^{\hf} e^{-it\jxi} \calF\bigl[ \calQ_{ren}(v,v)(t) \bigr](\xi) \label{equ:def_calE} \\
  &\quad \quad + \frac{1}{2i} \jxi^{\hf} e^{-it\jxi} \calF\bigl[ \calC_l(v+\bar{v})(t) \bigr](\xi) \notag \\
  &\quad \quad + \frac{1}{2i} \jxi^{\hf} e^{-it\jxi} \calF\bigl[ \calR_1(v+\bar{v})(t) \bigr](\xi) \notag \\
  &\quad \quad + \frac{1}{2i} \jxi^{\hf} e^{-it\jxi} \calF\bigl[ \calR_2(v+\bar{v})(t) \bigr](\xi). \notag
\end{align}
The leading order contributions to the right-hand side of~\eqref{equ:ode_profile_derive1} stem from the non-localized cubic nonlinearities $\calC_{nl}(v+\bv)$. All other nonlinearities contribute time-integrable errors. They are correspondingly collected in the term $\widehat{\calE}(t,\xi)$ that satisfies the following decay estimate.

\begin{lemma} \label{lem:ODE_profile_error_est}
 Let $v(t)$ be the solution to~\eqref{equ:nlkg_for_v} on the time interval $[0,T]$ and let $N(T)$ be defined as in~\eqref{equ:def_NT_setting_up}. Assume $N(T) \leq 1$. Then we have for all times $0 \leq t \leq T$ that
 \begin{equation}
  \bigl\| \widehat{\calE}(t,\cdot) \bigr\|_{L^\infty_\xi} \lesssim N(T)^2 \jt^{-\thf+\delta}.
 \end{equation}
\end{lemma}

The next proposition determines the leading order contributions of the non-localized cubic nonlinearities to the right-hand side of~\eqref{equ:ode_profile_derive1}.

\begin{proposition} \label{prop:ODE_profile_leading_order_contribution}
 Assume $T \geq 1$. Let $f(t) := e^{-it\jD} v(t)$ be the profile of the solution $v(t)$ to~\eqref{equ:nlkg_for_v} on the time interval $[0,T]$ and let $N(T)$ be defined as in~\eqref{equ:def_NT_setting_up}. Then uniformly for all $1 \leq t \leq T$ 
 \begin{equation} \label{equ:ODE_profile_leading_order_contribution}
  \begin{aligned}
   \frac{1}{2i} \jxi^{\hf} e^{-it\jxi} \calF\bigl[ \calC_{nl}(v+\bar{v})(t) \bigr](\xi) &= \frac{1}{t} \frac{1}{36 \sqrt{3}} e^{it(-\jxi + 3 \jap{\frac{\xi}{3}})} \jxi^{\hf} \jap{{\textstyle \frac{\xi}{3}}}^{-3} (3+\xi^2) \hat{f}\bigl(t, {\textstyle \frac{\xi}{3}}\bigr)^3 \\
   &\quad \quad + \frac{1}{4i t} \jxi^{-\frac52} (1+3\xi^2) |\hat{f}(t,\xi)|^2\hat{f}(t,\xi) \\
   &\quad \quad + \frac{1}{4i t}  e^{-2 i t \jxi} \jxi^{-\frac52} (1+3\xi^2) |\hat{f}(t, -\xi)|^2 \bar{\hat{f}}(t,-\xi) \\
   &\quad \quad - \frac{1}{t} \frac{1}{36 \sqrt{3}} e^{-it(\jxi + 3 \jap{\frac{\xi}{3}})} \jap{\xi}^{\hf} \jap{{\textstyle \frac{\xi}{3}}}^{-3} (3+\xi^2) \hat{\bar{f}}\bigl(t, {\textstyle \frac{\xi}{3}}\bigr)^3 \\
   &\quad \quad + \calO_{L^\infty_\xi}\bigl( N(T)^3 t^{-\frac65+3\delta} \bigr).
  \end{aligned}
 \end{equation}
\end{proposition}

At this point the proof of Proposition~\ref{prop:ODE_profile} is an immediate consequence of the differential equation~\eqref{equ:ode_profile_derive1}, Lemma~\ref{lem:ODE_profile_error_est}, and Proposition~\ref{prop:ODE_profile_leading_order_contribution} together with the observation that in view of the definition~\eqref{equ:def_variable_coeff_normal_form} of the variable coefficient quadratic normal form, we easily obtain uniformly for all times $0 \leq t \leq T$  
\begin{equation}
 \begin{aligned}
  \|\hatr(t,\cdot)\|_{L^\infty_\xi} \lesssim \sum_{k=1}^3 \bigl\|\jxi^3 \widehat{\alpha}_{1k} \bigr\|_{L^\infty_\xi} |v(t,0)|^2 \lesssim N(T)^2 \jt^{-1}.
 \end{aligned}
\end{equation}

We conclude this subsection with the proof of Lemma~\ref{lem:ODE_profile_error_est}. 
The next subsections are then devoted to the proof of Proposition~\ref{prop:ODE_profile_leading_order_contribution}.

\begin{proof}[Proof of Lemma~\ref{lem:ODE_profile_error_est}]
Throughout we only consider times $0 \leq t \leq T$. We start off with estimating the renormalized quadratic nonlinearities defined in~\eqref{equ:def_renormalized_quad_nonlinearities}, 
\begin{equation*}
 \begin{aligned}
  \bigl\| \jxi^{\hf} \calF\bigl[ \calQ_{ren}(v,v)(t) \bigr](\xi) \bigr\|_{L^\infty_\xi} &\lesssim \sum_{k=1}^3 \bigl\| \jxi^\hf \calF\bigl[ \calQ_{1k}(v,v)(t) \bigr](\xi) \bigr\|_{L^\infty_\xi} + \bigl\| \jxi^\hf \calF\bigl[ \calQ_{14}(v+\bv)(t) \bigr](\xi) \bigr\|_{L^\infty_\xi} \\
  &\quad \quad +  \bigl\| \jxi^\hf \calF\bigl[ \calQ_{2}(v+\bv)(t) \bigr](\xi) \bigr\|_{L^\infty_\xi} + \bigl\| \jxi^\hf \calF\bigl[ \calQ_{3}(v+\bv)(t) \bigr](\xi) \bigr\|_{L^\infty_\xi}.
 \end{aligned}
\end{equation*}
Recalling the definitions~\eqref{equ:def_Q1k_quad_nonlinearities} of $\calQ_{1k}(v,v)(t)$, $1 \leq k \leq 3$, we obtain from the improved decay estimate~\eqref{equ:improved_decay_pt_phase_filtered_v} that
\begin{equation} \label{equ:ODE_profile_error_est_quad1}
 \begin{aligned}
  \sum_{k=1}^3 \bigl\| \jxi^\hf \calF\bigl[ \calQ_{1k}(v,v)(t) \bigr](\xi) \bigr\|_{L^\infty_\xi} &\lesssim \sum_{k=1}^3 \bigl\| \jxi^\thf \widehat{\alpha}_{1k}\bigr\|_{L^\infty_\xi} |\pt ( e^{-it} v(t,0) )| |v(t,0)| \\
  &\lesssim N(T)^2 \jt^{-(\thf-\delta)}.
 \end{aligned}
\end{equation}
Further, using the improved local decay estimates from Lemma~\ref{lem:improved_local_decay}, we conclude that
\begin{equation} \label{equ:ODE_profile_error_est_quad2}
 \begin{aligned}
  \bigl\| \jxi^\hf \calF\bigl[ \calQ_{14}(v+\bv)(t) \bigr](\xi) \bigr\|_{L^\infty_\xi} &\lesssim \bigl\| \jx \alpha_1(x) \bigl( w(t)^2 - w(t,0)^2 \bigr) \bigr\|_{H^1_x} \\
  &\lesssim \| \jD \jx^3 \alpha_1 \|_{L^\infty_x} \bigl\| \jx^{-2} \bigl( w(t)^2 - w(t,0)^2 \bigr) \bigr\|_{L^2_x} \\
  &\quad + \|\jx^2 \alpha_1\|_{L^\infty_x} \| \jx^{-1} \px v(t) \|_{L^2_x} \|v(t)\|_{L^\infty_x} \\
  &\lesssim N(T)^2 \jt^{-(\thf-\delta)}.
 \end{aligned}
\end{equation}
Finally, combining the bound~\eqref{equ:aux_bound_Linfty_wtIpxv} with the improved local decay estimates from Lemma~\ref{lem:improved_local_decay} we find that
\begin{equation} \label{equ:ODE_profile_error_est_quad3}
 \begin{aligned}
  \bigl\| \jxi^\hf \calF\bigl[ \calQ_2(v+\bv)(t) \bigr](\xi) \bigr\|_{L^\infty_\xi} &\lesssim \bigl\| \jx \alpha_2(x) \wtcalI[\px w(t)] w(t) \bigr\|_{H^1_x} \\ 
  &\lesssim \| \jD \jx^2 \alpha_2 \|_{L^\infty_x} \| \jx^{-1} \wtcalI[\px v(t)] \|_{L^2_x} \|v(t)\|_{L^\infty_x} \\
  &\quad + \| \jx^2 \alpha_2 \|_{L^\infty_x} \| \jx^{-1} \px \wtcalI[\px v(t)] \|_{L^2_x} \|v(t)\|_{L^\infty_x} \\
  &\quad + \| \jx^2 \alpha_2 \|_{L^\infty_x} \| \wtcalI[\px v(t)] \|_{L^\infty_x} \|\jx^{-1} \px v(t)\|_{L^2_x} \\
  &\lesssim N(T)^2 \jt^{-(\thf-\delta)},
 \end{aligned}
\end{equation}
and analogously, we infer that
\begin{equation} \label{equ:ODE_profile_error_est_quad4}
 \begin{aligned}
  \bigl\| \jxi^\hf \calF\bigl[ \calQ_3(v+\bv)(t) \bigr](\xi) \bigr\|_{L^\infty_\xi} \lesssim N(T)^2 \jt^{-(\thf-\delta)}.
 \end{aligned}
\end{equation}
Combining \eqref{equ:ODE_profile_error_est_quad1}--\eqref{equ:ODE_profile_error_est_quad4} yields the desired bound on the contributions of all renormalized quadratic nonlinearities
\begin{equation*}
 \bigl\| \jxi^{\hf} \calF\bigl[ \calQ_{ren}(v,v)(t) \bigr](\xi) \bigr\|_{L^\infty_\xi} \lesssim N(T)^2 \jt^{-(\thf-\delta)}.
\end{equation*}

Next, we estimate the localized cubic nonlinearities defined in~\eqref{equ:def_nonlinearities_Cl},
\begin{equation} \label{equ:ODE_profile_error_est_cubic1}
 \begin{aligned}
  \bigl\| \jxi^{\hf} \calF\bigl[ \calC_l(v+\bv)(t) \bigr](\xi) \bigr\|_{L^\infty_\xi} &\lesssim \bigl\| \jx \sech^2(x) \tanh(x) \bigl( \calI[w(t)] \bigr)^3 \bigr\|_{H^1_x} \\
  &\quad + \bigl\| \jx \sech^2(x) \bigl( \calI[w(t)] \bigr)^2 w(t) \bigr\|_{H^1_x}. 
 \end{aligned}
\end{equation}
Using~\eqref{equ:aux_bound_Linfty_Iv}, we can estimate the first term on the right-hand side of~\eqref{equ:ODE_profile_error_est_cubic1} in a simple manner by
\begin{equation} \label{equ:ODE_profile_error_est_cubic2}
 \begin{aligned}
  &\bigl\| \jx \sech^2(x) \tanh(x) \bigl( \calI[w(t)] \bigr)^3 \bigr\|_{H^1_x} \\
  &\lesssim \bigl\| \jx \sech^2(x) \tanh(x) \bigr\|_{H^1_x} \bigl( \|\calI[v(t)]\|_{L^\infty_x} + \|\px \calI[v(t)]\|_{L^\infty_x} \bigr) \|\calI[v(t)]\|_{L^\infty_x}^2 \\
  &\lesssim N(T)^3 \jt^{-\thf}.
 \end{aligned}
\end{equation}
For the second term on the right-hand side of~\eqref{equ:ODE_profile_error_est_cubic1} we use~\eqref{equ:aux_bound_Linfty_Iv} and the local decay estimate~\eqref{equ:improved_local_decay_px_v} to obtain that
\begin{equation} \label{equ:ODE_profile_error_est_cubic3}
 \begin{aligned}
  &\bigl\| \jx \sech^2(x) \bigl( \calI[w(t)] \bigr)^2 w(t) \bigr\|_{H^1_x} \\
  &\lesssim  \bigl\| \jx \sech^2(x) \bigr\|_{H^1_x} \bigl( \|\calI[v(t)]\|_{L^\infty_x} + \|\px \calI[v(t)]\|_{L^\infty_x} \bigr) \|\calI[v(t)]\|_{L^\infty_x} \|v(t)\|_{L^\infty_x} \\
  &\quad + \|\jx^2 \sech^2(x)\|_{L^\infty_x} \|\calI[v(t)]\|_{L^\infty_x}^2 \|\jx^{-1} \px v(t)\|_{L^2_x} \\
  &\lesssim N(T)^3 \jt^{-\thf}.
 \end{aligned}
\end{equation}
Combining~\eqref{equ:ODE_profile_error_est_cubic2} and \eqref{equ:ODE_profile_error_est_cubic3} yields the desired bound on the contributions of all localized cubic nonlinearities
\begin{equation*}
 \bigl\| \jxi^{\hf} \calF\bigl[ \calC_l(v+\bv)(t) \bigr](\xi) \bigr\|_{L^\infty_\xi} \lesssim N(T)^3 \jt^{-\thf}.
\end{equation*}

Since all quartic nonlinearities~\eqref{equ:def_nonlinearities_quartic} are spatially localized, we can proceed analogously to the treatment of the localized cubic nonlinearities, to find that
\begin{equation*}
 \bigl\| \jxi^{\hf} \calF\bigl[ \calR_1(v+\bv)(t) \bigr](\xi) \bigr\|_{L^\infty_\xi} \lesssim N(T)^4 \jt^{-2}.
\end{equation*}

Finally, an inspection of the quintic nonlinearities~\eqref{equ:def_nonlinearities_quintic} shows that they can all be bounded using variants of the following crude schematic estimate
\begin{equation*}
 \begin{aligned}
  \bigl\| \jxi^\hf \calF\bigl[ \bigl( \calI[w(t)] \bigr)^5 \bigr](\xi) \bigr\|_{L^\infty_\xi} &\lesssim \bigl\| \jD \bigl( \bigl( \calI[w(t)] \bigr)^5 \bigr) \bigr\|_{L^1_x} \\
  &\lesssim \|\calI[v(t)]\|_{H^1_x} \|\calI[v(t)]\|_{L^2_x} \|\calI[v(t)]\|_{L^\infty_x}^3 \\
  &\lesssim N(T)^5 \jt^{-(\thf-2\delta)},
 \end{aligned}
\end{equation*}
whence 
\begin{equation*}
 \bigl\| \jxi^\hf \calF\bigl[ \calR_2(v+\bv) \bigr](\xi) \bigr\|_{L^\infty_\xi} \lesssim N(T)^5 \jt^{-(\thf-2\delta)}.
\end{equation*}
This concludes the proof of the lemma.
\end{proof}

\subsection{Fourier analysis of the nonlinearities} \label{subsec:fourier_analysis_nonlin}

As a preparation for the proof of Proposition~\ref{prop:ODE_profile_leading_order_contribution}, in this subsection we determine the Fourier transform of the non-localized cubic nonlinearities that appears on the right-hand side of the differential equation~\eqref{equ:ode_profile_derive1}. In the next subsection we then compute its leading order contributions.

We begin by recalling the well-known fact that
\begin{equation} \label{eq:FT sech}
\wh{\sech}(\xi) = \sfa \sech\Bigl( \frac{\pi}{2} \xi \Bigr).
\end{equation}
In the next lemma we determine the Fourier transform of the integral operator $\calI[g]$ defined in~\eqref{equ:def_calI}.

\begin{lemma} \label{lem:Iop}
The operator
\[
\calI[g](x) = - \sech(x) \int_0^x \cosh(y) g(y) \, \ud y
\]
maps $\calS(\bbR) \to \calS(\bbR)$ and we have
\begin{equation} \label{eq:IFT}
 \begin{aligned}
  \wh{\calI[g]}(\xi) &= -\frac{i}{2} \sech\Bigl( \frac{\pi}{2} \xi \Bigr) \int_\R \frac{\eta}{\jap{\eta}^2} \hat{g}(\eta)\, \ud \eta + \frac{i}{2} \PV \int_\R \cosech\Bigl( \pih (\xi-\eta)\Bigr) \frac{\hat{g}(\eta)}{\jap{\eta}^2} \, \ud \eta.
 \end{aligned}
\end{equation}
\end{lemma}
\begin{proof}
The fact that the linear operator $\Iop \colon \calS(\bbR) \to \calS(\bbR)$ is elementary and is left to the reader. 
For the Fourier transform we compute 
\EQ{\nn 
\wh{\Iop[g]}(\xi) &= \ftn \int_\R \Iop[g](x) e^{-ix\xi} \, \ud x \\
&= -\frac{1}{2\pi}\lim_{\tau\to1+} \int_\R \sech(\tau x) \int_0^x \cosh(y) \int_{\R} \hat{g}(\eta) e^{iy\eta} \, \ud \eta \,  e^{-ix\xi} \, \ud y \, \ud x \\
&= - \frac{1}{4\pi}\lim_{\tau\to1+}\int_{\R}  \int_\R \sech(\tau x) \int_0^x \big(e^{y(1+i\eta)} + e^{y(-1+i\eta)}\big) \, \ud y \, e^{-ix\xi}  \, \ud x \, \hat{g}(\eta) \, \ud \eta \\
& = -  \frac{1}{4\pi}\lim_{\tau\to1+}\int_{\R}  \int_\R \sech(\tau x) \biggl( \frac{e^{x(1+i\eta)}-1}{1+i\eta} + \frac{ e^{x(-1+i\eta)}-1}{-1+i\eta} \biggr) e^{-ix\xi} \, \ud x \, \hat{g}(\eta) \, \ud \eta  \\
&=: \int_\R K(\xi,\eta)\, \hat{g}(\eta) \, \ud \eta
}
with
\EQ{\nn 
 K(\xi,\eta) &= \frac{1}{4\pi}\lim_{\tau\to1+} \int_{\R} \sech(\tau x) \biggl( \frac{1-e^{x(1+i\eta)}}{1+i\eta} - \frac{1- e^{-x(1-i\eta)}}{1-i\eta} \biggr)  e^{-ix\xi} \, \ud x \\
&= \frac{i}{2\pi} \lim_{\tau\to1+} \Im \int_{\R} \sech(\tau x) \frac{1-e^{x(1+i\eta)}}{1+i\eta}  e^{-ix\xi} \, \ud x.
}
To pass to the second line we substituted $x \mapsto -x$ in the second term inside the parentheses on the first line. Hence,
\EQ{\nn 
 K(\xi,\eta) &=  \frac{i}{2\pi}\lim_{\tau\to1+} \Im \int_{\R} \frac{ \sech(\tau x) }{1+i\eta}  \big(  e^{-ix\xi} -e^{-ix(i+(\xi-\eta))}\big) \,  \ud x \\
 & = \frac{i}{2} \Im \biggl( \frac{\sech(\pih \xi)}{1+i\eta} - \lim_{\tau\to1+}  \frac{ \sech\bigl( \tau^{-1}\pih (\xi-\eta +i
 ) \bigr)  }{\tau(1+i\eta)} \biggr) \\
 & = \frac{i}{2} \Im \biggr( \frac{\sech(\pih \xi)}{1+i\eta} + i \PV \frac{ \cosech\bigl( \pih (\xi-\eta
 )\bigr)  }{1+i\eta} \biggr) \\
  & = \frac{i}{2(1+\eta^2)}  \biggl( - \eta\sech\Bigl( \frac{\pi}{2} \xi \Bigr) + \PV \cosech\Bigl( \pih (\xi-\eta)\Bigr) \biggr),
 }
 which finishes the proof.
\end{proof}
 
 Next, we compute the Fourier transform of $\tanh(x)$. 
 
 \begin{lemma} \label{lem:tanh}
 In the sense of tempered distributions, 
 \EQ{
 \label{eq:tanhFT} 
 \wh{\tanh}(\xi) = -i\sfa \PV\cosech\Bigl( \frac{\pi}{2} \xi \Bigr).
 }
 \end{lemma}
 \begin{proof}
 We use Abel summation. In fact, since $\lim_{\eps\to0+} e^{-\eps|x|} \tanh(x) = \tanh(x)$ in the sense of $\calS'(\bbR)$, and since the Fourier transform is continuous on $\calS'(\bbR)$, we compute the usual Fourier transform of $ e^{-\eps|x|} \tanh(x)$ and then pass to the limit in $\calS'(\bbR)$. Fix $\eps>0$ and compute
\begin{align}
\int_\R e^{ix\xi} e^{-\eps|x|} \tanh(x) \, \ud x &= \int_0^\infty e^{ix(\xi+i\eps)} \Big(-1 + \frac{2}{1+e^{-2x}}\Big) \, \ud x \nn \\
&\qquad - \int_{-\infty}^0  e^{ix(\xi-i\eps)} \Big(-1 + \frac{2}{1+e^{2x}}\Big) \,\ud x  \nn \\
&= \frac{1}{i(\xi+i\eps)} + 2\sum_{n=0}^\infty (-1)^n \int_0^\infty e^{ix(\xi+i\eps)} e^{-2nx}\, \ud x \label{eq:Abel} \\
&\qquad +  \frac{1}{i(\xi-i\eps)} -  2\sum_{m=0}^\infty (-1)^m \int_{-\infty}^0 e^{ix(\xi-i\eps)} e^{2mx}\, \ud x \nn \\
& = -\frac{2i\xi}{\xi^2+\eps^2} + 2 \sum_{n=0}^\infty (-1)^n  (2n+\eps-i\xi)^{-1} \nn \\
&\qquad -  2\sum_{m=0}^\infty (-1)^m (2m+\eps+i\xi)^{-1} \nn \\
&= \frac{2i\xi}{\xi^2+\eps^2} + 4i\xi \sum_{n=1}^\infty (-1)^n \bigl( (2n+\eps)^2+\xi^2 \bigr)^{-1}. \nn 
\end{align}
The interchange between integration and summation is justified here by writing, for the first integral over $(0,\infty)$, 
\[
\frac{1}{1+e^{-2x}} = \sum_{n=0}^N (-1)^n e^{-2nx} + \frac{(-1)^{N+1} e^{-2(N+1)x}}{1+e^{-2x}},
\]
which yields
\EQ{\nn
\int_0^\infty e^{ix(\xi+i\eps)} \frac{1}{1+e^{-2x}} \,\ud x & = \sum_{n=0}^N (-1)^n \int_0^\infty e^{ix(\xi+i\eps)} e^{ix(\xi+i\eps)}  e^{-2nx}\,\ud x   \\
&\quad \quad + \int_0^\infty e^{ix(\xi+i\eps)}  \frac{(-1)^{N+1} e^{-2(N+1)x}}{1+e^{-2x}}\,\ud x.
}
The error here is bounded by
\EQ{\nn 
\biggl| \int_0^\infty e^{ix(\xi+i\eps)}  \frac{(-1)^{N+1} e^{-2(N+1)x}}{1+e^{-2x}} \, \ud x  \biggr| &\leq \int_0^\infty e^{-x(\eps+2(N+1))} \, \ud x \leq (2N+2)^{-1},
}
which allows us to pass to the limit $N\to\infty$. In particular, this proved convergence of the infinite series.  We leave the remaining details in justifying  the interchange of limits in~\eqref{eq:Abel} to the reader. 
Thus, 
\[
\wh{\tanh}(-\xi)= -\wh{\tanh}(\xi) = \frac{i}{\sqrt{2\pi}} \biggl( 2 \PV \frac{1}{\xi} +  4 \xi \sum_{n=1}^\infty \frac{(-1)^n}{4n^2+\xi^2} \biggr).
\]
On the other hand,  as meromorphic functions, 
\EQ{\nn 
\frac{\pi}{\sinh(\pih z)} & = 2 \sum_{\ell\in\bZ} \frac{(-1)^\ell}{z-2i\ell} 
 = \frac{2}{ z} + {4z} \sum_{\ell=1}^\infty \frac{(-1)^\ell}{z^2 + 4\ell^2} 
}
and the lemma follows. 
\end{proof}

We deduce a few more identities that will be needed in the sequel.
 
\begin{corollary} \label{cor:sech2}
 We have as equalities in $\calS(\bbR)$,
 \begin{align} 
  \wh{\sech^2}(\xi) &= \sfa\frac{\xi}{\sinh(\pih \xi)}, \label{equ:cor_sech2_identity1} \\
  \biggl( \sech\Bigl( \frac{\pi}{2} \cdot \Bigr) \ast \sech\Bigl( \frac{\pi}{2} \cdot \Bigr) \biggr)(\xi) & = \frac{2\xi}{ \sinh(\pih \xi) }, \label{equ:cor_sech2_identity2} \\
  \biggl( \sech\Bigl( \frac{\pi}{2} \cdot \Bigr) \ast \PV\cosech\Bigl( \frac{\pi}{2} \cdot \Bigr) \biggr)(\xi) &= 2 \xi \sech\Bigl( \frac{\pi}{2} \xi \Bigr), \label{equ:cor_sech2_identity3}
 \end{align}
 and as equalities in $\calS'(\bbR)$, 
 \begin{align}
 \wh{\tanh^2}(\xi) &= \wzwpi \delta_0(\xi) - \sfa \frac{\xi}{\sinh(\pih \xi)}, \label{equ:cor_sech2_identity4} \\
 \biggl( \PV \cosech\Bigl( \frac{\pi}{2} \cdot \Bigr) \ast  \PV\cosech\Bigl( \frac{\pi}{2} \cdot \Bigr) \biggr) (\xi) &= -4 \delta_0(\xi) +  \frac{2\xi}{\sinh(\pih \xi)}. \label{equ:cor_sech2_identity5}
 \end{align}
\end{corollary}
\begin{proof}
 We have $\tanh'(x) = \sech^2(x)$, and thus $\wh{\sech^2}(\xi)= \wh{\tanh'}(\xi) = i\xi \wh{\tanh}(\xi)$. By the previous Lemma~\ref{lem:tanh} this gives the first identity~\eqref{equ:cor_sech2_identity1}. 
 
 To prove the second identity~\eqref{equ:cor_sech2_identity2}, we observe that \eqref{eq:FT sech} implies 
 \begin{equation*}
  \FT \Bigl[ \sech\Bigl( \frac{\pi}{2} \cdot \Bigr) \Bigr](\eta) = \sfainv \sech(\eta),
 \end{equation*}
 and therefore 
 \begin{equation*}
  \FT \Bigl[ \sech\Bigl( \frac{\pi}{2} \cdot \Bigr) \ast \sech\Bigl( \frac{\pi}{2} \cdot \Bigr) \Bigr](\eta) = 2\sfainv \sech^2(\eta).
 \end{equation*}
 Now the second identity~\eqref{equ:cor_sech2_identity2} follows from~\eqref{equ:cor_sech2_identity1}.
 
 To deduce the third identity~\eqref{equ:cor_sech2_identity3}, we use~\eqref{eq:FT sech} and Lemma~\ref{lem:tanh} to compute the Fourier transform of the left-hand side of~\eqref{equ:cor_sech2_identity3}
 \begin{align*}
  \FT\Bigl[ \sech\Bigl( \frac{\pi}{2} \cdot \Bigr) \ast \PV \cosech\Bigl( \frac{\pi}{2} \cdot \Bigr) \Bigr](\eta) = -2i \sfainv \sech(\eta) \tanh(\eta) = 2i \sfainv \sech'(\eta),
 \end{align*}
 whence 
\begin{equation*}
 \begin{aligned}
   \biggl( \sech\Bigl( \frac{\pi}{2} \cdot \Bigr) \ast \PV\cosech\Bigl( \frac{\pi}{2} \cdot \Bigr) \biggr)(\xi)  &= 2i\sfainv\FT^{-1}\bigl[ \sech' \bigr](\xi) = 2\sfainv\xi \wh{\sech}(\xi) =  2\xi \sech\Bigl( \frac{\pi}{2} \xi \Bigr).
 \end{aligned}
\end{equation*}
 
 Next, the identity $\tanh^2(x) = 1 -\sech^2(x)$ together with~\eqref{equ:cor_sech2_identity1} imply the fourth identity~\eqref{equ:cor_sech2_identity4} in the sense of $\calS'(\bbR)$. Finally, by Lemma~\ref{lem:tanh},
 \[
  \FT \Bigl[ \PV \cosech\Bigl( \frac{\pi}{2} \cdot \Bigr) \Bigr](\eta) = -i\sfainv\tanh(\eta),
 \]
 whence 
 \[
  \FT\Bigl[ \PV \cosech\Bigl( \frac{\pi}{2} \cdot \Bigr) \ast \PV \cosech\Bigl( \frac{\pi}{2} \cdot \Bigr) \Bigr](\eta) = -2\sfainv \tanh^2(\eta),
 \]
 which leads from the fourth~\eqref{equ:cor_sech2_identity4} to the fifth identity~\eqref{equ:cor_sech2_identity5}.
 Here the convolution in $\calS'(\bbR)$ on the left-hand side is well-defined since 
 \begin{equation*}
  \omega:=\PV\cosech\Bigl( \frac{\pi}{2} \cdot \Bigr)
 \end{equation*}
 satisfies $\omega \ast g \in \calS(\bbR)$ if $g \in \calS(\bbR)$. Therefore, in the sense of the duality pairing between $\calS'(\bbR)$ and~$\calS(\bbR)$,
 \[
 \langle \omega\ast \omega,g\rangle = \langle \omega, \omega(-\cdot)\ast g\rangle 
 \]
 and hence 
 \EQ{\nn 
  \langle \wh{\omega\ast \omega},g\rangle &=  \langle \omega\ast \omega, \hat{g} \rangle = \langle \omega, \omega(-\cdot)\ast \hat{g}\rangle  = \langle\hat{ \omega}, [\omega(-\cdot)\ast \hat{g}]^{\vee}\rangle = \wzwpi \langle\hat{ \omega}, \hat{\omega} g\rangle, 
 }
 as desired. 
\end{proof}

Next, we take on the computation of the Fourier transform of the non-localized cubic nonlinearities~\eqref{equ:def_nonlinearities_Cnl} that involve $\Iop$. 
To this end we introduce a few more short-hand notations.
The following definitions of $A,B$ are taken directly from~\eqref{eq:IFT} in the statement of Lemma~\ref{lem:Iop}.
 
\begin{definition}
 We define
 \EQ{\nn
 \omega_1(\xi) &:= \frac{\xi}{2\sinh(\pih \xi)} \in \calS(\R), \\
 \omega_2(\xi) &:= 2 \xi\sech\Bigl(\pih \xi\Bigr) \in \calS(\R),  \\
 \Omega &:= \frac{i}{2} \PV \cosech\Bigl(\pih \cdot \Bigr) \in \calS'(\R),
 }
 and 
 \begin{align*}
 A(g)(\xi) &:= \frac{i}{2} \PV \int_\R \cosech \Bigl( \pih (\xi-\eta) \Bigr) \jap{\eta}^{-2} g(\eta) \, \ud \eta = \bigl( \Omega \ast \langle \cdot \rangle^{-2} g \bigr)(\xi), \\
 B(h) &:= -\frac{i}{2} \int_\R \frac{\eta}{\jap{\eta}^2} h(\eta)\ud \eta.
 \end{align*}
 Moreover, we set 
 \begin{align*}
  \Sech(\xi) &:= \sech \Bigl(\pih \xi\Bigr), \\
  \Cosech(\xi) &:= \cosech \Bigl( \pih \xi\Bigr).
 \end{align*}
\end{definition}
 
 With the aid of these short-hand notations we may write~\eqref{eq:IFT} from Lemma~\ref{lem:Iop} succinctly as
 \begin{equation} \label{equ:FT_calI_short}
  \wh{\Iop[w]}(\xi)= \Sech(\xi) B(\hat{w}) + A(\hat{w})(\xi).
 \end{equation}
 Now recall from~\eqref{equ:def_nonlinearities_Cnl} that the non-localized cubic nonlinearities are given by
 \begin{equation} \label{equ:FT_section_recalling_Cnl}
  \calC_{nl}(w) = \frac12 \bigl( \calI[w] \bigr)^2 w + \frac13 \tanh(x) \bigl( \calI[w] \bigr)^3.
 \end{equation}
 The next lemma determines the Fourier transform of the first cubic nonlinearity on the right-hand side of~\eqref{equ:FT_section_recalling_Cnl}.

\begin{lemma} \label{lem:3I}
 For any $w \in \calS(\R)$, 
 \begin{equation} \label{equ:FT_calCnl1}
  \FT\bigl[ \Iop[w]^2 w \bigr] = \frac{1}{2\pi} \big( T_1(w) + T_2(w) + T_3(w)\big), 
 \end{equation}
 where 
 \begin{align*}
  T_1(w) &= 4B(\hat{w})^2\; \omega_1 \ast \hat{w}, \\
  T_2(w) &= iB(\hat{w})\, \omega_2 \ast \wh{\jD^{-2}w}\ast \hat{w}, \\
  T_3(w) &= (\delta_0 - \omega_1 )\ast \wh{\jD^{-2}w}\ast  \wh{\jD^{-2}w}\ast \hat{w}.
 \end{align*}
\end{lemma}
\begin{proof}
 By \eqref{equ:FT_calI_short} we have 
 \EQ{\nn
  2\pi \FT\bigl[ \Iop[w]^2 w \bigr] &= \wh{\Iop[w]}\ast\wh{\Iop[w]}\ast\hat{w} \\
  &= \bigl( \Sech B(\hat{w}) + A(\hat{w}) \bigr) \ast  \bigl( \Sech B(\hat{w}) + A(\hat{w}) \bigr) \ast\hat{w} \\
  &= B(\hat{w})^2  \Sech\ast \Sech\ast\hat{w} + 2 B(\hat{w}) \Sech\ast A(\hat{w}) \ast \hat{w} +  A(\hat{w})\ast A(\hat{w})\ast \hat{w} \\
  &=: T_1(w) + T_2(w) + T_3(w).
 }
 Then \eqref{equ:cor_sech2_identity2} implies 
 \begin{equation*}
  T_1(w) = 4B(\hat{w})^2\; \omega_1 \ast \hat{w}.
 \end{equation*}
 Similarly, we obtain from~\eqref{equ:cor_sech2_identity3} that 
 \begin{equation*}
  T_2(w) = i B(\hat{w}) \Sech\ast\PV\Cosech  \ast   \wh{\jD^{-2}w}\ast \hat{w} = i B(\hat{w}) \, \omega_2  \ast   \wh{\jD^{-2}w}\ast \hat{w}.
 \end{equation*}
 Finally, \eqref{equ:cor_sech2_identity5} gives
 \begin{align*}
  T_3(w) &= \Omega\ast\Omega\ast   \wh{\jD^{-2}w}\ast  \wh{\jD^{-2}w}\ast \hat{w} \\
  & = -\frac14 \PV\Cosech \ast\PV\Cosech  \ast   \wh{\jD^{-2}w}\ast  \wh{\jD^{-2}w}\ast \hat{w} \\
  & =  (\delta_0 - \omega_1 )\ast   \wh{\jD^{-2}w}\ast  \wh{\jD^{-2}w}\ast \hat{w},
 \end{align*}
 as claimed. 
\end{proof}

Next we compute the Fourier transform of the second cubic nonlinearity on the right-hand side of~\eqref{equ:FT_section_recalling_Cnl}.
 
\begin{lemma} \label{lem:tanhI3}
 For any $w\in \calS(\R)$, 
 \begin{equation} \label{equ:FT_calCnl2}
  \FT \bigl[ \tanh(\cdot) (\Iop[w])^3 \bigr]  = \frac{1}{4 \pi}\sum_{j=1}^4 S_j(w),  
 \end{equation}
 where 
 \begin{align*}
  S_1(w) &= -4i B(\hat{w})^3\,\omega_1\ast\omega_2, \\
  S_2(w) &= \frac32  B(\hat{w})^2  \omega_2\ast\omega_2 \ast\wh{\jD^{-2}w}, \\
  S_3(w) &= 3 i B(\hat{w}) (-\omega_2+ \omega_1\ast\omega_2)\ast  \wh{\jD^{-2}w}\ast  \wh{\jD^{-2}w}, \\
  S_4(w) &= -2 (\delta_0-2 \omega_1+\omega_1\ast\omega_1)\ast  \wh{\jD^{-2}w}\ast  \wh{\jD^{-2}w}  \ast\wh{\jD^{-2}w}.
 \end{align*}
\end{lemma}
\begin{proof}
 By Lemma~\ref{lem:tanh} and \eqref{equ:FT_calI_short} we have 
 \EQ{\nn 
 4\pi \FT\bigl[ \tanh (\Iop[w])^3 \bigr]  &= -i \PV\Cosech \ast  \bigl( \Sech B(\hat{w}) + A(\hat{w}) \bigr) \\
 &\quad \quad \quad \ast \bigl( \Sech B(\hat{w}) + A(\hat{w}) \bigr) \ast \bigl( \Sech B(\hat{w}) + A(\hat{w}) \bigr) \\
 &= -i B(\hat{w})^3 \;  \PV\Cosech \ast \Sech\ast\Sech\ast\Sech \\
 &\quad - 3i B(\hat{w})^2\;  \PV\Cosech\ast \Sech\ast\Sech \ast  A(\hat{w}) \\
 &\quad - 3i B(\hat{w})\;  \PV\Cosech\ast \Sech \ast A(\hat{w}) \ast  A(\hat{w})  \\
 &\quad -i \PV\Cosech  \ast A(\hat{w}) \ast  A(\hat{w})   \ast  A(\hat{w}) \\
 &=: S_1(w) + S_2(w) + S_3(w) + S_4(w).
 }
 Then~\eqref{equ:cor_sech2_identity2} and \eqref{equ:cor_sech2_identity3} imply
 \begin{equation*}
  S_1(w) = -i B(\hat{w})^3 \PV\Cosech \ast \Sech\ast\Sech\ast\Sech = -4i B(\hat{w})^3 \omega_1\ast\omega_2.
 \end{equation*}
 Similarly, we infer from~\eqref{equ:cor_sech2_identity3} that 
 \EQ{ \nn 
 S_2(w) &= \frac{3}{2}   B(\hat{w})^2 \, \PV\Cosech \ast \Sech \ast  \PV\Cosech\ast\Sech\ast \wh{\jD^{-2}w} \\ 
 &=  \frac{3}{2} B(\hat{w})^2\,  \omega_2\ast\omega_2 \ast \wh{\jD^{-2}w},
 }
 and from~\eqref{equ:cor_sech2_identity3} together with~\eqref{equ:cor_sech2_identity5} that
 \EQ{ \nn 
 S_3(w) &= -3i B(\hat{w}) \PV \Cosech \ast \Omega \ast  \Omega \ast \Sech \ast \wh{\jD^{-2}w} \ast \wh{\jD^{-2}w} \\ 
 &= \frac{3}{4} i B(\hat{w})  \omega_2\ast (-4\delta_0+4\omega_1)  \ast \wh{\jD^{-2}w} \ast \wh{\jD^{-2}w}\\
 &= 3i B(\hat{w}) (-\omega_2+\omega_1\ast\omega_2)  \ast \wh{\jD^{-2}w} \ast \wh{\jD^{-2}w}.
 }
 Finally, \eqref{equ:cor_sech2_identity5} yields
 \EQ{\nn
 S_4(w) &= -i \Bigl( \frac{i}{2} \Bigr)^3 \PV\Cosech \ast \PV\Cosech \ast \PV\Cosech\ast \PV\Cosech \ast \wh{\jD^{-2}w} \ast \wh{\jD^{-2}w} \ast \wh{\jD^{-2}w}\\
 &= -2 (\delta_0-2\omega_1+\omega_1\ast\omega_1) \ast \wh{\jD^{-2}w} \ast \wh{\jD^{-2}w} \ast \wh{\jD^{-2}w},
 }
 which concludes the proof.
\end{proof}

\subsection{Proof of Proposition~\ref{prop:ODE_profile_leading_order_contribution}} \label{subsec:stat_phase}

In this subsection we finally determine via a stationary phase analysis the leading order behavior of the evolution of the Fourier transform of the non-localized cubic nonlinearities asserted in~\eqref{equ:ODE_profile_leading_order_contribution}.

We obtained the precise expression for the Fourier transform of the non-localized cubic nonlinearities $\calC_{nl}(v+\bv)$ in~\eqref{equ:FT_calCnl1} and~\eqref{equ:FT_calCnl2}. In the next lemma we conclude that all terms that involve the convolution with Schwartz functions only contribute time-integrable errors. The dominant contributions to the Fourier transform of the non-localized cubic nonlinearities are therefore obtained from the $\delta_0$-convolutions in $T_3$, respectively~$S_4$.

\begin{lemma} \label{lem:ODE_profile_nonlocal_cubic_contribution_peeled_off}
Assume $T \geq 1$. Let $v(t)$ be the solution to~\eqref{equ:nlkg_for_v} on the time interval $[0,T]$. Then we have uniformly for all $\xi \in \bbR$ and for all $1 \leq t \leq T$ that 
\begin{equation} \label{equ:ODE_profile_nonlocal_cubic_contribution_peeled_off}
 \begin{aligned}
  \jxi^\hf e^{-it\jxi} \calF\bigl[ \calC_{nl}(v+\bv)(t) \bigr](\xi) &= \frac{1}{4\pi} \jxi^\hf e^{-it\jxi} \Bigl( \wh{\jD^{-2}w}(t)\ast  \wh{\jD^{-2}w}(t)\ast \hatw(t) \Bigr)(\xi) \\
  &\quad - \frac{1}{6 \pi} \jxi^\hf e^{-it\jxi} \Bigl( \wh{\jD^{-2}w}(t) \ast \wh{\jD^{-2}w}(t) \ast \wh{\jD^{-2}w}(t) \Bigr)(\xi) \\
  &\quad + \calO_{L^\infty_\xi}\bigl( N(T)^3 t^{-\frac54+\frac{\delta}{2}} \bigr),
 \end{aligned}
\end{equation}
where $w(t) = v(t) + \bv(t)$.
\end{lemma}
\begin{proof}
Throughout we only consider times $1 \leq t \leq T$.
We begin by deriving a decay estimate for $B(\hatw(t))$.
Inserting $w(t) = e^{it\jD} f(t)+ e^{-it\jD} \bar{f}(t)$ and integrating by parts, we find
\begin{equation*}
 \begin{aligned}
  B(\hatw(t)) &= -\frac{i}{2} \biggl( \int_\bbR \frac{\eta}{\jap{\eta}^2} e^{it\jap{\eta}} \hatf(t,\eta) \, \ud \eta + \int_\bbR \frac{\eta}{\jap{\eta}^2} e^{-it\jap{\eta}} \overline{\hatf(t,-\eta)} \, \ud \eta \biggr) \\
  &= \Im \int_\bbR \frac{\eta}{\jap{\eta}^2} e^{it\jap{\eta}} \hatf(t,\eta) \, \ud \eta \\
  &= \frac{1}{t} \Re \int_\bbR e^{it\jap{\eta}} \partial_\eta \bigl( \jap{\eta}^{-1} \hatf(t,\eta) \bigr) \, \ud \eta,
 \end{aligned}
\end{equation*}
whence by~\eqref{equ:relation_L_partial_xi},
\begin{equation*}
 \begin{aligned}
  |B(\hatw(t))| &\leq \frac{1}{t} \int_\bbR \bigl( \jap{\eta}^{-2} |\hatf(t,\eta)| + \jap{\eta}^{-1} |\partial_\eta \hatf(t,\eta)| \bigr) \, \ud \eta \\
  &\lesssim t^{-1} \|\hatf(t)\|_{H^1_\eta} \lesssim N(T) t^{-1+\delta}.
 \end{aligned}
\end{equation*}

Then for $T_1$ in Lemma~\ref{lem:3I} we bound 
\begin{align*}
 \bigl\| \jxi^\hf ( \omega_1 \ast \hatw(t) ) \bigr\|_{L^\infty_\xi} &\lesssim \| \check{\omega}_1 w(t) \|_{L^1_x} + \| \px (\check{\omega}_1 w(t) ) \|_{L^1_x} \\
 &\lesssim \|\check{\omega}_1\|_{L^1_x} \|v(t)\|_{L^\infty_x} + \|\check{\omega}_1\|_{L^2_x} \|\px v(t)\|_{L^2_x} \\
 &\lesssim N(T) \jt^\delta,
\end{align*}
and thus,
\begin{equation*}
 \bigl\| \jxi^\hf T_1(\hatw(t)) \bigr\|_{L^\infty_\xi} \lesssim |B(\hatw(t))|^2 \bigl\| \jxi^\hf ( \omega_1 \ast \hatw(t) ) \bigr\|_{L^\infty_\xi} \lesssim N(T)^3 t^{-2+3\delta}.
\end{equation*}

For $T_2$ in that same lemma we estimate
\EQ{\nn 
 \bigl\|  \omega_2\ast \hatw(t) \ast\wh{\jD^{-2} w}(t) \bigr\|_{L^\infty_\xi} \le  \| \check{\omega}_2\, {w}(t) (\jD^{-2} w)(t)\|_{L^1_x} &\le  \| \check{\omega}_2\|_{L^1_x} \|w(t)\|_{L^\infty_x} \|{\jD^{-2} w}(t)\|_{L^\infty_x} \\
 &\les  \| v(t)\|_{L^\infty_x}^2 \lesssim N(T)^2 t^{-1}
}
and 
\EQ{\nn 
\bigl\| \xi \bigl( \omega_2 \ast \hatw(t) \ast \wh{\jD^{-2} w}(t) \bigr) \bigr\|_{L^\infty_\xi}  &\le  \| \px (\check{\omega}_2 \, {w}(t)(\jD^{-2} w)(t)) \|_{L^1_x} \\
&\le  \| \px \check{\omega}_2\|_{L^2_x} \| {w}(t)\|_{L^\infty_x} \|\jD^{-2} w(t)\|_{L^2_x} \\
&\quad + \| \check{\omega}_2\|_{L^2_x} \|\px {w}(t)\|_{L^2_x} \|\jD^{-2} w(t)\|_{L^\infty_x} \\
&\quad +\| \check{\omega}_2\|_{L^2_x} \| {w}(t)\|_{L^\infty_x} \|\jD^{-1} w(t)\|_{L^2_x} \\
&\les N(T)^2 t^{-\frac12+\delta}. 
}
The conclusion is that
\[
 \bigl\| \jxi^{\frac12} T_2(w(t)) \|_{L^\infty_\xi} \lesssim |B(\hatw(t))| \bigl\| \jxi \bigl( \omega_2 \ast \hatw(t) \ast \wh{\jD^{-2} w}(t) \bigr) \bigr\|_{L^\infty_\xi} \les  N(T)^3 t^{-\frac32+2\delta}.
\]
Finally, for the part of $T_3$ that involves convolution with the Schwartz function $\omega_1$, we bound, on the one hand, 
\EQ{\nn 
\bigl\| \omega_1 \ast \hatw(t) \ast \wh{\jD^{-2} w}(t) \ast\wh{\jD^{-2} w}(t) \bigr\|_{L^\infty_\xi} & \le  \| \check{\omega}_1\, {w}(t) (\jD^{-2} w(t))^2\|_{L^1_x} \\
&\le \| \check{\omega}_1 \|_{L^1_x} \|w(t)\|_{L^\infty_x} \|\jD^{-2} w(t)\|_{L^\infty_x}^2\\
&\les  \| v(t)\|_{L^\infty_x}^3 \les N(T)^3 t^{-\frac32} 
}
and, on the other hand, 
\EQ{\nn 
\bigl\| \xi \bigl( \omega_1 \ast \hatw(t) \ast \wh{\jD^{-2} w}(t) \ast \wh{\jD^{-2} w}(t) \bigr) \bigl\|_{L^\infty_\xi} &\leq  \| \px ( \check{\omega}_1 \, {w}(t) (\jD^{-2} w(t))^2)\|_{L^1_x} \\
&\le \|\px\check{\omega}_1\|_{L^1_x} \|{w}(t)\|_{L^\infty_x} \|{\jD^{-2} w}(t)\|_{L^\infty_x}^2 \\
&\quad + \| \check{\omega}_1\|_{L^{\frac43}_x} \| \px w(t) \|_{L^4_x} \|{\jD^{-2} w}(t)\|_{L^\infty_x}^2 \\
&\quad + 2\| \check{\omega}_1\|_{L^1_x} \| w(t)\|_{L^\infty_x} \|{\jD^{-2} w}(t)\|_{L^\infty_x} \|{\jD^{-1} w}(t)\|_{L^\infty_x}  \\
&\les N(T)^3 t^{-\frac54+\frac{\delta}{2}}.
}
Here we used the Gagliardo-Nirenberg-Sobolev bound
\[
  \|\px w(t)\|_{L^4_x} \les \|\px^2 w(t)\|_{L^2_x}^{\frac12} \|w(t)\|_{L^\infty_x}^{\frac12}.
\]
In summary,
\begin{equation}
 \bigl\| \jxi \bigl( \omega_1 \ast \hatw(t) \ast \wh{\jD^{-2} w}(t) \ast \wh{\jD^{-2} w}(t) \bigr) \bigl\|_{L^\infty_\xi} \lesssim N(T)^3 t^{-\frac54+\frac{\delta}{2}}. 
\end{equation}
This shows that all terms but the $\delta_0$ in $T_3$ contribute errors that have time-integrable decay at least of the order $t^{-\frac54+\frac{\delta}{2}}$. Arguing in an analogous fashion for the terms $S_j$, $1\le j\le 4$, in Lemma~\ref{lem:tanhI3}, we arrive at the same conclusion. 
\end{proof}

It now remains to determine the leading order contributions of the first two terms on the right-hand side of~\eqref{equ:ODE_profile_nonlocal_cubic_contribution_peeled_off} via a stationary phase analysis. 
We treat the first term in detail, the analysis of the second term being analogous.
To this end we define the $2$-plane 
\begin{equation*}
 \Pi_\xi = \Bigl\{ \bxi := (\xi_1,\xi_2,\xi_3) \in \R^3 \: : \: \sum_{j=1}^3 \xi_j =\xi \Bigr\}, \quad \xi \in \bbR, 
\end{equation*}
and introduce the short-hand notation 
\begin{equation*}
 \hat{h}(t,\xi) := \jap{\xi}^{-2} \hat{f}(t,\xi).
\end{equation*}
Denoting by $\calH^2$ the two-dimensional Hausdorff measure, we write the first term on the right-hand side of~\eqref{equ:ODE_profile_nonlocal_cubic_contribution_peeled_off} in the form 
\begin{equation*}
 \frac{1}{4\pi} \jxi^\hf e^{-it\jxi} \Bigl( \wh{\jD^{-2}w}(t)\ast  \wh{\jD^{-2}w}(t)\ast \whatw(t) \Bigr)(\xi) = \frac{1}{4 \pi} \sum_{j=1}^6 \calT_j(t,\xi),
\end{equation*}
where
\EQ{\nn 
\calT_1(t,\xi) & = \jxi^{\frac12} \int_{\Pi_\xi} e^{it\phi_1(\bxi)} \; \hat{h}(t,\xi_1)\hat{h}(t,\xi_2)\hat{f}(t,\xi_3) \, \ud \calH^2(d\bxi), \\
\calT_2(t,\xi) & = 2\jxi^{\frac12} \int_{\Pi_\xi} e^{it\phi_2(\bxi)}\; \hat{h}(t,\xi_1)\hat{\bar{h}}(t,\xi_2)\hat{f}(t,\xi_3) \, \ud \calH^2(d\bxi),
}
respectively, 
\EQ{\nn 
\calT_3(t,\xi) & =   \jxi^{\frac12} \int_{\Pi_\xi} e^{it\phi_2(\bxi)} \;\hat{h}(t,\xi_1)\hat{\bar{f}}(t,\xi_2)\hat{h}(t,\xi_3) \, \ud \calH^2(d\bxi), \\
\calT_4(t,\xi) & =  2\jxi^{\frac12} \int_{\Pi_\xi} e^{it\phi_3(\bxi)}\; \hat{h}(t,\xi_1)\hat{\bar{h}}(t,\xi_2)\hat{\bar{f}}(t,\xi_3) \, \ud \calH^2(d\bxi),  \\
} 
and, finally, 
\EQ{\nn 
\calT_5(t,\xi) & =  \jxi^{\frac12} \int_{\Pi_\xi} e^{it\phi_3(\bxi)}\; \hat{{f}}(t,\xi_1)\hat{\bar{h}}(t,\xi_2)\hat{\bar{h}}(t,\xi_3) \, \ud \calH^2(d\bxi),  \\
\calT_6(t,\xi) & = \jxi^{\frac12} \int_{\Pi_\xi} e^{it\phi_4(\bxi)}\; \hat{\bar{h}}(t,\xi_1)\hat{\bar{h}}(t,\xi_2)\hat{\bar{f}}(t,\xi_3) \, \ud \calH^2(d\bxi). 
}
The phases are given by
\EQ{\label{eq:phases}
\phi_1(\bxi) &= -\jxi + \jap{\xi_1} + \jap{\xi_2} + \jap{\xi_3}, \\
\phi_2(\bxi) &= -\jxi + \jap{\xi_1} - \jap{\xi_2} + \jap{\xi_3}, \\
\phi_3(\bxi) &= -\jxi + \jap{\xi_1} - \jap{\xi_2} - \jap{\xi_3}, \\
\phi_4(\bxi) &= -\jxi - \jap{\xi_1} - \jap{\xi_2} - \jap{\xi_3}.
}
The critical points of the phases are characterized by $d\Psi_j(\bxi_*) =0$, where $\Psi_j$ is the pullback of $\phi_j$ onto the plane $\Pi_\xi$ with global coordinates $(\xi_1,\xi_2)$.  The unique solutions are given by 
\EQ{\label{eq:critpts}
d\Psi_1\Bigl(\frac{\xi}{3}, \frac{\xi}{3}\Bigr) &= 0, \\
d\Psi_2(\xi,-\xi) &= 0, \\ 
d\Psi_3(-\xi,\xi) &= 0, \\
d\Psi_4\Bigl(\frac{\xi}{3}, \frac{\xi}{3}\Bigr) &= 0,
}
with respective values 
\EQ{\label{eq:critvals}
 \phi_1\Bigl(\frac{\xi}{3}, \frac{\xi}{3}, \frac{\xi}{3} \Bigr) &= -\jxi + 3\jap{ {\textstyle \frac{\xi}{3} } } \simeq \jxi^{-1}, \\
 \phi_2(\xi,-\xi,\xi) &= 0, \\ 
 \phi_3(-\xi,\xi,\xi) &= -2\jxi, \\
 \phi_4\Bigl(\frac{\xi}{3}, \frac{\xi}{3}, \frac{\xi}{3} \Bigr) &= -\jxi - 3\jap{ {\textstyle \frac{\xi}{3} } },
}
and Hessians
\EQ{\label{eq:Hess}
 \mathrm{Hess} \, \Psi_1\Bigl(\frac{\xi}{3}, \frac{\xi}{3}\Bigr) &= \jap{ {\textstyle \frac{\xi}{3} } }^{-3} \left[ \begin{matrix} 
                                  2 & 1 \\
                                  1 & 2
                                 \end{matrix} \right], \\       
 \mathrm{Hess} \, \Psi_2(\xi,-\xi) &= \jap{\xi}^{-3} \left[ \begin{matrix} 
                                  2 & 1 \\
                                  1 & 0
                                 \end{matrix} \right], \\
 \mathrm{Hess} \, \Psi_3(-\xi,\xi) &= -\jap{\xi}^{-3} \left[ \begin{matrix} 
                                  0  & 1 \\
                                  1 & 2
                                 \end{matrix} \right], \\
 \mathrm{Hess} \, \Psi_4\Bigl(\frac{\xi}{3}, \frac{\xi}{3}\Bigr) &= - \jap{ {\textstyle \frac{\xi}{3} } }^{-3} \left[ \begin{matrix} 
                                  2 & 1 \\
                                  1 & 2
                                 \end{matrix} \right].                       
}

We now describe in detail how to extract the leading order term from $\calT_2(t,\xi)$ via stationary phase. 
We introduce a Littlewood-Paley decomposition and write
\begin{align} \label{eq:T2LP}
\calT_2(t,\xi)  &= 2\sum_{k,\ell\ge0} J_{k\ell}, \\
J_{k\ell} &:= \jxi^{\frac12} \iint e^{it\Psi_2(\xi_1,\xi_2)}\; \hat{h}_k(t,\xi_1)\hat{\bar{h}}_\ell(t,\xi_2)\hat{f}(t,\xi-\xi_1-\xi_2) \, \ud \xi_1 \, \ud\xi_2, \nn \\
\hat{h}_k(t,\xi_1) &:= \psi_k(\xi_1)\hat{h}(t,\xi_1), \nn
\end{align}
where $\psi_k(\xi_1)$ is supported on $\{ |\xi_1| \simeq 2^k\}$ for $k \geq 1$ and on $\{ |\xi_1| \lesssim 1\}$ for $k=0$.
We only consider the case $|\xi|\simeq 2^j$ for $j \gg 1$, the case $|\xi|\les 1$ being easier. Then we decompose $\calT_2(t,\xi)$ into the {\em high-low, low-high, high-high} and {\em critical contributions}, viz. 
\EQ{\label{eq:highlow}
 \calT_2(t,\xi) & = \calJ_{hl} + \calJ_{lh}+ \calJ_{hh} +\calJ_{crit}, \\
 \calJ_{hl} & = \sum_{k\ge j+10} \; \sum_{0\le \ell\le k-5} J_{k\ell}, \\
  \calJ_{lh} & = \sum_{\ell\ge j+10} \; \sum_{0\le k\le \ell-5} J_{k\ell}, \\
  \calJ_{hh} & = \sum_{k\ge j+10} \; \sum_{|\ell- k|< 5} J_{k\ell}, \\
  \calJ_{crit} & = \sum_{0\le k,\ell < j+10} J_{k\ell}.
 }
 The final $\calJ_{crit}$ gives the main contribution via stationary phase. We first show that the first three are error terms. 
 \begin{lemma} \label{lem:hllhhh}
  Assume $T \geq 1$. We have uniformly for all $\xi \in \bbR$ and all times $1 \leq t \leq T$ that
  \[
   |\calJ_{hl}| + |\calJ_{lh}| + |\calJ_{hh}|\les N(T)^3 \, t^{-\frac32 + 2\delta}.
  \]
 \end{lemma}
 
 In the proof of Lemma~\ref{lem:hllhhh} we repeatedly use the following trilinear estimate.
\begin{lemma} \label{lem:pseudoprodop}
 Assume that $m \in L^1(\bbR^2)$ satisfies
 \begin{equation} \label{equ:trilinear1}
  \biggl\| \int_{\bbR^2} m(\xi_1, \xi_2) e^{i x_1 \xi_1} e^{i x_2 \xi_2} \, \ud \xi_1 \, \ud \xi_2 \biggr\|_{L^1_{x_1, x_2}(\bbR^2)} \leq A
 \end{equation}
 for some $A > 0$. 
 Then we have for any exponents $p, q, r, \in [1,\infty]$ with $\frac{1}{p} + \frac{1}{q} + \frac{1}{r} = 1$ that
 \begin{equation} \label{equ:trilinear2}
  \biggl| \int_{\bbR^2} m(\xi_1, \xi_2) \hatf(\xi_1) \hatg(\xi_2) \hath(-\xi_1-\xi_2) \, \ud \xi_1 \, \ud \xi_2 \biggr| \lesssim A \|f\|_{L^p_x} \|g\|_{L^q_x} \|h\|_{L^r_x}.
 \end{equation}
\end{lemma}
\begin{proof}
 By direct computation we find that
 \begin{equation*}
 \begin{aligned}
  &\int_{\bbR^2} m(\xi_1, \xi_2) \hatf(\xi_1) \hatg(\xi_2) \hath(-\xi_1-\xi_2) \, \ud \xi_1 \, \ud \xi_2 \\
  &= \frac{1}{(2\pi)^\thf} \int_{\bbR^3} \biggl( \int_{\bbR^2} m(\xi_1, \xi_2) e^{i\xi_1 x_1} e^{i\xi_2 x_2} \, \ud \xi_1 \, \ud \xi_2 \biggr) \\
  &\qquad \qquad \qquad \qquad \qquad \times f(x_3-x_1) g(x_3-x_2) h(x_3) \, \ud x_1 \, \ud x_2 \, \ud x_3.
 \end{aligned}
 \end{equation*}
 Then~\eqref{equ:trilinear2} follows from~\eqref{equ:trilinear1} and H\"older's inequality.
\end{proof}

We are now prepared for the proof of Lemma~\ref{lem:hllhhh}. 
 
\begin{proof}[Proof of Lemma~\ref{lem:hllhhh}]
In the high-low case we integrate by parts in $\xi_2$ and write, with $\xi_3=\xi-\xi_1-\xi_2$, 
\EQ{\nn
J_{k\ell} & = - \frac{\jxi^{\frac12}}{it} \iint e^{it\Psi_2(\xi_1,\xi_2)}\; \partial_2 \biggl( \frac{1}{\partial_2 \Psi_2(\xi_1,\xi_2)} \hat{h}_k(t,\xi_1)\hat{\bar{h}}_\ell(t,\xi_2)\hat{f}(t,\xi_3) \biggr) \, \ud \xi_1 \, \ud\xi_2. 
}
We apply Lemma~\ref{lem:pseudoprodop} to bound  this in $L^\infty_\xi$. The choices of $m(\xi_1,\xi_2)$ in that lemma are, respectively,  
\EQ{\nn 
m_1(\xi_1, \xi_2) &=   \frac{1}{\partial_2 \Psi_2(\xi_1,\xi_2)} \psi_k(\xi_1)\psi_\ell(\xi_2) 
}
or $m_2=\partial_2 m_1$ depending on where $\partial_2$ falls inside the integral. To verify the conditions of Lemma~2.7 we compute 
\[
 \frac{1}{\partial_2 \Psi_2(\xi_1,\xi_2)} = \frac{\jap{\xi_2}\jap{\xi_3}(\xi_2\jap{\xi_3}-\xi_3\jap{\xi_2})}{\xi_3^2-\xi_2^2}
\]
In the high-low regime, $|\xi_3|\simeq|\xi_1|\simeq 2^k$, $|\xi_2|\simeq 2^\ell$ whence 
\EQ{\label{eq:m1bd}
 \bigg| \partial_1^{n_1} \partial_2^{n_2} \frac{1}{\partial_2 \Psi_2(\xi_1,\xi_2)} \bigg| &\le C_{n_1,n_2}\, 2^{2\ell} 2^{-n_1 k - n_2 \ell} \quad \text{for all } n_1, n_2 \geq 0.
}
Thus, writing $\psi_k(\xi_1) = \psi(2^{-k}\xi_1)$ for $k\ge1$, we have 
for $m=m_1$, 
\begin{equation} \label{eq:L1m}
\begin{aligned}
&\bigg\| \int_{\R^2} e^{i(x_1\xi_1 + x_2\xi_2)} \; m(\xi_1,\xi_2) \, \ud \xi_1 \, \ud \xi_2 \bigg\|_{L^1_{x_1,x_2}(\R^2)}   \\
&= \bigg\| \int_{\R^2} e^{i(x_1\eta_1 + x_2\eta_2)} \; \frac{1}{\partial_2 \Psi_2(2^k\eta_1,2^\ell \eta_2)} \psi(\eta_1)\psi(\eta_2) \, \ud \eta_1 \, \ud \eta_2 \bigg\|_{L^1_{x_1,x_2}(\R^2)}\les 2^{2\ell} 
\end{aligned}
\end{equation}
at least for $k\ge1$ and $\ell\ge1$. In case $k=0$, say, then a $\psi_0$ appears in the second line and similarly with~$\ell$. The corresponding bound for $m=m_2$ is by a factor of $2^\ell$ smaller.  
To apply Lemma~\ref{lem:pseudoprodop} we use that by \eqref{equ:relation_L_partial_xi}, 
\EQ{\nn
 \| \jxi^2 \partial_\xi \hat{f}(t,\xi) \|_{L^2_\xi} &\simeq \|\jap{D} Lv(t)\|_{L^2_x}.
}
Thus, by Lemma~\ref{lem:pseudoprodop},
\EQ{\label{eq:Lem2.7ap1}
|J_{k\ell}| &\les \frac{2^{\hf j}}{t} 2^{2\ell} \Big( 2^{-4k} \| \jD^2 v(t)\|_{L^2_x} 2^{-2\ell} \| \jap{D} L \jap{D}^{-2} v(t)\|_{L^2_x} \|v(t)\|_{L^\infty_x} \\
&\qquad \qquad \quad + 2^{-4k} \| \jD^2 v(t)\|_{L^2_x} 2^{-2\ell} \|v(t)\|_{L^\infty_x} 2^{-2k} \| \jap{D} L v(t)\|_{L^2_x} \\
&\qquad \qquad \quad + 2^{-4k} \| \jD^2 v(t)\|_{L^2_x} 2^{-2\ell} \|v(t)\|_{L^\infty_x} 2^{-2k } \| \jap{D}^2 v(t)\|_{L^2_x} \Big) \\
&\les \frac{2^{\hf j}}{t^{\frac32-2\delta}} 2^{-4k} N(T)^3.
}
Note that we obtained better decay in frequency due to the $\jap{D}^{-2}$ smoothing in the first two $h$-factors. However, the final bound in~\eqref{eq:Lem2.7ap1} does not require it. Moreover, in the first line we carried out the commutator
\EQ{\nn 
[L,\jap{D}^{-2}] &= \jap{D}^{-2} [ \jap{D}^2,L]\jap{D}^{-2} =  \jap{D}^{-1} [ \jap{D}^2,x]\jap{D}^{-2}= -2 \px \jap{D}^{-3},
}
and hence,
\begin{equation*}
 \| \jap{D} L \jap{D}^{-2} v(t)\|_{L^2_x} \les \| \jap{D}^{-1} L v(t)\|_{L^2_x} + \| \px \jap{D}^{-2} v(t)\|_{L^2_x} \lesssim N(T) \jt^\delta.
\end{equation*}
Localizing $v(t)$ to frequency $2^\ell$ in the last line, we could gain another factor of~$2^{-2\ell}$. But we do not exploit this extra gain here. 
Finally, summing over the parameters of the high-low case yields
\[
 |\calJ_{hl}| \les N(T)^3 t^{-\frac32+2\delta}.
\]

By an analogous argument,  we arrive at the low-high bound 
\[
 |\calJ_{l h}| \les N(T)^3 t^{-\frac32+2\delta}.
\]
In fact, this follows by interchanging $\xi_1$ and $\xi_2$ in the high-low analysis. 

In the high-high regime, we make the following claim
\EQ{
\label{eq:untere1}
(\partial_1\Psi_2)^2+ (\partial_2\Psi_2)^2 \gtrsim \jap{\xi_3}^{-4}.
}
To see this, we note that if $\xi_1\xi_3\le 0$, then 
\[
 |\partial_1\Psi_2(\xi_1,\xi_2)|= \bigg| \frac{\xi_1}{\jap{\xi_1}} - \frac{\xi_3}{\jap{\xi_3}} \bigg| \gtrsim 1. 
\]
On the other hand, if $\xi_1\xi_3>0$, then
\EQ{\nn 
 |\partial_1\Psi_2(\xi_1,\xi_2)| &= \frac{|\xi-\xi_2||\xi-(2\xi_1+\xi_2)|}{\jap{\xi_1}\jap{\xi_3}|\xi_1\jap{\xi_3}+\xi_3\jap{\xi_1}|} \\
 &\simeq \frac{|\xi_2||\xi-(2\xi_1+\xi_2)|}{ \xi_1^2 \jap{\xi_3}^2 } \\
 &\gtrsim 2^{-k} \jap{\xi_3}^{-2} |\xi-(2\xi_1+\xi_2)|.
 }
Analogously, if $\xi_2\xi_3\ge 0$, then 
\[
 |\partial_2\Psi_2(\xi_1,\xi_2)|= \bigg| \frac{\xi_2}{\jap{\xi_2}} + \frac{\xi_3}{\jap{\xi_3}} \bigg| \gtrsim 1. 
\]
On the other hand, if $\xi_2\xi_3<0$, then
\EQ{\nn 
 |\partial_2\Psi_2(\xi_1,\xi_2)|  \gtrsim 2^{-k} \jap{\xi_3}^{-2} |\xi-(2\xi_2+\xi_1)|.
 }
But in the high-high regime, one has
\[
 |2\xi_1+\xi_2| + |\xi_1+2\xi_2|\simeq 2^k,
\]
whence the claim. Define
\EQ{\label{eq:calLdef}
 \calL:= \big( (\partial_1\Psi_2)^2+ (\partial_2\Psi_2)^2\big)^{-1} \bigl( (\partial_1\Psi_2) \partial_1 + (\partial_2\Psi_2) \partial_2\bigr)
}
so that $\calL ( e^{it\Psi_2} ) = it e^{it\Psi_2}$ and
\EQ{\label{eq:Jkln}
J_{k\ell} & = \frac{\jxi^{\frac12}}{it} \iint e^{it\Psi_2(\xi_1,\xi_2)}\; \calL^* \bigl( \hat{h}_k(t,\xi_1)\hat{\bar{h}}_\ell(t,\xi_2)\hat{f}(t,\xi_3) \bigr) \, \ud \xi_1 \, \ud\xi_2 \\
& = \frac{\jxi^{\frac12}}{it} \sum_{0\le n\le k+10}  \iint e^{it\Psi_2(\xi_1,\xi_2)}\; \calL^* \bigl( \hat{h}_k(t,\xi_1)\hat{\bar{h}}_\ell(t,\xi_2)\hat{f}_n(t,\xi_3) \bigr) \, \ud \xi_1 \, \ud\xi_2 \\
&=: \sum_{0\le n\le k+10}  J_{k\ell n}.
}
In the last line we introduced another dyadic frequency decomposition relative to~$\xi_3$. 
One checks that 
\begin{equation*}
 |\partial^\beta \Psi_2(\xi_1,\xi_2)|\les \jap{\xi_3}^{-1-|\beta|}
\end{equation*}
for any multi-index $\beta=(\beta_1,\beta_2)$ with $|\beta| \geq 2$. Then we apply Lemma~\ref{lem:pseudoprodop} with the following choices of $m(\xi_1,\xi_2)$:
\EQ{
\label{eq:m choice}
m_1(\xi_1,\xi_2) &:= - \big( (\partial_1\Psi_2)^2+ (\partial_2\Psi_2)^2\big)^{-1}\partial_1\Psi_2(\xi_1,\xi_2) \psi_k(\xi_1)\psi_\ell(\xi_2) \psi_n(\xi_3), \\
m_2(\xi_1,\xi_2) &:= - \big( (\partial_1\Psi_2)^2+ (\partial_2\Psi_2)^2\big)^{-1}\partial_2\Psi_2(\xi_1,\xi_2) \psi_k(\xi_1)\psi_\ell(\xi_2)\psi_n(\xi_3), \\
m_3 &:= \partial_1 m_1 + \partial_2 m_2.
}
Now, with the same conventions as in \eqref{eq:L1m}, 
\EQ{ \nn
& \bigg\| \int_{\R^2} e^{i(x_1\xi_1+x_2\xi_2)} m_1(\xi_1,\xi_2)\, \ud\xi_1\ud \xi_2 \bigg\|_{L^1_{x_1,x_2}(\bbR^2)} \\
& = \bigg\| \int_{\R^2} e^{i(x_1\eta_1+x_2\eta_2)} \big( (\partial_1\Psi_2)^2+ (\partial_2 \Psi_2)^2\big)^{-1} \partial_1 \Psi_2(2^k \eta_1,2^\ell\eta_2) \psi(\eta_1)\psi(\eta_2) \\
&\qquad \qquad \qquad \qquad \qquad \qquad \quad \times \psi(2^{-n}(\xi-2^k\eta_1 - 2^\ell\eta_2)) \, \ud\eta_1 \, \ud \eta_2 \bigg\|_{L^1_{x_1,x_2}(\bbR^2)} \\
&\les 2^{2n} (2^k 2^\ell 2^{-2n})^{\frac32}.
}
The $\frac32$ power here produces pointwise decay of the form $(\jap{x_1}\jap{x_2})^{-\frac32}$. The other choices of $m$, i.e., $m_2$ and $m_3$, satisfy the same bound.  In analogy to~\eqref{eq:Lem2.7ap1} we conclude that  
\begin{equation} \label{eq:Lem2.7ap2}
\begin{aligned}
|J_{k\ell n}| &\les \frac{2^{\hf j}}{t} 2^{3k} 2^{-n} \Big( 2^{-4k} \| \jD^2 v(t)\|_{L^2_x} 2^{-2\ell} \| \jap{D} L \jap{D}^{-2} v(t)\|_{L^2_x}  \|v(t)\|_{L^\infty_x} \\
&\qquad \qquad \qquad \quad + 2^{-4k} \|\jD^2 v(t)\|_{L^2_x} 2^{-2\ell} \|v(t)\|_{L^\infty_x} 2^{-2n} \| \jap{D} L v(t)\|_{L^2_x}  \\
&\qquad \qquad \qquad \quad + 2^{-4k} \|\jD^2 v(t)\|_{L^2_x} 2^{-2\ell} \|v(t)\|_{L^\infty_x} 2^{-2n} \| \jap{D}^2 v(t)\|_{L^2_x} \Big) \\
&\les \frac{2^{\hf j}}{t^{\frac32-2\delta}} 2^{-3k} 2^{-n} N(T)^3. 
\end{aligned}
\end{equation}
Summing over the high-high parameter regime  yields 
\[
 |\calJ_{hh}| \les N(T)^3 t^{-\frac32+2\delta},
\]
as claimed. 
\end{proof}

It remains to consider the integral $\calJ_{crit}$, which contains the critical point $(\xi,-\xi,\xi)$. The region in question is of the form $|\xi_1|+|\xi_2|\les |\xi|$ and on the hyperplane $\Pi_\xi$ we have 
$
 \sum_{j=1}^3 |\xi_j|\simeq|\xi|
$. 
The unique critical point is at $(\xi_1, \xi_2) = (\xi,-\xi)$ and we denote 
\[
 U_\ast := \biggl\{ (\xi_1, \xi_2) \in \bbR^2 \, : \, \max \, \bigl\{ |\xi_1-\xi|, |\xi_2+\xi| \bigr\} \leq c_* |\xi| \biggr\}.
\] 
For simplicity we assume $\xi\gg1$. Here $0 < c_\ast \ll 1$ is a small absolute constant that will be specified further below.

\begin{lemma} \label{lem:nablaPsi2 1}
 The neighborhood $U_\ast$ is characterized by the property $$|\nabla \Psi_2(\xi_1,\xi_2)|\ll|\xi|^{-2}$$ and in $U_\ast$
\EQ{\label{eq:near crit}
|\nabla\Psi_2(\xi+\eta_1,-\xi+\eta_2)|\simeq |\xi|^{-3}(|\eta_1|+|\eta_2|).
}
 In the region $I:=\{ \xi_1\xi_3\le0\}\cup\{\xi_2\xi_3\ge0\}$ we have $|\nabla \Psi_2|\simeq 1$, while both in 
region $II:=\{ \xi_1 \le0\}\cap\{\xi_2\ge0\}\cap\{\xi_3\le0\}$ 
as well as in 
region $III:=\{ \xi_1 \ge0\}\cap\{\xi_2\le0\}\cap\{\xi_3\ge0\}$, but outside of $U_*$,  we have 
\EQ{\label{eq:III}
|\nabla\Psi_2(\xi_1,\xi_2)|\simeq \jap{\xi_1}^{-2}+\jap{\xi_2}^{-2}+\jap{\xi_3}^{-2}.
}
In Figure~\ref{fig:gebiete}, Region~$I$ is represented by the shaded areas, Region~$II$ is the upper blank triangle, Region~$III$ the lower blank area which contains the disk depicting $U_*$. 
\end{lemma}
\begin{proof}
 The stated property in region $I$ follows from 
 \EQ{\label{eq:nabP2}
  \nabla\Psi_2(\xi_1,\xi_2) = \bigg( \frac{\xi_1}{\jap{\xi_1}} - \frac{\xi_3}{\jap{\xi_3}}, -\frac{\xi_2}{\jap{\xi_2}} - \frac{\xi_3}{\jap{\xi_3}} \bigg)
 }
and $\max \, \{ |\xi_1|,|\xi_2|,|\xi_3| \} \simeq |\xi| \gg 1$. 

\noindent To analyze the gradient near the critical point we set $\xi_1=\xi+\eta_1$, $\xi_2=-\xi+\eta_2$ and calculate
\EQ{ \label{eq:FPhi} 
 &|\nabla\Psi_2(\xi+\eta_1,-\xi+\eta_2)|^2 \\
 &\quad = \bigg| \frac{\xi-\eta_1-\eta_2}{\jap{\xi-\eta_1-\eta_2}}-\frac{\xi+\eta_1}{\jap{\xi+\eta_1}} \bigg|^2 + \bigg| \frac{\xi-\eta_1-\eta_2}{\jap{\xi-\eta_1-\eta_2}}-\frac{\xi-\eta_2}{\jap{\xi-\eta_2}} \bigg|^2 \\
 &\quad = (2\eta_1+\eta_2)^2 \Phi^2(\xi+\eta_1,2\eta_1+\eta_2) + \eta_1^2 \Phi^2(\xi-\eta_2,\eta_1), 
}
where we introduce 
\begin{equation*}
 \Phi(\sigma,\eta) := \int_0^1 \jap{\sigma-s\eta}^{-3} \, \ud s.
\end{equation*}
Note that $\Phi$ is even and $\Phi(\sigma,\eta)=\Phi(\sigma-\eta,-\eta)=\Phi(\eta-\sigma,\eta)$. 
Next, we will show that  for $\sigma \ge 0$, the function $\Phi$ has the following shape:
\EQ{\label{eq:PhiTab}
\Phi(\sigma,\eta) \simeq \left\{ \begin{array}{ccl}
  \jap{\eta}^{-1}\jap{\sigma}^{-2}   & \text{if} & \eta\le -\sigma,  \\ 
    \jap{\sigma}^{-1}\jap{\sigma-\eta}^{-2} & \text{if} &  -\sigma\le\eta\le\sigma, \\
     \jap{\eta}^{-1} &\text{if} &  \eta\ge\sigma.
\end{array}\right.
}
In particular, if $|\eta|\geq c_0\jap{\sigma}$ for some absolute constant $0 < c_0 \ll 1$,  then 
\EQ{ \label{eq:lemimpl}
 |\eta| \Phi(\sigma,\eta) \gtrsim c_0 \jap{\sigma}^{-2}.
}
All implied constants are absolute. Since $\Phi$ is even, \eqref{eq:lemimpl} also holds for $\sigma\le0$. 
To verify \eqref{eq:PhiTab}, note that if $\eta\ne0$, then 
\EQ{\label{eq:Phi2}
\Phi(\sigma,\eta) = \eta^{-1}\int_0^\eta \jap{\sigma-\sigma'}^{-3} \, \ud\sigma'.
}
Suppose $0\le\sigma\le 100$. Then $\jap{\sigma-\sigma'}\simeq \jap{\sigma'}$ and by~\eqref{eq:Phi2} we conclude that $\Phi(\sigma,\eta)\simeq\jap{\eta}^{-1}$, as claimed in~\eqref{eq:PhiTab}. 
Henceforth $\sigma\ge100$. If $\eta\le-\sigma$, and $\sigma'$ is as in \eqref{eq:Phi2}, then $\jap{\sigma-\sigma' } = \jap{\sigma + |\sigma'|}$ and 
\begin{align*}
 \Phi(\sigma,\eta) \simeq |\eta|^{-1}\int_0^{|\eta|} \jap{\sigma+\zeta}^{-3} \, \ud \zeta \simeq \jap{\eta}^{-1} \jap{\sigma}^{-2},
\end{align*}
which gives the first line of~\eqref{eq:PhiTab}. 
If $-\sigma \le \eta\le \frac12\sigma $, then $\jap{\sigma-\sigma'}\simeq \jap{\sigma}$, whence $\Phi(\sigma,\eta)\simeq \jap{\sigma}^{-3}$. This agrees with the second line of~\eqref{eq:PhiTab}. 
If $ \frac12\sigma\le\eta\le\sigma$, then $\eta=\sigma-\ell$ with $0\le\ell\le\frac12\sigma$ and 
\EQ{\nn 
\Phi(\sigma,\eta) = \eta^{-1}\int_0^{\sigma-\ell} \jap{\sigma-\sigma'}^{-3} \, \ud\sigma' &= \eta^{-1}\int_\ell^\sigma \jap{\zeta}^{-3}\, \ud\zeta \\
&\simeq \jap{\eta}^{-1}\jap{\ell}^{-2}\simeq \jap{\sigma}^{-1}\jap{\sigma-\eta}^{-2},
}
which concludes the proof of the second line of~\eqref{eq:PhiTab}.
Finally, if $\eta\ge\sigma$ we have
\[
 \Phi(\sigma,\eta) = \eta^{-1} \int_{\sigma-\eta}^{\sigma} \jap{\zeta}^{-3} \, \ud\zeta\simeq \jap{\eta}^{-1},
\]
as claimed. 
\begin{figure}[ht]
\includegraphics[width=0.5\textwidth]{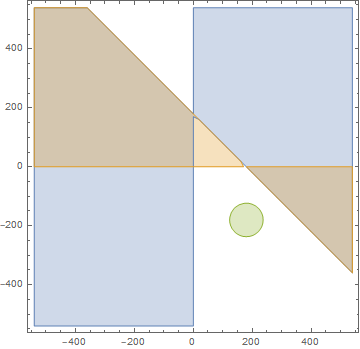}
\caption{Regions  in the $(\xi_1,\xi_2)$ plane, $\xi>0$  }
\label{fig:gebiete}
\end{figure}

Next, we characterize the region $U_\ast$.
By \eqref{eq:FPhi} and \eqref{eq:lemimpl}, if $|\eta_1|\ge c_0\jap{\xi-\eta_2}$ then with absolute implied constants, 
\[
|\nabla \Psi_2(\xi_1,\xi_2)|  \gtrsim c_0\jap{\xi-\eta_2}^{-2}. 
\]
If $|2\eta_1+\eta_2|\ge  c_0\jap{\xi+\eta_1}$ then similarly  
\[
|\nabla \Psi_2(\xi_1,\xi_2)|  \gtrsim c_0\jap{\xi+\eta_1}^{-2}. 
\]
Note that  \[|2\eta_1+\eta_2|+|\eta_1|\gtrsim \jap{\xi-\eta_2} + \jap{\xi+\eta_1}\]  holds  with a uniform constant in the region $\max \, \{ |\eta_1|,|\eta_2| \} \ge \frac{1}{10} \jap{\xi}$.  
If $|\xi| \gtrsim \max\, \{ |\eta_1|, |\eta_2| \} \ge \frac{1}{10} |\xi|$, one therefore has 
\EQ{ \label{eq:Flow}
|\nabla \Psi_2(\xi_1,\xi_2)| \ge c_1 \min\, \bigl\{ \jap{\xi-\eta_2}^{-2}, \jap{\xi+\eta_1 }^{-2} \bigr\} \gtrsim |\xi|^{-2}
}
with an absolute constant $c_1>0$. Finally, if $ \max\, \{ |\eta_1|, |\eta_2| \} \le \frac{1}{10} |\xi|$, then from \eqref{eq:FPhi}  and the second line of~\eqref{eq:PhiTab}, 
\[
 |\nabla \Psi_2(\xi_1,\xi_2)|  \simeq (|2\eta_1+\eta_2|+|\eta_1|)|\xi|^{-3},
\]
which is the same as \eqref{eq:near crit}. 
This  concludes our characterization of~$U_*$. 

To prove \eqref{eq:III} in Region~$II$, we start from 
\EQ{\label{eq:IIbew}
 |\nabla\Psi_2(\xi_1,\xi_2)|\simeq |\xi_1-\xi_3|\Phi(-\xi_3,\xi_1-\xi_3) + |\xi_1-\xi|\Phi(\xi_2,\xi-\xi_1).
}
This follows from \eqref{eq:FPhi}, $\xi_1=\xi+\eta_1, \xi-\eta_2=-\xi_2$, $\xi_3=\xi-\eta_1-\eta_2$, and the symmetries of~$\Phi$. Note that the first arguments in both $\Phi$ terms are nonnegative.  
Since $1\ll\xi\le \xi-\xi_1\le\xi_2$ in $II$, the final inequality here being due to $\xi_3\le0$, the second line of \eqref{eq:PhiTab} applies to the last term on the right-hand side of~\eqref{eq:IIbew}, whence 
\[
 |\xi_1-\xi|\Phi(\xi_2,\xi-\xi_1)\simeq \jap{\xi_3}^{-2}.
\]
 In total we infer from \eqref{eq:PhiTab} that in Region~$II$ 
 \begin{equation} \label{eq:IIzwischen}
 \begin{aligned}
 |\nabla\Psi_2(\xi_1,\xi_2)| &\simeq |\xi_1-\xi_3|\big(  \jap{\xi_1}^{-2}\jap{\xi_3}^{-1}\one_{[\xi_3\le\xi_1-\xi_3\le-\xi_3]} \\
 &\qquad \qquad \qquad + \jap{\xi_1-\xi_3}^{-1}\jap{\xi_3}^{-2}\one_{[\xi_1-\xi_3\le\xi_3]} \big)+ \jap{\xi_3}^{-2} \\
 &\simeq |\xi_1-\xi_3|  \jap{\xi_1}^{-2}\jap{\xi_3}^{-1}\one_{[\xi_3\le\xi_1-\xi_3\le-\xi_3]} + \jap{\xi_3}^{-2}, 
 \end{aligned}
 \end{equation}
 where we absorbed the second term on the right-hand side of the first line into the $\jap{\xi_3}^{-2}$. 
 The first term of~\eqref{eq:IIzwischen} is bounded above by
 \[
   |\xi_1-\xi_3|  \jap{\xi_1}^{-2}\jap{\xi_3}^{-1}\les  \jap{\xi_1}^{-1}\jap{\xi_3}^{-1} +  \jap{\xi_1}^{-2} \les  \jap{\xi_1}^{-2}+ \jap{\xi_3}^{-2}.
 \]
The goal is therefore to show that the right-hand side of~\eqref{eq:IIzwischen} is $\gtrsim \jap{\xi_1}^{-2}+ \jap{\xi_3}^{-2}$ as then 
\eqref{eq:III} follows easily since $\xi_2\simeq\xi$ in~$II$. The desired lower bound holds if $\jap{\xi_3}\les \jap{\xi_1}$ so that only $\jap{\xi_3}\gg\jap{\xi_1}$ remains as a possible obstruction. However, in that case $|\xi_1-\xi_3|\gtrsim \jap{\xi_3}$ and $\xi_3\le\xi_1-\xi_3$ since $\xi_3\le0$,  while $\xi_1-\xi_3\le-\xi_3$ holds automatically in~$II$. In summary,  
\[
 |\xi_1-\xi_3|  \jap{\xi_1}^{-2}\jap{\xi_3}^{-1}\one_{[\xi_3\le\xi_1-\xi_3\le-\xi_3]}\gtrsim \jap{\xi_1}^{-2} 
\]
and we are done.

In Region~$III$ we modify \eqref{eq:IIbew} to ensure the first argument of $\Phi$ is nonnegative, viz.
\EQ{\nn
 &|\nabla\Psi_2(\xi_1,\xi_2)| \\
 &\quad \simeq |\xi_1-\xi_3|\Phi(\xi_3,\xi_3-\xi_1) + |\xi_1-\xi|\Phi(-\xi_2,\xi_1-\xi) \\
 &\quad \simeq |\xi_1-\xi_3|\Big (\jap{\xi_3}^{-1}\jap{\xi_1}^{-2} \one_{[-\xi_3\le \xi_3-\xi_1]} + \jap{\xi_3}^{-2}\jap{\xi_1-\xi_3}^{-1} \one_{[-\xi_3\ge \xi_3-\xi_1]} \Big) \\
 &\quad \quad + |\xi_1-\xi| \Big(\jap{\xi_2}^{-1}\jap{\xi_3}^{-2} \one_{[\xi_2\le \xi_1-\xi]} + \jap{\xi_2}^{-2}\jap{\xi_1-\xi}^{-1} \one_{[\xi_2\ge \xi_1-\xi]} \Big).
}
In the first indicator, we automatically have $\xi_3-\xi_1\le\xi_3$ due to $\xi_1\ge0$ and in the third, $\xi_1-\xi\le-\xi_2$ holds due to $\xi_3\ge0$. 
  Figure~\ref{fig:gebieteIII} shows Region~$III$ in the fourth quadrant, below the line $\xi_3=0$ (thus, we remove the triangle in the upper right-hand corner). Also shown are subregions determined by the lines $\xi_1=\xi_3$, respectively $\xi_1=2\xi_3$, and $\xi=\xi_1-\xi_2$. The critical point lies on the line~$\xi_1=\xi_3$. We refer to the red triangle in the upper left corner as $A$ (it is $\xi>\xi_1-\xi_2$), and 
  we denote the three different colored regions in $III$ outside $A$ as respectively $B,C,D$, moving from left to right. 
  
If $\xi_2\le\xi_1-\xi$ and $-\xi_3\le\xi_3-\xi_1$, which is $B\cup C$,  then $\xi_3\simeq-\xi_2\simeq\xi$ and 
\EQ{\nn
  |\nabla\Psi_2(\xi_1,\xi_2)|
& \simeq |\xi_1-\xi_3| \jap{\xi_3}^{-1}\jap{\xi_1}^{-2} 
+ |\xi_1-\xi| \jap{\xi_2}^{-1}\jap{\xi_3}^{-2} \\
&\simeq |\xi_1-\xi_3| \jap{\xi}^{-1}\jap{\xi_1}^{-2} 
+ |\xi_1-\xi| \jap{\xi}^{-3}. 
  }
  Outside of $U_*$ the last line is $\simeq \jap{\xi_1}^{-2}$. Indeed, if $|\xi_1-\xi_3|\ll\xi$, then in $(B\cup C)\setminus U_*$ it follows that $\xi_1\simeq\xi$ and $|\xi_1-\xi|\simeq \xi$. 
\begin{figure}[t]
\includegraphics[width=0.5\textwidth]{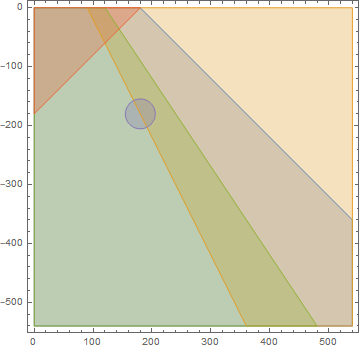}
\caption{Region $III$ with $U_*$, $\xi_1>\xi_3$, $\xi_1<2\xi_3$, $\xi>\xi_1-\xi_2$  }
\label{fig:gebieteIII}
\end{figure}

\noindent If $\xi_2\le\xi_1-\xi$ and $-\xi_3\ge\xi_3-\xi_1$, which is $D$, then $\xi_3\les-\xi_2$, $\xi_1\simeq\xi_1-\xi_3\simeq\xi$, $|\xi_1-\xi|\les -\xi_2$ and 
\EQ{\nn
 |\nabla\Psi_2(\xi_1,\xi_2)|
 &\simeq \big(|\xi_1-\xi_3| \jap{\xi_1-\xi_3}^{-1}  
 + |\xi_1-\xi| \jap{\xi_2}^{-1}\big)\jap{\xi_3}^{-2} \simeq  \jap{\xi_3}^{-2}.  
}
  
If $\xi_2\ge\xi_1-\xi$,  which is $A$, then $\xi-\xi_1\simeq\xi_3$,   and 
\begin{align}
  |\nabla\Psi_2(\xi_1,\xi_2)| 
& \simeq |\xi_1-\xi_3|\Big (\jap{\xi_3}^{-1}\jap{\xi_1}^{-2} \one_{[-\xi_3\le \xi_3-\xi_1]} + \jap{\xi_3}^{-2}\jap{\xi_1-\xi_3}^{-1} \one_{[-\xi_3\ge \xi_3-\xi_1]} \Big) \label{eq:ersteZ} \\
&\qquad + |\xi_1-\xi| \jap{\xi_1-\xi}^{-1}   \jap{\xi_2}^{-2}
\label{eq:zweiteZ} \\
&\simeq (\jap{\xi_2}^{-2}+\jap{\xi_3}^{-2}) \one_{[-\xi_3\ge \xi_3-\xi_1]} + (\jap{\xi_1}^{-2}+\jap{\xi_2}^{-2})\one_{[-\xi_3\le \xi_3-\xi_1]} \nn \\ 
&\simeq \jap{\xi_1}^{-2}+\jap{\xi_2}^{-2} + \jap{\xi_3}^{-2}. \nn
\end{align}
To see this in more detail, let $A_1=:A\cap\{ 2\xi_3\ge\xi_1\}$ and $A_2:=A\cap\{ 2\xi_3\le \xi_1\}$. 
Thus, $A_1$ corresponds to the first indicator in~\eqref{eq:ersteZ}, whereas $A_2$ corresponds to the second indicator. In the figure, $A_2$ is the small triangle in $A$ above the line $2\xi_3=\xi_1$, whereas $A_1$ is the quadrilateral in $A$ below that line.  In $A_1$, we have $\xi-\xi_1\simeq \xi\simeq\xi_3$ and in view of~\eqref{eq:ersteZ}, \eqref{eq:zweiteZ},
\[
 |\nabla\Psi_2(\xi_1,\xi_2)| \simeq  |\xi_1-\xi_3|\jap{\xi_3}^{-1}\jap{\xi_1}^{-2} +  \jap{\xi_2}^{-2}
 \les \jap{\xi_1}^{-2} +  \jap{\xi_2}^{-2}.
\]
The reverse inequality holds if $\jap{\xi_1}\gtrsim\jap{\xi_2}$. On the other hand, if $\jap{\xi_1}\ll\jap{\xi_2}$, then $|\xi_1-\xi_3|\sim\xi$ and we are again done. 

In $A_2$, we have $|\xi_1-\xi_3|\simeq \xi$ and so \eqref{eq:ersteZ}, \eqref{eq:zweiteZ} imply that 
\[
 |\nabla\Psi_2(\xi_1,\xi_2)| \simeq  \jap{\xi_3}^{-2} + |\xi_1-\xi| \jap{\xi_1-\xi}^{-1}  \jap{\xi_2}^{-2}
 \les \jap{\xi_2}^{-2} +  \jap{\xi_3}^{-2}.
\]  
The reverse inequality holds if $\jap{\xi_3}\les\jap{\xi_2}$. On the other hand, if  $\jap{\xi_3}\gg\jap{\xi_2}$, then $|\xi_1-\xi|\simeq \xi_3\simeq \jap{\xi_3}$, whence $|\xi_1-\xi| \jap{\xi_1-\xi}^{-1} \simeq1$, which concludes this analysis. 
Finally, in all of these cases the $\jap{\xi_j}^{-2}$ absent from the final estimate give smaller contributions.   
\end{proof}

Returning to the oscillatory integral $\calJ_{crit}$, we write
\[
 \calJ_{crit} = \calJ_{U_\ast} + \calJ_{U_\ast}^c
\]
with 
\EQ{
\label{eq:JU*}
\calJ_{U_*} &:=  \jxi^{\frac12} \iint e^{it\Psi_2(\xi_1,\xi_2)}\; \hat{h}(t,\xi_1)\hat{\bar{h}}(t,\xi_2)\hat{f}(t,\xi_3)\chi_{U_*}(\xi_1,\xi_2) \, \ud \xi_1 \, \ud \xi_2, 
}
where $\chi_{U_\ast}(\xi_1,\xi_2)$ is a smooth bump function adapted to~$U_*$. 
Next, we determine the leading order behavior of $\calJ_{U_\ast}$.

\begin{lemma} \label{lem:U*StatPhase}
 Assume $T \geq 1$. For any $0<\alpha<\frac14$, we have uniformly for all $\xi \in \bbR$ and all times $1 \leq t \leq T$ that
 \EQ{ \label{eq:JU* final}
  \calJ_{U_*} 
  &= 2\pi \jxi^{-\frac12}\; t^{-1} |\hat{f}(t,\xi)|^2\hat{f}(t,\xi)  + \calO_{L^\infty_\xi} \big( N(T)^3 t^{-1-\alpha+3\delta} \big).
 }
\end{lemma}
\begin{proof}
 We substitute 
\begin{equation*}
 \xi_1=\xi+\jxi \zeta_1, \quad \xi_2=-\xi+\jxi\zeta_2,
\end{equation*}
in \eqref{eq:JU*} and rescale the phase as follows
\[
 \Psi_2(\xi_1,\xi_2) =  \jxi^{-1} \Psi(\zeta_1,\zeta_2).
\]
Then $\partial_{\zeta_1} \Psi(0,0) = \partial_{\zeta_2} \Psi(0,0) = 0$, $\Psi(0,0)=0$,  and by \eqref{eq:Hess}, 
\[
 \mathrm{Hess} \, \Psi(0,0) = \left[ \begin{matrix} 
                                  2 & 1 \\
                                  1 & 0
                                 \end{matrix} \right].
\]
Moreover, by Lemma~\ref{lem:nablaPsi2 1} we have for $|\zeta_1|+|\zeta_2|\les 1$ and all multi-indices $\beta=(\beta_1,\beta_2)$ with $|\beta| \geq 1$,
\[
 |\partial^\beta \Psi(\zeta_1,\zeta_2)| \le C_\beta.
\]
We set
\[
  F(\zeta_1,\zeta_2) := \hath_j(t,\xi_1) \hat{\bar{h}}_j(t,\xi_2) \hatf_j(t,\xi_3),  
\]
and  $\lambda:=t  \jxi^{-1}$, and  $\chi_{U_*}(\xi_1,\xi_2)=\chi_{0}(\zeta_1,\zeta_2)$, the latter being a smooth cutoff to a neighborhood of $(0,0)$ of size $c_*\ll1$, which equals $1$ near the origin. Then  
\EQ{ 
\calJ_{U_\ast} & = \jxi^{\frac52} \iint e^{i\lambda\Psi(\zeta_1,\zeta_2)}\,\chi_{0}(\zeta_1,\zeta_2) \, F(\zeta_1,\zeta_2) \, \ud \zeta_1 \, \ud\zeta_2  \\
&= (2\pi)^{-1} \jxi^{\frac52} \iint G_\lambda(z_1,z_2) \, \wh{F}(z_1,z_2) \, \ud z_1 \, \ud z_2, 
}
where
\EQ{\nn
 G_\lambda(z_1,z_2) &:= \iint  e^{i(z_1\zeta_1+z_2\zeta_2)} e^{i\lambda\Psi(\zeta_1,\zeta_2)}\,\chi_{0}(\zeta_1,\zeta_2) \, \ud \zeta_1 \, \ud \zeta_2. 
}
We conclude that
\EQ{ \label{eq:JU*2 StatPhase}
  \calJ_{U_\ast} &=   \jxi^{\frac52} G_\lambda(0,0) F(0,0) + \calO \Big( \jxi^{\frac52} \bigl\| ( G_\lambda-G_\lambda(0,0)) \widehat{F} \bigr\|_{L^1_{z_1,z_2}} \Big).
}
By stationary phase, see~\cite[Theorem 7.7.5]{Hor}, we have for $\lambda \geq 1$,
\[
 G_\lambda(0,0) = 2\pi \lambda^{-1} + \calO(\lambda^{-2}),
\]
while trivially $G_\lambda(0,0)=\calO(1)$ for $0<\lambda<1$. 
Moreover, if $R:=|z_1|+|z_2|\ge c_{**}\lambda$, then 
\begin{equation*}
 |G_\lambda(z_1,z_2)|\les R^{-1}\les\lambda^{-1-\alpha} R^\alpha
\end{equation*}
by one integration by parts. Here $0<c_{**}<1$ is a constant that can be taken to be a multiple of $c_*$ in the definition of~$U_*$. Hence, for any $\alpha\ge0$, we have uniformly in $\lambda > 0$,
\EQ{ \label{eq:G1}
 \iint_{\{R\ge c_{**} \lambda\}} \bigl| G_\lambda(z_1,z_2)-G_\lambda(0,0) \bigr| \bigl| \widehat{F}(z_1,z_2) \bigr| \, \ud z_1 \, \ud z_2 \les \lambda^{-1-\alpha} \bigl\| (|z_1|+|z_2|)^\alpha \widehat{F}(z_1,z_2) \bigr\|_{L^1_{z_1,z_2}}.
}
On the other hand, for $R\le c_{**}\lambda$ we apply~\cite[Theorem 7.7.6]{Hor}
to $G_\lambda$ with phase function 
\begin{equation*}
 \tilde \Psi(\zeta_1, \zeta_2) := (z_1\zeta_1+z_2\zeta_2 )\lambda^{-1} + \Psi(\zeta_1,\zeta_2).
\end{equation*}
This phase has a unique critical point  $(\zeta_1^*,\zeta_2^*)$ and for $\lambda\ge1$,
\[
 G_{\lambda}(z_1,z_2) = 2\pi \lambda^{-1} e^{i\lambda \tilde\Psi(\zeta_1^*,\zeta_2^*)} \bigl| \det \Hess\, \tilde\Psi(\zeta_1^*,\zeta_2^*) \bigr|^{-\frac12} \chi_{0}(\zeta_1^*,\zeta_2^*) + \calO(\lambda^{-2}), 
\]
where $\calO$ is uniform in $(z_1,z_2)$. It follows from $|(\zeta_1^*,\zeta_2^*)|\les \lambda^{-1} R$ and $\lambda |\tilde\Psi(\zeta_1^*,\zeta_2^*)|\les \lambda^{-1} R^2$ that for $R\les\lambda$, $\lambda\ge1$, and any $0 \leq \alpha \leq 1$,
\begin{equation*}
 \begin{aligned}
  &\Big| e^{i\lambda \tilde\Psi(\zeta_1^*,\zeta_2^*)} \big|\det \Hess\, \tilde\Psi(\zeta_1^*,\zeta_2^*)\big|^{-\frac12} \chi_{0}(\zeta_1^*,\zeta_2^*) -1 \Big| \\
  &\lesssim \bigl| e^{i\lambda \tilde\Psi(\zeta_1^*,\zeta_2^*)} - 1 \bigr| \big|\det \Hess\, \tilde\Psi(\zeta_1^*,\zeta_2^*)\big|^{-\frac12} \bigl|\chi_{0}(\zeta_1^*,\zeta_2^*)\big| + \Bigl| \big|\det \Hess\, \tilde\Psi(\zeta_1^*,\zeta_2^*)\big|^{-\frac12} \chi_{0}(\zeta_1^*,\zeta_2^*) -1 \Big| \\
  &\lesssim \lambda^{-\alpha} R^{2\alpha} + \lambda^{-\alpha} R^\alpha.
 \end{aligned}
\end{equation*}
Thus, we obtain uniformly for $\lambda \geq 1$ that  
\begin{align*}
 \iint_{\{ R \leq c_{\ast \ast} \lambda\}} \bigl| G_\lambda(z_1,z_2)-G_\lambda(0,0) \bigr| \bigl| \widehat{F}(z_1,z_2) \bigr| \, \ud z_1 \, \ud z_2 \les \lambda^{-1-\alpha} \bigl\| \langle |z_1|+|z_2| \rangle^{2 \alpha} \widehat{F}(z_1,z_2) \bigr\|_{L^1_{z_1, z_2}}.
\end{align*}
Using the trivial bounds $|G_\lambda(z_1,z_2)|\les 1$, we infer that the preceding estimate continues to hold for $\lambda > 0$. 
Returning to~\eqref{eq:JU*2 StatPhase} we conclude that 
\EQ{\nn 
  \calJ_{U_\ast} 
  &= 2\pi \jxi^{\frac72}\; t^{-1} |\hat{h}(t,\xi)|^2\hat{f}(t,\xi)  + \calO \Big( \jxi^{\frac72 + \alpha} \bigl\| \langle |z_1| + |z_2| \rangle^{2\alpha} \widehat{F} \bigr\|_{L^1_{z_1,z_2}} t^{-1-\alpha}\Big),
}
where
\begin{equation} \label{eq:factor j}
 \begin{aligned}
  \widehat{F}(z_1,z_2) := \jxi^{-2} e^{i\frac{\xi}{\jxi}(z_2-z_1)} \int_{\R} e^{-i\xi x_3} h_j\bigl( t,\jxi^{-1}z_1+x_3 \bigr) \bar{h}_j \bigl(t,\jxi^{-1} z_2+x_3\bigr) f_j(t,x_3) \, \ud x_3.
 \end{aligned}
\end{equation}
In \eqref{eq:factor j}  we set $\xi\simeq 2^j$ and used that in $U_*$ we have $|\xi_j|\simeq \xi$.  To bound the error in~\eqref{eq:JU*2 StatPhase} we compute
\EQ{\nn
 \|\widehat{F}\|_{L^1_{z_1,z_2}} &\le \|h_j(t) \|_{L^1_x}^2 \|{f_j}(t)\|_{L^1_x} \les 2^{-4j} \|f_j(t)\|_{L^1_x}^3, \\
 \bigl\| ( |z_1| + |z_2| )^{2\alpha} \widehat{F} \bigr\|_{L^1_{z_1,z_2}} &\les 2^{2j(\alpha-2)} \||x|^{2\alpha} f_j(t) \|_{L^1_x}  \|{f_j}(t)\|_{L^1_x}^2,  
}
and
\EQ{\nn
\|f_j(t)\|_{L^1_x} &\les 2^{-2j} \bigl( \|\jap{D}^2 f\|_{L^2_x} + \|x \jap{D}^2 f\|_{L^2_x} \bigr) \\
&\les 2^{-2j} \bigl( \|\jap{D}^2 v(t)\|_{L^2_x} + \| \jap{D} L  v(t) \|_{L^2_x} \bigr) \\
&\les 2^{-2j} N(T) \jap{t}^\delta,
}
as well as, with $0<\alpha<\frac14$, 
\EQ{\nn
 \bigl\||x|^{2\alpha} f_j(t) \bigr\|_{L^1_x}  &\les \| \jap{x} f_j(t)\|_{L^2_x} \les \|  v(t)\|_{L^2_x} + \|\jap{D}^{-1} L v(t)\|_{L^2_x} \les N(T) \jap{t}^\delta.
}
In summary,
\[
 \bigl\| \langle |z_1| + |z_2| \rangle^{2\alpha} \widehat{F} \bigr\|_{L^1_{z_1,z_2}} \les 2^{2j(\alpha-4)} N(T)^3 \jap{t}^{3\delta},
\]
and \eqref{eq:JU* final} holds. 
\end{proof}

By means of Lemma~\ref{lem:nablaPsi2 1}, we now show that the contribution of $\calJ_{U_\ast}^c$ is just an error term. 
This completes the analysis of the oscillatory integral $\calT_2(t,\xi)$. 

\begin{lemma} \label{lem:JU*compl}
 Assume $T \geq 1$. We have uniformly for all $\xi \in \bbR$ and all $1 \leq t \leq T$ that
 \[
  |\calJ_{U_\ast}^c| \les N(T)^3 t^{-\frac54+2\delta}.
 \]
\end{lemma}
\begin{proof}
We proceed as in the high-high case above, i.e., with $\calL$ as in~\eqref{eq:calLdef}, 
\EQ{\nn
 \calJ_{U_\ast}^c &= \sum_{0\le k,\ell,n\le j+10} J_{k\ell n}, \\
 J_{k\ell n} & = \frac{\jxi^{\frac12}}{it} \iint e^{it\Psi_2(\xi_1,\xi_2)}\; \calL^* \bigl( \hat{h}_k(t,\xi_1)\hat{\bar{h}}_\ell(t,\xi_2)\hat{f}_n(t,\xi_3) (1-\chi_{U_*})(\xi_1,\xi_2) \bigr) \, \ud \xi_1 \, \ud\xi_2,
}
where $\xi\simeq 2^j$. Expanding, we arrive at
\begin{equation} \label{eq:JklnExp}
\begin{aligned}
J_{k \ell n}  &= \frac{\jxi^{\frac12}}{it} \iint e^{it\Psi_2(\xi_1,\xi_2)}\;m_1(\xi_1,\xi_2)\: \partial_1 \bigl( \hat{h}_k(t,\xi_1)\hat{\bar{h}}_\ell(t,\xi_2)\hat{f}_n(t,\xi_3) \bigr) \, \ud \xi_1 \, \ud \xi_2 \\
&\quad + \frac{\jxi^{\frac12}}{it} \iint e^{it\Psi_2(\xi_1,\xi_2)}\;m_2(\xi_1,\xi_2)\: \partial_2 \bigl( \hat{h}_k(t,\xi_1)\hat{\bar{h}}_\ell(t,\xi_2)\hat{f}_n(t,\xi_3) \bigr) \, \ud \xi_1 \, \ud \xi_2 \\
&\quad + \frac{\jxi^{\frac12}}{it} \iint e^{it\Psi_2(\xi_1,\xi_2)}\;m_3(\xi_1,\xi_2)  \; \hat{h}_k(t,\xi_1)\hat{\bar{h}}_\ell(t,\xi_2)\hat{f}_n(t,\xi_3) \, \ud \xi_1 \, \ud \xi_2 
\end{aligned}
\end{equation}
with 
\EQ{\label{eq:mpdef} 
\mu_i(\xi_1,\xi_2) &:= -\big( (\partial_1\Psi_2)^2+ (\partial_2\Psi_2)^2\big)^{-1}\partial_i\Psi_2(\xi_1,\xi_2), \\ 
m_1(\xi_1,\xi_2) &:= \mu_1(\xi_1,\xi_2) \psi_k(\xi_1)\psi_\ell(\xi_2) \psi_n(\xi_3)(1-\chi_{U_*})(\xi_1,\xi_2), \\
m_2(\xi_1,\xi_2) &:= \mu_2(\xi_1,\xi_2) \psi_k(\xi_1)\psi_\ell(\xi_2)\psi_n(\xi_3) (1-\chi_{U_*})(\xi_1,\xi_2), \\
m_3 &:= \partial_1 m_1 + \partial_2 m_2,
}
cf.~\eqref{eq:m choice}. Define
\[
 A_{k \ell n}:= \sum_{p=1}^3 \| \widehat{m}_p\|_{L^1_{x_1,x_2}}.
 \]
Then Lemma~\ref{lem:pseudoprodop} implies as before, see  \eqref{eq:Lem2.7ap1} and \eqref{eq:Lem2.7ap2}, 
\EQ{\label{eq:Lem2.7ap3}
|J_{k\ell n}| &\les 2^{\hf j} 2^{-4k-2\ell-\frac34 n}\, A_{k\ell n}\, N(T)^3 t^{-\frac54 + \frac52 \delta}.
}
The only difference with \eqref{eq:Lem2.7ap2} is that we invoke the estimate
\EQ{\nn 
 \| v_n(t)\|_{L^\infty_x} &\les \|v_n(t)\|_{L^\infty_x}^\hf \big( 2^{-\frac32 n} \|\jD^2 v_n(t)\|_{L^2_x} \big)^{\frac12} \les 2^{-\frac34 n} N(T) t^{-\frac14 + \frac{\delta}{2}}
 }
in order to gain decay in $|\xi_3|\simeq 2^n$, if $|\xi_3|\ge1$. It remains to bound $A_{k\ell n}$ and to show that~\eqref{eq:Lem2.7ap3} can be summed over $0\leq k,\ell,n\le j+10$. By Lemma~\ref{lem:nablaPsi2 1}, for $i=1,2$, 
\EQ{\label{eq:mu bds}
|\mu_i(\xi_1,\xi_2)| &\les B_1^{-1}, \\
|\nabla \mu_i (\xi_1,\xi_2)| &\les B_1^{-2} B_2, \\
|\nabla^2\mu_i (\xi_1,\xi_2)| &\les B_1^{-3} B_2^2 + B_{1}^{-2} B_3, \\
|\nabla^3\mu_i (\xi_1,\xi_2)| &\les B_1^{-4}  B_2^3 +  B_{1}^{-3}B_2 B_3 + B_{1}^{-2} B_4,
}
with 
\[
 B_i := \bigl( \min_{1\le q\le 3} \, \jap{\xi_q} \bigr)^{-i-1}, \qquad 1 \le i \le 4.
\]
Therefore, for any multi-index $\beta$ of length $1\le |\beta|\le 3$, we conclude from~\eqref{eq:mu bds} that 
\EQ{ \label{eq:Dmui}
 |\partial^\beta \mu_i(\xi_1,\xi_2)| \les \bigl( \min_{1\le q\le 3} \, \jap{\xi_q} \bigr)^{2-|\beta|}.
}
To bound $\|\widehat{m}_p\|_{L^1_{x_1, x_2}}$ we infer from the pointwise estimate, for $1\le p\le 3$, 
\[
 |\widehat{m}_p(x_1,x_2)| \les \min \Bigl\{ \|m_p\|_{L^1_{\xi_1,\xi_2}}, (|x_1|+|x_2|)^{-3} \|\nabla^3 m_p\|_{L^1_{\xi_1,\xi_2}} \Bigr\} 
\]
that
\EQ{\label{eq:mhatm}
 \|\widehat{m}_p\|_{L^1_{x_1, x_2}} \les \|m_p\|_{L^1_{\xi_1,\xi_2}}^{\frac13} \|\nabla^3 m_p\|_{L^1_{\xi_1,\xi_2}}^{\frac23}.
}
From \eqref{eq:mpdef} and \eqref{eq:Dmui}, 
\EQ{ \nn
 \|m_p\|_{L^1_{\xi_1,\xi_2}} &\les B_1^{-1} 2^{k+\ell}, \\
 \|\nabla^3 m_p\|_{L^1_{\xi_1,\xi_2}} &\les \bigl( \min_{1\le q\le 3} \jap{\xi_q} \bigr)^{-1} \, 2^{k+\ell}.
}
The same bounds hold for $m_2$. 
Thus by \eqref{eq:mhatm}, 
\[
 \|\widehat{m}_1\|_{L^1_{x_1, x_2}} + \|\widehat{m}_2\|_{L^1_{x_1, x_2}} \les 2^{k+\ell}.
\]
In view of the preceding, $\|\widehat{m}_3\|_{L^1_{x_1, x_2}}$ satisfies a better bound, so that we have $A_{k\ell n}\les 2^{k+\ell}$. 
Finally, since $\max \, \{ 2^k, 2^\ell, 2^n \} \simeq 2^j$, 
\[
\sum_{k,\ell,n\le j+10}  2^{\hf j} 2^{-4k-2\ell-\frac34 n}\, A_{k\ell n} \les 1
\]
uniformly in $j$. 
\end{proof}

To summarize the results of Lemma~\ref{lem:hllhhh}, Lemma~\ref{lem:U*StatPhase} (with the choice $\alpha = \frac15$), and Lemma~\ref{lem:JU*compl}, we find that the leading order behavior of $\calT_2(t,\xi)$ is given by
\begin{equation} \label{equ:asymptotics_calT2}
 \calT_2(t,\xi) = \frac{4\pi}{t} \jxi^{-\frac12} |\hat{f}(t,\xi)|^2\hat{f}(t,\xi)  + \calO_{L^\infty_\xi} \Big( N(T)^3 t^{-\frac65 +3\delta} \Big).  
\end{equation}
In an analogous fashion one shows that
\begin{equation} \label{equ:asymptotics_calT3}
 \calT_3(t,\xi) = \frac{2\pi}{t} \jxi^{-\frac12} |\hat{f}(t,\xi)|^2\hat{f}(t,\xi)  + \calO_{L^\infty_\xi} \Big( N(T)^3 t^{-\frac65 +3\delta} \Big).  
\end{equation}
As far as the analysis of the remaining four oscillatory integrals $\calT_j(t,\xi)$ is concerned, the phases in~\eqref{eq:phases} satisfy
\[
 \phi_3(\xi_1,\xi_2,\xi_3)= -2\jxi -\phi_2(\xi_2,\xi_1,\xi_3),\quad \phi_4 = -2\jxi -\phi_1.
\]
Therefore, the geometric Lemma~\ref{lem:nablaPsi2 1} carries over to $\phi_3$ by interchanging $\xi_1$ with $\xi_2$.


\begin{figure}[ht]
\includegraphics[width=0.5\textwidth]{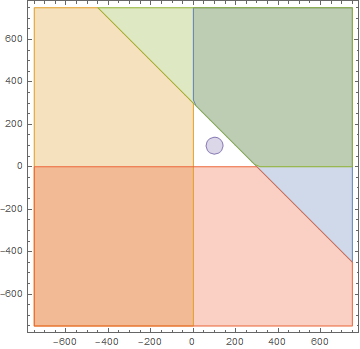}
\caption{Regions in the $(\xi_1, \xi_2)$ plane for the phases $\phi_1$ and $\phi_4$}
\label{fig:gebietePhi1}
\end{figure}


For the phase $\phi_1$, and thus also for the phase $\phi_4$, the shaded regions in Figure~\ref{fig:gebietePhi1} depict the area in which $|\nabla \Psi_1(\xi_1,\xi_2)| \simeq 1$. The oblique line corresponds to the equation $\xi_3 = \xi-\xi_1 - \xi_2 = 0$. The blank triangle contains a disk of size $c_\ast \xi$ centered at the critical point $(\xi/3,\xi/3)$ (assuming as before that $\xi\ge1$). We again denote this disk by~$U_\ast$. Inside this triangle, but outside of~$U_\ast$, similar  arguments as in Lemma~\ref{lem:nablaPsi2 1} lead to the conclusion that
\[
|\nabla\Psi_1(\xi_1,\xi_2)|\simeq \jap{\xi_1}^{-2} + \jap{\xi_2}^{-2} + \jap{\xi_3}^{-2}.
\]
In fact, the following analogue of Lemma~\ref{lem:nablaPsi2 1} holds.
\begin{lemma} \label{lem:nablaPsi1 1}
The neighborhood $U_*$ is characterized by the property 
$$|\nabla \Psi_1(\xi_1,\xi_2)|\ll|\xi|^{-2},$$ 
and in $U_\ast$ we have 
\EQ{\label{eq:near crit 2}
|\nabla\Psi_1(\xi/3+\eta_1,\xi/3+\eta_2)|\simeq |\xi|^{-3}(|\eta_1|+|\eta_2|).
}
Outside of the white triangle in Figure~\ref{fig:gebietePhi1},  we have $|\nabla \Psi_1|\simeq 1$, while inside of it,  but outside of $U_\ast$,  we have
\EQ{\label{eq:Psi1size}
|\nabla\Psi_1(\xi_1,\xi_2)|\simeq \jap{\xi_1}^{-2}+\jap{\xi_2}^{-2}+\jap{\xi_3}^{-2}.
}
\end{lemma}
\begin{proof}
Outside of the white triangle we have either $\xi_1\xi_3\le0$ or $\xi_2\xi_3\le0$. Thus,
\[
 \nabla \Psi_1(\xi_1,\xi_2) = \biggl( \frac{\xi_1}{\jap{\xi_1}} - \frac{\xi_3}{\jap{\xi_3}}, \frac{\xi_2}{\jap{\xi_2}} - \frac{\xi_3}{\jap{\xi_3}} \biggr)
\]
satisfies $|\nabla\Psi_1(\xi_1,\xi_2)|\simeq 1$ in that region. Inside of it we compute
\EQ{ \label{eq:NabPsi1}
 |\nabla\Psi_1(\xi_1,\xi_2)| \simeq |\xi_1-\xi_3||\Phi(\xi_3,\xi_3-\xi_1)|+ |\xi_2-\xi_3||\Phi(\xi_3,\xi_3-\xi_2)|,
}
where $\Phi$ satfisfies \eqref{eq:PhiTab}. We divide the white triangle into the following four regions
\EQ{\nn
I &:= \{ \xi_1>0,\xi_2>0,\xi_3>0, 2\xi_3>\xi_1, 2\xi_3>\xi_2\}, \\
II &:= \{ \xi_1>0,\xi_2>0,\xi_3>0, 2\xi_3<\xi_1, 2\xi_3<\xi_2\}, \\
III &:=  \{ \xi_1>0,\xi_2>0,\xi_3>0, 2\xi_3<\xi_1, 2\xi_3>\xi_2\}, \\
IV &:=  \{ \xi_1>0,\xi_2>0,\xi_3>0, 2\xi_3>\xi_1, 2\xi_3<\xi_2\}.
}
The region $I$ contains $U_*$ and by \eqref{eq:PhiTab}, in $I\setminus U_*$ we have $\xi_3\simeq\xi$ and
\EQ{ \nn
 |\nabla\Psi_1(\xi_1,\xi_2)| &\simeq |\xi_1-\xi_3|\jap{\xi_3}^{-1} \jap{\xi_1}^{-2} + |\xi_2-\xi_3|\jap{\xi_3}^{-1} \jap{\xi_2}^{-2} \\
 &\simeq  \jap{\xi_1}^{-2} +  \jap{\xi_2}^{-2}.
}
In the region $II$, we have 
\EQ{ \nn
 |\nabla\Psi_1(\xi_1,\xi_2)| &\simeq |\xi_1-\xi_3|\jap{\xi_3}^{-2} \jap{\xi_1-\xi_3}^{-1} + |\xi_2-\xi_3|\jap{\xi_3}^{-2} \jap{\xi_2-\xi_3}^{-1} \\
 &\simeq  \jap{\xi_3}^{-2}  \simeq  \jap{\xi_1}^{-2} +  \jap{\xi_2}^{-2} + \jap{\xi_3}^{-2}.
}
In the region $III$, $\xi_1\simeq\xi$ and
\EQ{ \nn
 |\nabla\Psi_1(\xi_1,\xi_2)| &\simeq |\xi_1-\xi_3|\jap{\xi_3}^{-2} \jap{\xi_1-\xi_3}^{-1} + |\xi_2-\xi_3|\jap{\xi_3}^{-1} \jap{\xi_2}^{-2} \\
 & \simeq   \jap{\xi_3}^{-2} + |\xi_2-\xi_3|\jap{\xi_3}^{-1} \jap{\xi_2}^{-2} \\
 & \simeq   \jap{\xi_3}^{-2} +  \jap{\xi_2}^{-2},
}
while the assertion for region $IV$ follows from region $III$ by symmetry. This shows in particular that $|\nabla \Psi_1|\gtrsim \jxi^{-2}$ outside of~$U_\ast$. Finally, in $U_\ast$, by~\eqref{eq:NabPsi1} and the second case in~\eqref{eq:PhiTab},
\[
 |\nabla\Psi_1(\xi_1,\xi_2)| \simeq (|\xi_1-\xi_3|+ |\xi_2-\xi_3|)\xi^{-3},
\]
which implies \eqref{eq:near crit 2}.
\end{proof}

To determine the leading order behaviors of the oscillatory integrals $\calT_j(t,\xi)$, $j = 1, 4, 5, 6$, we can carry out an analogous stationary phase analysis as above. We use Lemma~\ref{lem:nablaPsi2 1} for the phase~$\phi_3$ and Lemma~\ref{lem:nablaPsi1 1} for the phases $\phi_1$ and $\phi_4$. Then we find that
\begin{equation} \label{equ:asymptotics_calT1to6}
\begin{aligned}
 \calT_1(t,\xi) &= \frac{2\pi}{t} \frac{i}{\sqrt{3}} e^{it(-\jxi + 3 \jap{\frac{\xi}{3}})} \jxi^{\hf} \jap{{\textstyle \frac{\xi}{3}}}^{-1} \hat{f}\bigl(t, {\textstyle \frac{\xi}{3}}\bigr)^3  + \calO_{L^\infty_\xi} \Big( N(T)^3 t^{-\frac65 +3\delta} \Bigr), \\
 \calT_4(t,\xi) &= \frac{4\pi}{t} e^{-2 i t \jxi} \jxi^{-\hf} |\hat{f}(t, -\xi)|^2 \bar{\hat{f}}(t,-\xi) + \calO_{L^\infty_\xi} \Big( N(T)^3 t^{-\frac65 +3\delta} \Big), \\
 \calT_5(t,\xi) &= \frac{2\pi}{t} e^{-2 i t \jxi} \jxi^{-\hf} |\hat{f}(t, -\xi)|^2 \bar{\hat{f}}(t,-\xi) + \calO_{L^\infty_\xi} \Big( N(T)^3 t^{-\frac65 +3\delta} \Big), \\
 \calT_6(t,\xi) &= \frac{2\pi}{t} \frac{(-i)}{\sqrt{3}} e^{-it(\jxi + 3 \jap{\frac{\xi}{3}})} \jap{\xi}^{\hf} \jap{{\textstyle \frac{\xi}{3}}}^{-1}  \hat{\bar{f}}\bigl(t, {\textstyle \frac{\xi}{3}}\bigr)^3 + \calO_{L^\infty_\xi} \Big( N(T)^3 t^{-\frac65 +3\delta} \Big).
\end{aligned}
\end{equation}

Finally, it remains to determine the leading order behavior of the second term on the right-hand side of~\eqref{equ:ODE_profile_nonlocal_cubic_contribution_peeled_off}. We write it as 
\begin{equation*}
 \begin{aligned}
  - \frac{1}{6 \pi} \jxi^\hf e^{-it\jxi} \Bigl( \wh{\jD^{-2}w}(t) \ast \wh{\jD^{-2}w}(t) \ast \wh{\jD^{-2}w}(t) \Bigr)(\xi) = -\frac{1}{6 \pi} \sum_{j=7}^{10} \calT_j(t,\xi),
 \end{aligned}
\end{equation*}
where 
\begin{align*}
 \calT_7(t,\xi) &= \jxi^{\frac12} \int_{\Pi_\xi} e^{it\phi_1(\bxi)} \; \hat{h}(t,\xi_1)\hat{h}(t,\xi_2)\hat{h}(t,\xi_3) \, \ud \calH^2(d\bxi), \\
 \calT_8(t,\xi) &= 3 \jxi^{\frac12} \int_{\Pi_\xi} e^{it\phi_2(\bxi)}\; \hat{h}(t,\xi_1)\hat{\bar{h}}(t,\xi_2)\hat{h}(t,\xi_3) \, \ud \calH^2(d\bxi), \\
 \calT_9(t,\xi) &= 3 \jxi^{\frac12} \int_{\Pi_\xi} e^{it\phi_3(\bxi)}\; \hat{h}(t,\xi_1) \hat{\bar{h}}(t,\xi_2)\hat{\bar{h}}(t,\xi_3) \, \ud \calH^2(d\bxi), \\
 \calT_{10}(t,\xi) &= \jxi^{\frac12} \int_{\Pi_\xi} e^{it\phi_4(\bxi)} \; \hat{\bar{h}}(t,\xi_1) \hat{\bar{h}}(t,\xi_2)\hat{\bar{h}}(t,\xi_3) \, \ud \calH^2(d\bxi),
\end{align*}
and the phases $\phi_j(\bxi)$, $1 \leq j \leq 4$, are the same as in~\eqref{eq:phases}.
Then an analogous stationary phase analysis as above yields
\begin{equation} \label{equ:asymptotics_calT7to9}
\begin{aligned}
 \calT_7(t,\xi) &= \frac{2\pi}{t} \frac{i}{\sqrt{3}} e^{it(-\jxi + 3 \jap{\frac{\xi}{3}})} \jxi^{\hf} \jap{{\textstyle \frac{\xi}{3}}}^{-3} \hat{f}\bigl(t, {\textstyle \frac{\xi}{3}}\bigr)^3 + \calO_{L^\infty_\xi} \Big( N(T)^3 t^{-\frac65 +3\delta} \Big), \\
 \calT_8(t,\xi) &= \frac{6\pi}{t} \jxi^{-\frac52} | \hat{f}(t, \xi) |^2 \hat{f}(t,\xi) + \calO_{L^\infty_\xi} \Big( N(T)^3 t^{-\frac65 +3\delta} \Big), \\
 \calT_9(t,\xi) &= \frac{6\pi}{t} e^{-2 i t \jxi} \jxi^{-\frac52}  |\hat{f}(t, -\xi)|^2 \bar{\hat{f}}(t,-\xi) + \calO_{L^\infty_\xi} \Big( N(T)^3 t^{-\frac65 +3\delta} \Big), \\
 \calT_{10}(t,\xi) &= \frac{2\pi}{t} \frac{(-i)}{\sqrt{3}} e^{-it(\jxi + 3 \jap{\frac{\xi}{3}})} \jap{\xi}^{\hf} \jap{{\textstyle \frac{\xi}{3}}}^{-3}  \hat{\bar{f}}\bigl(t, {\textstyle \frac{\xi}{3}}\bigr)^3 + \calO_{L^\infty_\xi} \Big( N(T)^3 t^{-\frac65 +3\delta} \Big).
\end{aligned}
\end{equation}

In summary, by Lemma~\ref{lem:ODE_profile_nonlocal_cubic_contribution_peeled_off} and by \eqref{equ:asymptotics_calT2}, \eqref{equ:asymptotics_calT3}, \eqref{equ:asymptotics_calT1to6}, \eqref{equ:asymptotics_calT7to9}, we obtain uniformly for all $\xi \in \bbR$ and all $1 \leq t \leq T$ that 
\begin{equation*} 
 \begin{aligned}
  \frac{1}{2i} \jxi^\hf e^{-it\jxi} \calF\bigl[ \calC_{nl}(v+\bv)(t) \bigr](\xi) &= \frac{1}{8\pi i} \sum_{j=1}^6 \calT_j(t,\xi) - \frac{1}{12 \pi i} \sum_{j=7}^{10} \calT_j(t,\xi) + \calO_{L^\infty_\xi}\bigl( N(T)^3 t^{-\frac65 +3\delta} \bigr) \\
  &\quad = \frac{1}{t} \frac{1}{36 \sqrt{3}} e^{it(-\jxi + 3 \jap{\frac{\xi}{3}})} \jxi^{\hf} \jap{{\textstyle \frac{\xi}{3}}}^{-3} (3+\xi^2) \hat{f}\bigl(t, {\textstyle \frac{\xi}{3}}\bigr)^3 \\
  &\quad \quad + \frac{1}{4i t} \jxi^{-\frac52} (1+3\xi^2) |\hat{f}(t,\xi)|^2\hat{f}(t,\xi) \\
  &\quad \quad + \frac{1}{4i t}  e^{-2 i t \jxi} \jxi^{-\frac52} (1+3\xi^2) |\hat{f}(t, -\xi)|^2 \bar{\hat{f}}(t,-\xi) \\
  &\quad \quad - \frac{1}{t} \frac{1}{36 \sqrt{3}} e^{-it(\jxi + 3 \jap{\frac{\xi}{3}})} \jap{\xi}^{\hf} \jap{{\textstyle \frac{\xi}{3}}}^{-3} (3+\xi^2) \hat{\bar{f}}\bigl(t, {\textstyle \frac{\xi}{3}}\bigr)^3 \\
  &\quad \quad + \calO_{L^\infty_\xi}\bigl( N(T)^3 t^{-\frac65 +3\delta} \bigr).
 \end{aligned}
\end{equation*}
This completes the proof of Proposition~\ref{prop:ODE_profile_leading_order_contribution}.

\section{Proof of Theorem~\ref{thm:main}} \label{sec:proof_of_thm}

In this final section we combine the results from Sections~\ref{sec:setting_up}--\ref{sec:pointwise_estimates} to prove Theorem~\ref{thm:main}. 
By time-reversal symmetry it suffices to consider only positive times. 
Using a standard fixed-point argument, we construct a unique local-in-time solution $(u, \pt u) \in C([0,T]; H^3_x \times H^2_x)$ on some time interval $[0, T]$ to the Klein-Gordon equation~\eqref{equ:evol_equ_for_u} for the odd perturbation $u(t)$ of the sine-Gordon kink with (odd) initial data $(u, \pt u)|_{t=0} = (u_0, u_1)$. 
We can ensure that $T \geq 1$ for all sufficiently small initial data $\|(u_0, u_1)\|_{H^3_x \times H^2_x} \ll 1$ with an absolute implied constant.
In order to conclude global existence of $u(t)$, we seek to deduce an a priori estimate for the quantity 
\begin{equation*}
 M(T) := \sup_{0 \leq t \leq T} \, \Bigl\{ \jt^{-\delta} \|( u(t), \pt u(t) )\|_{H^3_x \times H^2_x} + \jt^\hf \|u(t)\|_{L^\infty_x} \Bigr\},
\end{equation*}
where $0 < \delta \ll 1$ is a small absolute constant. To this end we pass to the new dependent variable 
\begin{equation*}
 v(t) := w(t) - i \jD^{-1} \pt w(t), \quad w(t) := (\calD^\ast u)(t),
\end{equation*}
introduced in~\eqref{equ:def_new_variable_w} and~\eqref{equ:def_new_variable_v}, which is a solution to the transformed first-order Klein-Gordon equation~\eqref{equ:nlkg_for_v} on the time interval $[0,T]$ with initial datum 
\begin{equation*}
 v(0) = v_0 = (\calD^\ast u_0) - i \jD^{-1} (\calD^\ast u_1). 
\end{equation*}
Observe that $v(t,x)$ is even since $u(t,x)$ is odd, and that
\begin{equation*}
 \|\jx v_0\|_{H^2_x} \lesssim \|\jx (u_0, u_1)\|_{H^3_x \times H^2_x} =: \varepsilon.
\end{equation*}
We now consider the bootstrap quantity
\begin{equation*} 
 \begin{aligned}
  N(T) &:= \sup_{0 \leq t \leq T} \, \biggl\{ \jt^{\frac{1}{2}} \|v(t)\|_{L^\infty_x} + \jt^{-\delta} \| \jD^2 v(t) \|_{L^2_x} + \jt^{-\delta} \| \jD L v(t) \|_{L^2_x} \\
  &\qquad \qquad \qquad \qquad \qquad \qquad \qquad \qquad + \jt^{-1-\delta} \|x v(t)\|_{L^2_x} + \bigl\| \jap{\xi}^{\frac{3}{2}} \hat{f}(t,\xi)  \bigr\|_{L^\infty_\xi} \biggr\}.
 \end{aligned}
\end{equation*}
For sufficiently small data we can propagate bounds on all norms of the solution $v(t)$ in $N(T)$ for short times. Thus, we may assume that $N(1) \lesssim \varepsilon$ with an absolute implied constant. 

By the main energy estimates from Proposition~\ref{prop:growth_H2v}, Proposition~\ref{prop:growth_xv}, and Proposition~\ref{prop:growth_H1Lv} we have
\begin{equation*}
 \begin{aligned}
  &\sup_{0 \leq t \leq T} \, \Bigl\{ \jt^{-\delta} \|\jD^2 v(t)\|_{L^2_x} + \jt^{-\delta} \| \jD L v(t) \|_{L^2_x} + \jt^{-1-\delta} \|x v(t)\|_{L^2_x}  \Bigr\} \\
  &\lesssim \|\jx v_0\|_{H^2_x} + \|v_0\|_{H^1_x}^2 + N(T)^2 \\
  &\lesssim \varepsilon + N(T)^2,
 \end{aligned}
\end{equation*}
and by Proposition~\ref{prop:Linftyxi_bound_profile} we have 
\begin{equation*}
 \begin{aligned}
  \sup_{1 \leq t \leq T} \, \bigl\| \jap{\xi}^{\frac{3}{2}} \hat{f}(t,\xi) \bigr\|_{L^\infty_\xi} \lesssim \bigl\| \jap{\xi}^{\frac{3}{2}} \hat{f}(1,\xi) \bigr\|_{L^\infty_\xi} + N(T)^2 \lesssim \varepsilon + N(T)^2.
 \end{aligned}
\end{equation*}
Then the asymptotics for the linear Klein-Gordon evolution from Lemma~\ref{lem:asymptotics_KG} further imply
\begin{equation*}
 \begin{aligned}
  \sup_{1 \leq t \leq T} t^\hf \|v(t)\|_{L^\infty_x} &\lesssim \sup_{1 \leq t \leq T} \bigl\| \jap{\xi}^{\frac{3}{2}} \hat{f}(t,\xi) \bigr\|_{L^\infty_\xi} + \sup_{1 \leq t \leq T} t^{-\frac16} \bigl( \|\jD L v(t)\|_{L^2_x} + \|\jD^2 v(t)\|_{L^2_x} \bigr) \\
  &\lesssim \varepsilon + \sup_{1 \leq t \leq T} \, t^{-\frac16} \jt^\delta \bigl( \varepsilon + N(T)^2 \bigr) \\
  &\lesssim \varepsilon + N(T)^2.
 \end{aligned}
\end{equation*}
Combining the preceding estimates yields
\begin{equation*}
 N(T) \lesssim \varepsilon + N(T)^2.
\end{equation*}
By a standard continuity argument we can now infer that there exists a small absolute constant $0 < \varepsilon_0 \ll 1$ such that if $\varepsilon \leq \varepsilon_0$, then $N(T) \lesssim \varepsilon$ independently of $T$. In view of the identity 
\begin{equation} \label{equ:proof_of_thm_u_v_relation}
 u(t) = \calI[v(t)] + \calI[\bv(t)]
\end{equation}
with the integral operator $\calI[v(t)]$ defined in~\eqref{equ:def_calI}, using Lemma~\ref{lem:aux_bounds_Ioperators} it is easy to conclude the a priori bound
\begin{equation*}
 M(T) \lesssim N(T) \lesssim \varepsilon.
\end{equation*}
This proves global existence of $u(t)$ and establishes the decay estimate~\eqref{equ:main_thm_decay_perturbation} for the perturbation of the sine-Gordon kink asserted in Theorem~\ref{thm:main}. Moreover, we obtain that 
\begin{equation} \label{equ:proof_of_thm_final_apriori_bounds}
 \begin{aligned}
  &\sup_{t \geq 0} \, \Bigl\{ \jt^{-\delta} \|( u(t), \pt u(t) )\|_{H^3_x \times H^2_x} + \jt^\hf \|u(t)\|_{L^\infty_x} \Bigr\} \\
  &\lesssim \sup_{t \geq 0} \, \biggl\{ \jt^{\frac{1}{2}} \|v(t)\|_{L^\infty_x} + \jt^{-\delta} \| \jD^2 v(t) \|_{L^2_x} + \jt^{-\delta} \| \jD L v(t) \|_{L^2_x} \\
  &\qquad \qquad \qquad \qquad \qquad \qquad + \jt^{-1-\delta} \|x v(t)\|_{L^2_x} + \bigl\| \jap{\xi}^{\frac{3}{2}} \hat{f}(t,\xi)  \bigr\|_{L^\infty_\xi} \biggr\} \\
  &\lesssim \varepsilon.
 \end{aligned}
\end{equation}

It remains to infer asymptotics for $u(t)$.
From~\eqref{equ:Linftyxi_bound_difference_profile}, \eqref{equ:proof_of_thm_final_apriori_bounds}, and the fact that $v(t,x)$ is even, we conclude that there exists an even profile $\widehat{V} \in L^\infty_\xi$ such that
\begin{equation} \label{equ:proof_of_thm_decay_diff_asympt_profile}
 \bigl\| \jxi^\thf \hatf(t,\xi) e^{i\Phi(t,\xi)} - \widehat{V}(\xi) \bigr\|_{L^\infty_\xi} \lesssim \varepsilon^2 t^{-\frac15 +3\delta}, \quad t \geq 1.
\end{equation}
We then multiply the differential equation~\eqref{equ:ode_profile} for the Fourier transform of the profile by the integrating factor $e^{i\Psi(t,\xi)}$ with
\begin{equation*}
 \Psi(t,\xi) := \frac14 \jxi^{-7} (1+3\xi^2) \int_1^t \frac{1}{s} \bigl| \widehat{V}(\xi) \bigr|^2 \, \ud s = \frac14 \jxi^{-7} (1+3\xi^2) \bigl| \widehat{V}(\xi) \bigr|^2 \log(t).
\end{equation*}
Repeating the arguments in the proof of Proposition~\ref{prop:Linftyxi_bound_profile} and exploiting~\eqref{equ:proof_of_thm_decay_diff_asympt_profile}, we obtain that there exists an even asymptotic profile $\widehat{W} \in L^\infty_\xi$ with $|\widehat{W}(\xi)| = |\widehat{V}(\xi)|$ such that  
\begin{equation} 
 \bigl\| \jxi^\thf \hatf(t,\xi) e^{i\Psi(t,\xi)} - \widehat{W}(\xi) \bigr\|_{L^\infty_\xi} \lesssim \varepsilon^2 t^{-\frac15 +3\delta}, \quad t \geq 1.
\end{equation}
The asymptotics for the linear Klein-Gordon evolution from Lemma~\ref{lem:asymptotics_KG} then give that 
\begin{equation*}
 v(t,x) = \frac{1}{t^\hf} e^{i\frac{\pi}{4}} e^{i\rho} e^{-i\Psi(t, \xi_0)} \jap{\xi_0}^\thf \widehat{W}(\xi_0) \one_{(-1,1)}({\textstyle \frac{x}{t}}) + \calO_{L^\infty_x} \bigl( t^{-\frac23+\delta} \varepsilon \bigr), \quad t \geq 1, 
\end{equation*}
with $\rho = \sqrt{t^2-x^2}$ and $\xi_0 = -\frac{x}{\rho}$. This implies via~\eqref{equ:proof_of_thm_u_v_relation} the asserted asymptotics~\eqref{equ:main_thm_asymptotics} for $u(t)$ and finishes the proof of Theorem~\ref{thm:main}.

\bibliographystyle{amsplain}
\bibliography{references}

\end{document}